\documentclass[reqno,english,11pt]{amsart}
\usepackage{times}

\usepackage{amsmath,amsfonts,amssymb,graphicx,amsthm,enumerate,url,mathabx}
\usepackage[noadjust]{cite}
\usepackage{stmaryrd}
\usepackage{mathrsfs,booktabs}
\usepackage{xifthen,xcolor,tikz,setspace}
\usetikzlibrary{decorations.pathmorphing,patterns,shapes,calc,decorations}
\usetikzlibrary{decorations.pathreplacing}
\usepackage{mathtools,upgreek}
\usepackage[final]{showkeys} 

\definecolor{refkey}{gray}{.75}
\definecolor{labelkey}{gray}{.5}

\usepackage[colorlinks=true]{hyperref}
\usepackage{verbatim}

\usepackage{float}

\usepackage[letterpaper]{geometry}
\makeatletter
\numberwithin{equation}{section}
\numberwithin{figure}{section}

\newtheorem{theorem}{Theorem}[section]
\newtheorem*{theorem*}{Theorem}
\newtheorem{lemma}[theorem]{Lemma}

\newtheorem{claim}[theorem]{Claim}
\newtheorem{proposition}[theorem]{Proposition}

\newtheorem{observation}[theorem]{Observation}
\newtheorem{fact}[theorem]{Fact}
\newtheorem{corollary}[theorem]{Corollary}
\newtheorem{cor}[theorem]{Corollary}

\theoremstyle{definition}{
\newtheorem{example}[theorem]{Example}
\newtheorem{definition}[theorem]{Definition}
\newtheorem{process}[theorem]{Process}

\newtheorem*{definition*}{Definition}

\newtheorem{remark}[theorem]{Remark}

}

\newcommand{\R}{\mathbb R}

\newcommand{\cA}{\ensuremath{\mathcal A}}

\newcommand{\cC}{\ensuremath{\mathcal C}}
\newcommand{\cD}{\ensuremath{\mathcal D}}
\newcommand{\cE}{\ensuremath{\mathcal E}}
\newcommand{\cF}{\ensuremath{\mathcal F}}
\newcommand{\cG}{\ensuremath{\mathcal G}}
\newcommand{\cH}{\ensuremath{\mathcal H}}

\newcommand{\cT}{\ensuremath{\mathcal T}}

\newcommand{\cV}{\ensuremath{\mathcal V}}

\newcommand{\fm}{\mathfrak{m}}

\newcommand{\fM}{\mathfrak{M}}
\newcommand{\fPM}{\overline{\fM}}

\newcommand{\sT}{{\ensuremath{\mathscr T}}}

\newcommand{\bd}{{\ensuremath{\mathbf d}}}

\newcommand{\ps}{{\hat{p}}}
\newcommand{\bbP}{{\mathbb{P}}}

\newcommand{\T}{{\mathcal{T}}}

\newcommand{\E}{{\mathbb{E}}}

\newcommand{\treelike}{{\mathsf{Treelike}}}
\newcommand{\sparse}{{\mathsf{Sparse}}}
\newcommand{\w}{{\mathrm{wgt}}}

\newcommand{\vdn}{{\ensuremath{\mathbf{d}_n}}}
\newcommand{\Pvd}{{\ensuremath{\mathbb P_{\vdn}}}}
\newcommand{\Evd}{{\ensuremath{\mathbb E_{\vdn}}}}

\newcommand{\Prg}{{\mathbb{P}}_{\textsc{rg}(\vdn)}}
\newcommand{\Pcm}{{\mathbb{P}}_{\textsc{cm}(\vdn)}}

\newcommand{\Ecm}{{\mathbb{E}}_{\textsc{cm}(\vdn)}}

\newcommand{\Tburn}{T_{\textsc{burn}}}

\renewcommand{\Pr}{{\mathbb{P}}}

\renewcommand{\epsilon}{{\varepsilon}}

\DeclareMathOperator{\ber}{\mathrm{Ber}}
\DeclareMathOperator{\bin}{\mathrm{Bin}}

\newcommand{\tv}{{\textsc{tv}}}
\newcommand{\tmix}{{t_{\textsc{mix}}}}

\newcommand{\ew}{{\mathbf{m}}}
\usepackage[normalem]{ulem}

\makeatother

 \title[Sampling from Potts on random graphs of unbounded degree]
 {Sampling from Potts on random graphs of unbounded degree \\ via random-cluster dynamics}

\author{Antonio Blanca}
\address{A.\ Blanca\hfill\break
Department of CSE, Pennsylvania State University }
\email{ablanca@cse.psu.edu}
\author{Reza Gheissari}
\address{R.\ Gheissari\hfill\break
Department of Statistics and EECS \\ UC Berkeley }
\email{gheissari@berkeley.edu}

\begin{document}

\maketitle

\thispagestyle{empty}

\vspace{-.75cm}
\begin{abstract}
    We consider the problem of sampling from the ferromagnetic Potts and random-cluster models on a general family of random graphs via the Glauber dynamics for the random-cluster model. 
     The random-cluster model is parametrized by an edge probability $p \in (0,1)$ and a cluster weight $q > 0$.
    We establish that for every $q\ge 1$, the random-cluster Glauber dynamics mixes in optimal $\Theta(n\log n)$ steps on $n$-vertex random graphs having a prescribed degree sequence with bounded average branching $\gamma$ throughout the {full high-temperature} uniqueness regime $p<p_u(q,\gamma)$. 

    The family of random graph models we consider includes the Erd\H{o}s--R\'enyi random graph $G(n,\gamma/n)$, and so we provide the first polynomial-time sampling algorithm for the ferromagnetic Potts model {on Erd\H{o}s--R\'enyi random graphs
    for the full tree uniqueness regime}.
    We accompany our results with mixing time lower bounds (exponential in the largest degree) for the Potts Glauber dynamics, in the same settings where our $\Theta(n \log n)$ bounds for the random-cluster Glauber dynamics apply. This reveals a novel and significant computational advantage of random-cluster based algorithms for sampling from the  Potts model at high temperatures.
\end{abstract}

\thispagestyle{empty}
\section{Introduction}

The ferromagnetic Potts model is a classical spin system model in statistical physics and computer science. It is defined on a finite graph $G = (V,E)$, by a set of spins (or colors) $[q] = \{1,...,q\}$ and an edge weight or inverse temperature parameter $\beta > 0$. 
A configuration $\sigma \in \{1,\dots,q\}^V$ of the model is an assignment of spins to the vertices of $V$. The probability of $\sigma$ is given by the Gibbs distribution:
\begin{equation}
\label{eq:gibbs:potts}
\mu_{G,\beta,q}(\sigma) = \frac{1}{Z_{G,\beta,q}} \exp(-\beta D(\sigma))\,, 
\end{equation}
where $D(\sigma) = |\{\{v, w\} \in E : \sigma(v) \neq \sigma(w)\}|$ is the number of edges whose endpoints have different spins in $\sigma$, and $Z_{G,\beta,q}$ is a normalizing factor known as the partition function.
The Ising model of ferromagnetism corresponds to the case where $q = 2$.

Sampling from the Potts Gibbs distribution~\eqref{eq:gibbs:potts} is one of the most frequently encountered problems when running simulations in statistical physics or when solving a variety of inference tasks in computer science; see e.g.~\cite{Georgii,GG,Roth,OsinderoHinton,Felsenstein,Ellison,MontSab} and the references therein for a sample of these applications. 
There is a family of powerful sampling algorithms for the Potts model that are based on its \emph{random-cluster representation}, defined subsequently. 
Such algorithms, which include the Glauber dynamics of the random-cluster model and the widely-used Swendsen--Wang dynamics, 
are an attractive option computationally
since they are often efficient 
at ``low-temperatures'' (large $\beta$), a
parameter regime
where standard Markov chains for the Potts model (including the canonical Glauber dynamics)
often converge exponentially slowly; {see, e.g.,~\cite{BCFKTVV,BCT,CDLLPS,BGP}.} 

To be more precise, the \emph{random-cluster model} on a finite graph~$G=(V,E)$, is defined by an edge probability parameter
$p\in(0,1)$ and a cluster weight $q>0$.
The set of {configurations} of the model is the set of all subsets of edges $\omega \subseteq E$. The probability of each configuration $\omega$ is given by the Gibbs distribution: 
\begin{equation}\label{eq:rcmeasure}
\pi_{G,p,q}(\omega) = \frac{1}{Z_{G,p,q}} p^{|\omega|}(1-p)^{|E|-|\omega|} q^{c(\omega)},
\end{equation}
where $c(\omega)$ is the number of connected components (also called clusters) in the subgraph~$(V,\omega)$, and $Z_{G,p,q}$ is the corresponding {partition function}. The random-cluster model was introduced by Fortuin and Kasteleyn \cite{FK} as a unifying framework for studying random graphs, spin systems, and electrical networks, and it is also known as the \emph{FK-representation} of the Ising and Potts model.

For integer $q \ge 2$, a sample $\omega \subseteq E$ from the random-cluster Gibbs distribution $\pi_{G,p,q}$ can be easily transformed into one for the ferromagnetic $q$-state Potts model with inverse temperature $\beta(p) = -\ln ({1-p})$, by independently assigning a random spin from $\{1,\dots,q\}$ to (all vertices in) each connected component of $(V,\omega)$; {see, e.g,~\cite{FK,ES,Grimmett}}. As such, any sampling algorithm for the random-cluster model yields 
one for the ferromagnetic Potts model with essentially no computational overhead. This has led to significantly improved sampling algorithms for the Potts model in various low-temperature settings~\cite{SW,GSV,LNNP14,PerkinsPottsAllTemp,Ullrich-random-cluster,HeJePe20} 
and more generally, to a broad interest in dynamics for the random-cluster model~\cite{CM,GuoJer,BS,BS-MF,BlGh21,BSX-CM-MF,BGVfull}.

In this paper, we focus on the Glauber dynamics of the random-cluster model, which for easy distinction we will henceforth call the \emph{FK-dynamics}.
From a configuration $\omega_t \subseteq E$, one step of this Markov chain transitions to
a new configuration $\omega_{t+1}\subseteq E$ as follows: 
\begin{enumerate}
	\item Choose an edge $e_t\in E$ uniformly at random;
	\item Set $\omega_{t+1} = \omega_t \cup \{e_t\}$ with probability
	$\left\{\begin{array}{ll}
	\ps := \frac{p}{q(1-p)+p} & \mbox{if $e_t$ is a ``cut-edge'' in $(V,\omega_t)$;} \\
	p & \mbox{otherwise;}
	\end{array}\right.$
	\item Otherwise set $\omega_{t+1} = \omega_t \setminus \{e_t\}$.
\end{enumerate}
Here, we say $e$ is a {\it cut-edge} in $(V,\omega_t)$ 
if changing the state of $e_t$ changes the number of connected components $c(\omega_t)$ in $(V,\omega_t)$. The probabilities in step (2) are exactly the conditional probabilities of $e_t$ being in the configuration $\omega_t$ given the remainder of $\omega_t$. As such, this Markov chain is reversible with respect to $\pi_{G,p,q}$ and converges to it. 
We are interested in its \emph{mixing time} $\tmix$; i.e., the number of steps until the dynamics is within variation distance $1/4$ of $\pi_{G,p,q}$, starting from the worst possible initial configuration.

As mentioned, the FK-dynamics is by now well-studied in its own right, though sharp analyses of its mixing time are only available on certain structured graphs like the complete graph~\cite{BS-MF,GLP,BSX-CM-MF}, boxes in the infinite integer lattice graph $\mathbb Z^d$~\cite{BS,BGVfull,GL1,GL2,HarelSpinka,GS,BCT}, and trees~\cite{BZSV-SW-trees}. Recently, in~\cite{BlGh21}, the authors studied the FK-dynamics on random regular graphs and established optimal $\Theta(n\log n)$ mixing time for the FK-dynamics throughout the entire {high-temperature} tree uniqueness regime.

 Our aim in this paper is to study the FK-dynamics in settings in which the maximum degree of the underlying graph is much larger than its \emph{average} degree. In such settings, high-degree vertices are an obstruction to the fast convergence of the Ising/Potts Glauber dynamics. 
For instance, we later prove (see Section~\ref{subsec:slow}) that on a general class of random graphs on $n$ vertices with maximum degree $d_{\textsc{max}}$, the Ising/Potts Glauber dynamics
requires $n\cdot \exp(\Omega(d_{\textsc{max}}))$ steps to converge at high temperatures. 

We reveal here that, for the same general family of random graphs, random-cluster based algorithms are not affected by the presence of high-degree vertices; both their mixing times and fast mixing parameter regimes are determined instead by the \emph{average degree} of the graph. This reveals a novel and significant computational
advantage of random-cluster based algorithms for sampling from the ferromagnetic Potts model \emph{at high temperatures}.
Indeed, prior to this work, random-cluster based sampling algorithms were only found to be more efficient than Ising/Potts Glauber dynamics at  low temperatures.

More precisely, we study the mixing time of the FK-dynamics on random graphs of average branching $\gamma > 0$ in the full uniqueness (high-temperature) regime $p < p_u(q,\gamma)$.
At integer $\gamma$, the threshold $p_u(q,\gamma)$, formally defined in~\eqref{eq:p_u}, was identified in~\cite{Haggstrom} as a uniqueness/non-uniqueness phase transition point of the random-cluster model on the \emph{wired} $\gamma$-ary tree, i.e., where the leaves are externally wired to be in the same connected component. {For us, $p_u(q,\gamma)$ is the natural extension of that function to non-integer $\gamma$,
which we show corresponds to the high-temperature uniqueness threshold of the random-cluster model on general trees of average branching $\gamma$ for all $q\ge 1$ (see Corollary \ref{cor:rc:unique} in Section~\ref{sec:uniqueness})}.

Before we describe our general results for random graph models with fixed degree sequence (which we define in the next subsection) {we present a special case of our main result of particular interest concerning the FK-dynamics on sparse Erd\H{o}s--R\'enyi random graphs.}

\begin{theorem}
	\label{thm:intro:fk-er}
	Fix $q\ge 1$, $\gamma>0$ and $p<p_u(q,\gamma)$. If $\cG$ is an Erd\H{o}s--R\'enyi random graph $\cG\sim G(n,\gamma/n)$, then with probability $1-o(1)$, $\cG$ is such that the FK-dynamics on $\cG$ satisfies $\tmix = \Theta(n\log n)$. 
\end{theorem}

This yields a sampler for the Potts distribution on Erd\H{o}s--R\'enyi random graphs with near-optimal running time. Let $\beta_u(q,\gamma) = -\ln(1-p_u(q,\gamma))$ be the corresponding uniqueness point for the Potts model.

\begin{corollary}
	\label{cor:intro:potts}
	Fix $q\ge 2$, $\gamma>0$ and $\beta<\beta_u(q,\gamma)$. There is an MCMC sampling algorithm that, with probability $1-o(1)$ over the choice of an Erd\H{o}s--R\'enyi random graph $\cG\sim G(n,\gamma/n)$, outputs a configuration whose distribution is within total-variation distance $\delta>0$ of $\mu_{\cG,\beta,q}$ in time $O(n(\log n)^3 \log(1/\delta))$. 
\end{corollary}

Corollary~\ref{cor:intro:potts} is a direct consequence of Theorem~\ref{thm:intro:fk-er} and the aforementioned {connection between} the random-cluster model and the Potts model.
The extra $O((\log n)^2)$ factor in the running time of the algorithm comes 
from the (amortized) cost of checking
whether the chosen edge is a cut-edge in each step of the FK-dynamics (see, e.g.,~\cite{HLT,Thorup}).

To the best of our knowledge, this is the first polynomial-time sampling algorithm for the Potts model on Erd\H{o}s--R\'enyi random graphs for $q \ge 3$ and $\beta = \Omega(1)$.
{
Even for the better understood $q=2$ case (i.e., the Ising model), Corollary~\ref{cor:intro:potts}  provides the fastest known sampling algorithm, improving upon the running time 
of samplers based on the Glauber dynamics which, for the Ising model, is known to converge in $n^{1+\Theta(\frac{1 }{\log \log n})}$ steps for all $\beta<\beta_u(2,\gamma)$~\cite{MS}.}

We mention that the thresholds $p_u(q,\gamma)$ and $\beta_u(q,\gamma)$
should be sharp, in the sense that the FK-dynamics is conjectured to undergo polynomial or exponential slowdowns (depending on $q$) at the point $p_u(q,\gamma)$ (and when $q>2$ in a whole critical window $(p_u,p_u')$). This is 
by analogy with the FK-dynamics on the complete graph~\cite{GLP} and on random regular graphs~\cite{Coja-Oghlan}; see also~\cite{GSVY,HeJePe20,DMSS}.

\subsection{Results on random graphs with general degree sequences}
We next provide our main results on random graph models 
with a fixed degree sequence. 
Let $\vdn = (d_1,...,d_n)$ be the degree sequence giving the degree of each vertex $v\in \{1,...,n\}$. 
Our results will hold for uniform random graphs with degree sequence $\vdn$ under certain mild conditions on this degree sequence. 
The first condition we make on $\vdn$ is that the sequence is \emph{graphical}: i.e., that there exists at least one simple graph having degree sequence $\vdn$. 

Given a graphical sequence $\vdn$, we define $\Prg$ as the uniform distribution over all simple graphs on $n$ vertices having degree sequence $\vdn$. The governing quantity in this degree sequence, in terms of the uniqueness thresholds for the Potts and random-cluster models on $\cG \sim \Prg$, will be what we call the \emph{effective offspring distribution}
$\Pvd$, which is defined as the distribution over the set $\mathcal M(\vdn) = \{d_v-1: v \in \{1,...,n\}\}$ where $x \in \mathcal M(\vdn)$ is assigned probability: 
\begin{align}\label{eq:Pvd}
\mathbb P_{\vdn}(x) = \frac{\sum_{v} (x+1) \mathbf 1_{\{d_v= x+1\}}}{\sum_{v} d_v}\,.
\end{align}
In words, the distribution $\Pvd$ corresponds to choosing $d_v-1$ with probability proportional to
the total degree of vertices having degree
$d_v$. This distribution governs the offspring distribution corresponding to the random trees one obtains when looking at balls of small radius around a vertex of a random graph $\cG \sim \Prg$. Specifically, a vertex of degree $d$ is selected to be the next vertex added to the random tree with probability proportional to the total degree of all such vertices, and
once it is selected and connected to its parent, it has $d-1$ available edges to connect to other randomly chosen vertices.

Our results will apply to graphical degree sequences whose effective offspring distribution has a certain mean, and has bounded finite moments, as we detail next.

\begin{definition}\label{def:D-gamma-kappa}
	Let $\cD_{\gamma,\kappa}$ be the set of graphical degree sequences $(\vdn)_n$ such that $D\sim \Pvd$ has mean that is uniformly bounded away from $\gamma$ and uniformly bounded $\kappa$-th moment. Formally,
	\[
	\limsup_n \Evd[D]<\gamma \qquad \mbox{and}\qquad \limsup_n \Evd[D^\kappa]<\infty\,.
	\]
	Let us finally assume that $\sum_{1\le v\le n} d_v = \Omega(n)$; this is not strictly necessary, but will simplify presentation.
\end{definition}

This framework is fairly standard in the random graphs literature~\cite{BollobasBook} and is similar to e.g., the setting of~\cite{GGS21} for studying sampling from Potts on random graphs with fixed degree sequences at sufficiently low temperatures.
While Definition~\ref{def:D-gamma-kappa} yields a fairly general family of random graphs, we draw attention to some well-studied examples which fall under its umbrella. 

\begin{example}\label{eq:Delta-regular}
	\emph{$\Delta$-regular random graph.} 
	In this case, $\vdn = (\Delta,\ldots,\Delta)$ and the effective offspring distribution simply assigns probability $1$ to $\Delta-1$; thus $(\vdn)_n \in \cD_{\gamma,\kappa}$ for every $\gamma>\Delta-1$ and every $\kappa$. 
 \end{example}

\begin{example}\label{ex:Erdos-Renyi}
	\emph{Erd\H{o}s--R\'enyi random graph $G(n,\lambda/n)$}. It was shown in~\cite{KimPoissonCloning} that if $\vdn$ is drawn as an i.i.d.\ sequence of Poisson random variables of mean $\lambda$, then $\Prg$ is contiguous with respect to $G(n,\lambda/n)$.
	(Two random graph models are \emph{contiguous} when any sequence of events that has a probability of $1-o(1)$ in one has a probability of $1 - o(1)$ in the other model as well.)
	Hence, it suffices to prove the desired results with high probability over such $\vdn$ (see Lemma~\ref{lem:poisson-cloning-contiguous}).  
	Standard concentration estimates for Poisson random variables (see Lemma~\ref{lem:Poisson-D-gamma-kappa}) then give that for every $\gamma>\lambda$ and every $\kappa$, with high probability, $(\vdn)_n\in \cD_{\gamma,\kappa}$. 
\end{example}

Our main result is an optimal mixing time bound for the FK-dynamics on $\cG \sim \Prg$, which applies to all the examples above and more generally to random graphs with degree sequences in $\cD_{\gamma,\kappa}$.

\begin{theorem}
	\label{thm:intro:general}
	Fix $q\ge 1$, $\gamma>0$, and $p<p_u(q,\gamma)$. There exists $\kappa$ such that if $(\vdn)_n\in \cD_{\gamma,\kappa}$, then with probability $1-o(1)$, the FK-dynamics on $\cG\sim \Prg$ satisfies $\tmix = \Theta(n\log n)$. 
\end{theorem}

This parameter regime in Theorem~\ref{thm:intro:general} is tight as FK-dynamics have been very recently shown~\cite{Coja-Oghlan} to exponentially slow down as soon as $p>p_u(q,\gamma)$ for random regular graphs (Example~\ref{eq:Delta-regular}) at integer $q > 2$.

The proof of the upper bound in Theorem~\ref{thm:intro:general} is the main content of this paper. As mentioned, the special case of the $\Delta$-regular random graph (i.e., $\vdn = (\Delta,...,\Delta)$) was the content of an earlier paper~\cite{BlGh21}. However, as soon as the degree sequence is not homogeneous, substantial further obstacles arise.

First, even the uniqueness threshold for the random-cluster model 
on wired heterogeneous trees (specifically, with offspring distribution $\Pvd$) had not been established. 
In our proof of Theorem~\ref{thm:intro:general} we require something much stronger; namely,
an exponential decay of connectivities with the correct rate (see Lemma~\ref{lemma:exp:decay:wired:treelike}).
In the regular case, the fact that $p_u(q,\gamma)$ is the uniqueness threshold goes back to the work of H{\"a}ggstr{\"o}m~\cite{Haggstrom} (see also~\cite{Jonasson,BGGSVY}), and the exponential decay rate was established in~\cite{BlGh21}. 
To establish analogous results for the heterogeneous case, we combine the approach of~\cite{Lyons} (which considered the special case of the Ising model $q=2$) with ideas from~\cite{BGGSVY}, so as to recurse, not on the marginal of an edge of the tree, but rather on a nice functional of its probability of downwards connection to infinity.

The second technical obstacle concerns establishing that the FK-dynamics on $\cG\sim \Prg$ \emph{shatters}, i.e., that its components have size at most $O(n^\varepsilon)$ after $O(n)$ steps of the dynamics.
This is proved using a delicate revealing procedure for the random graph with the FK-dynamics configuration on top of it, a technique introduced in~\cite{BlGh21} for the case of random regular graphs. The heterogeneity of the degrees
in the current setting, however, 
introduces extra correlations between the underlying graph and the FK-dynamics configuration, necessitating substantial modifications to the revealing procedure from~\cite{BlGh21}.

The changes we make to deal with the above-described dependencies include: (i) modifications to the revealing process so that it is based on half-edges rather than vertices and the dynamics is run in continuous time, and (ii) a new criteria to truncate potentially unbounded increments in the revealing procedure.
The more robust procedure yields a notable further improvement: we show that the shattering time is $O(n)$ (as opposed to $O(n \log n)$ in~\cite{BlGh21}). Though this 
improvement has no impact on the eventual mixing time bound, the more precise understanding of the shattering phase may be useful in other settings.

A more detailed proof sketch of this theorem and the new complications that arise is provided in Section~\ref{sec:proof-sketch} and Remark~\ref{rem:revealing-process-modifications}.
    
\subsection{Slowdown for the corresponding Potts Glauber dynamics}\label{subsec:slow}
Returning to the advantage of FK-dynamics in the presence of high-degree vertices, the following theorem establishes that in the same setting as Theorem~\ref{thm:intro:general} the Ising/Potts Glauber dynamics slows down exponentially in the maximum degree.

\begin{theorem}\label{thm:Ising-Potts-lower-bound}
	Fix $q\ge 1$, $\gamma >0$ and $\beta<\beta_u(q,\gamma)$. Then there exists $\kappa$ such that if $(\vdn)_n \in \cD_{\gamma,\kappa}$, then with probability $1-o(1)$, $\cG \sim \Prg$ is such that the Glauber dynamics for the Potts model on $\cG$ has $\tmix = n \cdot \exp( \Omega({\|\vdn\|}_\infty))$.  
\end{theorem}

Intuitively, the slowdown comes from the fact that the neighborhood of a vertex of degree ${\|\vdn\|}_\infty$ is a star graph, in which the Ising/Potts Glauber dynamics mixes slowly when $\beta \gg \frac{1}{{\|\vdn\|}_\infty}$. In a random graph at high temperatures (i.e., when $\beta<\beta_u(q,\gamma)$) there is essentially no interference with this effect from the remainder of the graph. Note that the FK-dynamics in the star graph is fast mixing at all temperatures, so this obstruction is not present.

\begin{remark}
We remark that under various decay of correlation conditions (see, e.g.,~\cite{DGJ-dob,Hayes-cond,DGJ-dob1,Yitong}) the mixing time of this chain is known to be $\mathrm{poly}(n)$ when (roughly) $\beta \le 1/{{\|\vdn\|}_\infty}$. This does not contradict Theorem~\ref{thm:Ising-Potts-lower-bound}, which holds when $\beta = \Omega(1)$.
In fact, if one tracks the dependence on $\beta$ in our proof, it gives $\tmix = n \cdot \exp(\Omega(\beta^2 {\|\vdn\|}_\infty))$.
\end{remark}

The known $n^{1+\Omega(\frac{1}{\log \log n})}$ slowdown of the Ising/Potts Glauber dynamics on the Erd\H{o}s--R\'enyi random graph~\cite{MS09,MS}
is a special case of Theorem~\ref{thm:Ising-Potts-lower-bound} where ${\|\vdn\|}_\infty = \Theta(\frac{\log n}{\log \log n})$. 
Below are a few examples where this slowdown can be even more dramatic, indeed stretched exponential in the total number of vertices. 

\begin{example}
     \emph{Power-law degree distributions.} Consider graphical sequences $(\vdn)_n$ satisfying item~(1) in Definition~\ref{def:D-gamma-kappa}, and for which the fraction of degrees of size $\ell$ is $\Theta(\ell^{-\zeta})$. For every $\kappa$, if $\zeta>\kappa+2$, one would have $(\vdn)_n \in \cD_{\gamma,\kappa}$. In such situations, ${\|\vdn\|}_\infty = \Theta(n^{1/\zeta})$ and $\tmix = \exp( \Omega(n^{1/\zeta}))$.  
\end{example}

\begin{example}
	 \emph{Planted high-degree vertices.} Consider a random $\Delta$-regular random graph and change the degree of one vertex to $\Theta(n^{\epsilon})$. If $\epsilon<1/(\kappa+1)$ and $\gamma>\Delta-1$, then $(\vdn)\in \cD_{\gamma,\kappa}$ and $\tmix = \exp( \Omega(n^{\epsilon})$.    
\end{example}

In the above instances where the maximum degree is polynomial in $n$, there is an exponential vs.\ polynomial difference in the high-temperature mixing times of the Ising/Potts Glauber dynamics and of the FK-dynamics. 
At this level, the computational benefits of random-cluster based sampling methods also extend to the often implemented Swendsen--Wang dynamics~\cite{SW}. In particular, using 
the comparison inequalities from~\cite{Ullrich-random-cluster}
the upper bounds of Theorems~\ref{thm:intro:fk-er} and~\ref{thm:intro:general} translate into $O(n^2 \log n)$ upper bounds on the mixing time of the Swendsen--Wang dynamics in those settings. 

\subsection*{Acknowledgements}
The authors thank the anonymous referee for their helpful comments. The research of A.B.\ was supported in part by NSF grants CCF-1850443 and CCF-2143762. R.G.\ thanks the Miller Institute for Basic Research in Science for its support.

\section{Proof outline}\label{sec:proof-sketch}
In this section, we present the main technical contributions in our paper, and describe how they combine to yield the mixing time upper bound of Theorem~\ref{thm:intro:general}.

\medskip
\noindent \textbf{Notational disclaimers.}
Throughout the paper, a subset $\omega \subset E$ is naturally identified with an assignment of $\{0,1\}$, or closed and open, to $E$, via $\omega(e) = 1$ if and only if $e\in \omega$.
The parameters $p,q,\gamma$ will always be fixed quantities, and all constants in little-o, big-O, etc.\ notations may depend on these. As such, we also drop $p,q$ from subscripts when understood from context, e.g., $\pi_G = \pi_{G,p,q}$. All our results should be understood to hold uniformly over all sufficiently large $n$. We use $C$ to generally denote the existence of a constant (possibly depending on fixed parameters such as $p,q,\gamma$) such that the relevant statement holds for all large $n$; for ease of notation, this constant $C$ may change from line to line. 

\subsection{Random graphs}\label{subsec:rg}
We start by describing the locally treelike structure and exponential rate of volume growth of random graphs with fixed degree sequence $(\vdn)_n\in \cD_{\gamma,\kappa}$. 
It will be convenient to work with
the \emph{configuration model}, 
a useful and standard tool for studying random graphs with fixed degree sequence.
The configuration model $\Pcm$ is a distribution over multigraphs on $n$ vertices with degree sequence $\vdn$.
It is defined by giving $d_v$ half-edges to every vertex $v$ and drawing a uniform at random perfect matching on the $\sum_v d_v$ many half-edges to form the $\frac 1 2\sum_v d_v$  edges of the graph~\cite{BollobasConfig}. 
It is a standard fact that for any $(\vdn)_n \in \cD_{\gamma,\kappa}$, and any
sequence of sets $A_n$ of simple graphs on $n$ vertices, we have 
$$\Prg(\cG \in A_n) = o(1) \qquad\mbox{if and only if} \qquad \Pcm(\cG \in A_n) = o(1):$$ see~\cite{BollobasConfig,FriezeBook}. It thus suffices to prove Theorems~\ref{thm:intro:general}-\ref{thm:Ising-Potts-lower-bound} for $\cG \sim \Pcm$.

For a graph $G = (V,E)$ and a vertex $v \in V$, 
we define the ball of radius $R$ around $v$ as: 
$$B_R(v):= \{w \in V: d(w,v) \le R\}\,,$$
where $d(\cdot,\cdot)$ is the graph distance. For a set $B\subset V$ define $E(B) = \{\{v,w\}\in E: v,w\in B\}$. 
\begin{definition}\label{def:k-treelike}\label{def:(L,R)-treelike}
	We say that a graph $G = (V,E)$ is $L$-$\treelike$ if there is a set $H\subset E$ with $|H|\le L$ such that the graph $(V,E\setminus H)$ is a tree. 
	We say that $G$ is $(L,R)$-$\treelike$ 
	if for every $v \in V$ the subgraph $(B_R(v), E(B_R(v))$ is $L$-$\treelike$.
\end{definition}

The following lemma says that small balls of the random graph $\cG\sim \Pcm$ are close to trees. Indeed, for $R/\log_\gamma n$ uniformly less than $1/2$, the ball $B_R(v)$ in $\cG\sim \Pcm$ is typically a random tree with offspring distribution approximately $\Pvd$, defined in~\eqref{eq:Pvd}.

\begin{lemma}\label{lem:random-graph-treelike}
	There exists $\kappa$ such that if $(\vdn)_n \in \cD_{\gamma,\kappa}$ the following holds. For every $\delta>0$, there exists $L = L(\delta)$ such that if $1\le R \le (\frac 12 - \delta)\log_\gamma n$, we have 
	$$
	\Pcm\big(\cG \mbox{ is } (L,R)\mbox{-}\treelike\big) = 1-o(n^{-10})\,.
	$$
\end{lemma}

Using standard concentration estimates for the volume of Galton--Watson trees (see Lemma~\ref{lemma:gw:leaves}), we establish that if $(\vdn)_n \in \cD_{\gamma,\kappa}$, then $\cG\sim \Pcm$ has average exponential rate  $\gamma$ of volume growth.  

\begin{definition}\label{def:volume-growth}
	A graph $G=(V,E)$ on $n$ vertices is said to have $(\gamma,\epsilon)$-volume growth if for every $v\in V$ and every integer $r \in [\epsilon \log_\gamma n, \frac 12 \log_\gamma n]$ the graph has 
	$
	|B_r(v)|\le \gamma^{r}\,.
	$
\end{definition}

\begin{lemma}\label{lem:random-graph-volume-growth}
	Fix $\epsilon\in (0,\frac 12)$. There exists $\kappa(\epsilon)$ such that if $(\vdn)_n\in \cD_{\gamma,\kappa}$, then 
	\begin{align*}
	\Pcm\big(\cG \mbox{ has }(\gamma,\epsilon)\mbox{-growth}\big) \ge 1-o(n^{-10})\,.
	\end{align*}
\end{lemma}

\subsection{Exponential decay and uniqueness on general trees and treelike graphs}\label{subsec:exp-decay}
Given the local tree structure of the random graphs from $\Pcm$, 
to control the decay rate of connectivities of the random-cluster model on $\cG\sim \Pcm$, we need to first understand how these connectivities decay on heterogeneous (i.e., non-regular) trees. The relevant random-cluster measure on the tree requires the addition of \emph{boundary conditions} mimicking the possible presence of open edges in the random graph outside of the treelike ball. Towards this, let us formally define boundary conditions. 
\begin{definition}
     A random-cluster \emph{boundary condition} $\xi$ on $G=(V,E)$ is a partition of $V$, such that the vertices in each element of the partition are identified with one another. The random-cluster measure with boundary conditions $\xi$, denoted $\pi^{\xi}_{G,p,q}$, is the same as in~\eqref{eq:rcmeasure} except the number of connected components $c(\omega)= c(\omega;\xi)$ would be counted with this vertex identification, i.e., if $v,w$ are in the same element of $\xi$, they are always counted as being in the same connected component of $\omega$ in~\eqref{eq:rcmeasure}. 
    The boundary condition can alternatively be seen as external ``wirings" of the vertices in the same element of $\xi$.
\end{definition} 

\begin{remark}The \emph{free} boundary condition, $\xi = 0$, corresponds to the case of no external wirings; i.e., its  partition 
is the one consisting of only of singletons.
For a subset $\partial V \subset V$, the \emph{wired} boundary condition on $\partial V$, denoted $\xi = 1$, is the one whose partition has all vertices of $\partial V$ in the same element (and all vertices of $V\setminus \partial V$ as singletons); i.e., $\xi = \{\partial V\} \cup \bigcup \{v: v\in V\setminus \partial V\}$. For boundary conditions $\xi,\xi'$ we say $\xi \le \xi'$ if $\xi$ is a finer partition than $\xi'$. When $q\ge 1$, the random-cluster model has the following monotonicity property: for any two boundary conditions $\xi\ge \xi'$, $\pi_{G,p,q}^{\xi} \succcurlyeq \pi_{G,p,q}^{\xi'}$ where $\succcurlyeq$ denotes stochastic domination~\cite{Grimmett}.
\end{remark}

Now define the threshold
	\begin{align}\label{eq:p_u}
	p_u(q,\gamma): = 1- \frac{1}{1+\inf_{y>1} h(y)}\,, \qquad \mbox{where} \qquad h(y):= \frac{(y-1)(y^{\gamma}+q-1)}{y^\gamma -y}\,.
	\end{align}
The work~\cite{Haggstrom} studied the random-cluster measure on homogeneous, $d$-ary trees, with wired boundary conditions and identified $p_u(q,d)$ as the \emph{uniqueness threshold} such that whenever $p<p_u(q,d)$, the probability that the root is connected to a distance $h$ in the wired $d$-ary tree goes to zero as $h\to \infty$; a different proof was given in~\cite{BGGSVY}. In~\cite{BlGh21}, it was shown that this decay is in fact exponential with rate $\ps = p/(p+q(1-p))$. However, the methods of those papers do not easily extend to the non-regular setting, where there may be vertices of unbounded degree, but one would expect the threshold for connectivity decay to only depend on the \emph{average} branching rate. In~\cite{Lyons}, it was shown that the analogue $\beta_u(2,\gamma)$ of~\eqref{eq:p_u} gives the correct uniqueness threshold in the case of the Ising model $q=2$, for general (non-homogenous) trees of average branching $\gamma$. {However, the argument there recursed over the single-site spin marginals, and relied on the fact that it was an Ising model whose interactions are nearest-neighbor. In the case of the random-cluster model, interactions between edge-marginals are non-local, and we therefore have to work with a more complicated functional encoding the probability of an edge being downward connected to the wired boundary.}
Combining ideas from~\cite{Lyons} and~\cite{BGGSVY}, 
we are then able to establish uniqueness, and that connectivities decay exponentially with rate $\ps$ on general heterogenous trees of average branching factor $\gamma$ for all $q\ge 1$ and all $p<p_u(q,\gamma)$. When $p<p_u(q,\gamma)$, we have $\ps<1/\gamma$ (see e.g.,~\cite[Theorem 1.5]{Haggstrom}); this indicates by a union bound why there will typically be no connections to the boundary in a tree of average branching $\gamma$.

More formally, let $\T_h = (V(\T_h),E(\T_h))$ be an arbitrary finite tree, rooted at $\rho$, and of height $h$.
Let $\partial \T_h \subset V(\T_h)$ be the set of vertices of $\T_h$ at distance exactly $h$ from $\rho$. For $v \in V(\T_h)$, let $\T_v$ be the subtree of $\T_h$ rooted at $v$, let $h(v)$ denote the height of $\T_v$, and let $\partial \cT_v = \partial \T_h \cap \T_v$.
	For a random-cluster configuration $\omega$ on $\T_h$, let $\cC_\rho(\omega)$ denote the connected component of $\omega$ that contains the root $\rho$ of $\T_h$.
	Finally, let $(1,\circlearrowleft)$ denote the boundary condition that wires all vertices of $\partial \T_h$ together, and also wires them up to the root, and let $\pi_{\cT_h}^{(1,\circlearrowleft)}$ be the random-cluster measure with this boundary condition.
	\begin{lemma}
		\label{lemma:exp:decay:wired:treelike}
		Fix $q \ge 1$, $\gamma > 1$, $p < p_u(q,\gamma)$, and $\varepsilon \in [0,1)$. Suppose that
		$|\partial \cT_v| \le \gamma^{h(v)} $ for every $v \in V(\T_h)$ with $h(v) > \varepsilon h$.
		Then, there exists a constant $C = C(p,q,\gamma)$ such that for any  $u\in \partial\T_h$ 
		\begin{align*}
		\pi_{\T_h}^{(1,\circlearrowleft)} (\omega: u\in \cC_\rho(\omega) ) \le  C \ps^{(1-\varepsilon)h}\,.
		\end{align*}
\end{lemma}
We note that the condition that $|\partial \cT_v| \le \gamma^{h(v)} $ for every $v \in V(\T_h)$ with $h(v) > \varepsilon h$ in the lemma
holds with high probability for random trees with averaging branching $\gamma$: see Corollary~\ref{cor:random-tree-has-tree-growth}.
In addition, the exponential decay rate in Lemma~\ref{lemma:exp:decay:wired:treelike} is essentially optimal, and together with Lemmas~\ref{lem:random-graph-treelike}--\ref{lem:random-graph-volume-growth}, allows us to derive precise estimates on the exponential decay of connectivities on the treelike balls around each
vertex of the random graph $\cG\sim \Pcm$.
We will actually need a sharp bound on the rate of influence decay between the boundary and the center of the ball $B_R(v)$; we find that this is the square of the rate of connectivity decay on a corresponding tree of depth $R$. 
(Intuitively, this is because \emph{two} disjoint paths are required to reach the center of the ball in order for the boundary to have any effect on it.)
To be more precise, let $G=(V,E)$ be a graph and for $v \in V$, let $E_v \subseteq E$ denote the set of edges incident to $v$.

\begin{definition}\label{def:k-sparse-bc}
	A random-cluster boundary condition $\xi$ on a graph $H$ is said to be $K$-$\sparse$ if the number of vertices in non-trivial (non-singleton) boundary components of $\xi$ is at most $K$. 
\end{definition}

\begin{theorem}\label{theorem:influence-probability-new}
	Fix $\gamma > 0$, $q\ge 1$, and $p<p_u(q,\gamma)$. Suppose $G$ is $(L,R)$-$\treelike$ for some $L$ and some $R\le \frac{1}{2}\log_\gamma n$. Also suppose $G$ has $(\gamma,\epsilon)$-volume growth for some $\epsilon>0$ sufficiently small. 
	There exists a constant $C > 0$ such that for every $v \in G$, and any two $K$-$\sparse$ boundary conditions $\xi$ and~$\tau$ on $B_R(v)$: 
	\begin{align*}
	\|\pi_{B_R(v)}^{\xi}(\omega(E_v)\in \cdot)- \pi_{B_R(v)}^{\tau}(\omega(E_v)\in \cdot ) \|_\tv \le C 
	\ps^{(2-CL\sqrt{\epsilon})R}\,.
	\end{align*}
\end{theorem}
A similar influence decay bound was proven for the regular case in~\cite[Section 5.2]{BlGh21}.

\subsection{Shattering of the FK-dynamics}\label{subsec:shattering}
With Theorem~\ref{theorem:influence-probability-new} in hand, the core of our argument becomes establishing that the boundary conditions induced by the FK-dynamics chains from all possible initializations, on balls of radius $R\le \frac{1}{2}\log_\gamma n$ are $K$-$\sparse$. This will follow from \emph{shattering} of the FK-dynamics, by which we mean the time at which the connected components of the FK-dynamics configuration are all small, say of size $n^{o(1)}$. 

\begin{remark}It will be technically convenient to prove our results in continuous time instead of discrete time. In the continuous-time FK-dynamics, each edge of the graph has a rate-1 Poisson clock and every time a clock rings, the corresponding edge is updated 
as in the discrete-time version of the FK-dynamics; that is,
according to the conditional distribution given the configuration off of this edge. It is a standard fact (see e.g.,~\cite[Theorem 20.3]{LP}) that the
discrete-time mixing time is comparable to $|E(\cG)|$ times the continuous-time mixing time. 
It therefore suffices for us to establish the mixing time bounds of Theorems~\ref{thm:intro:fk-er} and~\ref{thm:intro:general} as $\Theta(\log n)$ bounds for the continuous-time version of the FK-dynamics. From this point on, we let $X_{t}^{x_0}$ denote the continuous-time FK-dynamics on $\cG$ initialized from the configuration $x_0$, and use the superscripts $1$ and $0$ to denote the full (all-open) and empty (all-closed) configurations, respectively. 
\end{remark}

We now formalize what we mean by a shattered random-cluster configuration, and establish that the FK-dynamics shatters after an $O(1)$ continuous-time burn in period.

\begin{definition}\label{def:(k,r)-sparse}
	A random-cluster configuration $\omega$ on $\cG = (V(\cG),E(\cG))$ is \emph{$(K,R)$-$\sparse$} 
	if, for every $v \in V(\cG)$,
	the boundary conditions induced on $B_R(v)$ by $\omega(E(\cG)\setminus E(B_r(v)))$ are $K$-$\sparse$. 
\end{definition}

\begin{theorem}\label{thm:k-R-sparse-whp}
	Fix $q\ge 1$, $\gamma>0$ and $p<p_u(q,\gamma)$. For every $\delta>0$, there exists $\kappa$ such that if $(\vdn)_n \in \cD_{\gamma,\kappa}$, there exists $T= T(p,q,\gamma)$ and $K = K(p,q,\gamma,\delta)$ such that for any $t\ge T$, and every $1\le R \le (\frac 12 -\delta)\log_\gamma n$, with probability $1-o(1)$, $\cG\sim \Pcm$ is such that 
	\begin{align}
	\bbP\big(X_{t}^1 \mbox{ is $(K,R)$-$\sparse$}\big)\ge 1-o(n^{-5})\,.
	\end{align}
\end{theorem}

    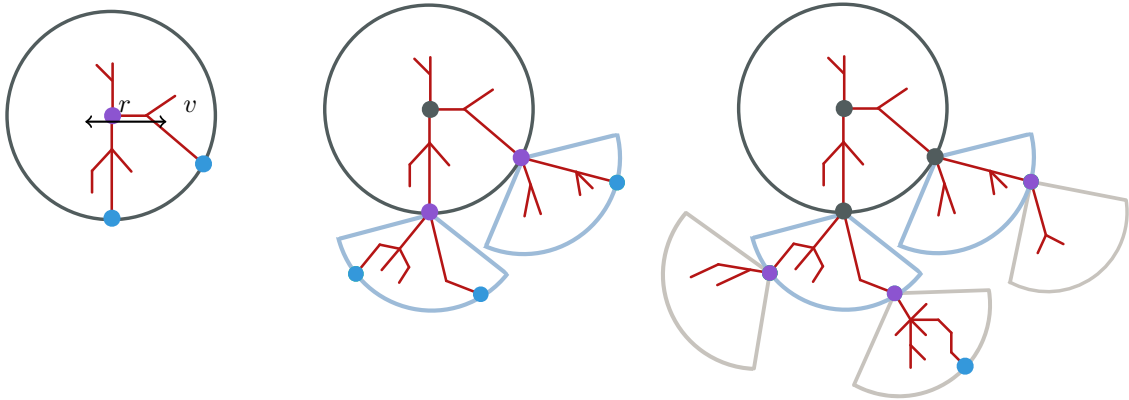
\begin{figure}
        \centering
        \begin{tikzpicture}
        \node at (0,0) {
        \includegraphics[width = .926\textwidth]{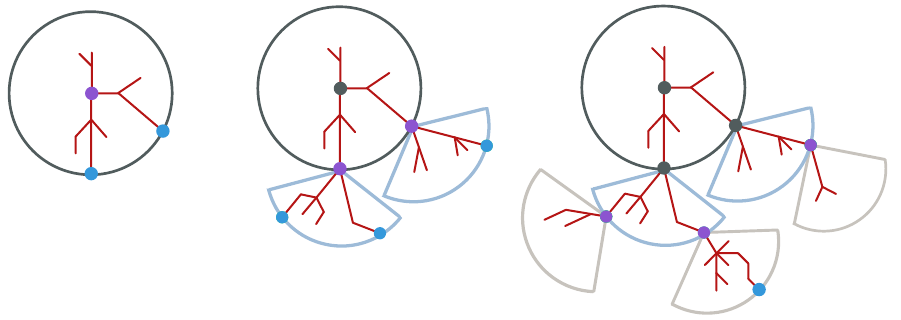}};
        \draw[<->, thick] (-6.45,1.05)--(-5.37,1.05);
        \node[font = \small] at (-5.925, 1.25) {$r$};
        \node[font = \small] at (-5.05, 1.25) {$v$};

        \end{tikzpicture}
        \vspace{-.35cm}
        \caption{Three ``generations" of the revealing procedure. In each figure, the purple vertices are the current generation of exposed vertices; the revealing procedure reveals the ball of radius $r$ around such a vertex $v$, and a dominating localized FK-dynamics configuration $\tilde \omega(B_r(v))$ on that ball. The next generation of exposed vertices (blue) consists of those on the boundary $B_r(v)$ that are in the connected component of $v$ in the configuration $\tilde \omega(B_r(v))$. Exposed vertices from previous generations are then colored black.}
        \label{fig:overview}
    \end{figure}

Our starting point for the proof of Theorem~\ref{thm:k-R-sparse-whp} is a proof of shattering for the FK-dynamics on $\Delta$-regular random graphs from~\cite{BlGh21}. Hence, as in~\cite{BlGh21}, our proof relies on a delicate simultaneous revealing procedure for the random graph, along with the connected component of a vertex $v$ in $X_t^1$, showing that after a burn-in period, the configuration $X_t^1$ is shattered. The revealing scheme for the component of a vertex $v$ in the FK-dynamics chain $X_t^{1}$ roughly proceeds as follows (see the accompanying Figure~\ref{fig:overview}). First ``expose" the starting vertex $v$, and iteratively, for each exposed vertex $u$ do the following:
\begin{enumerate}
    \item Reveal the ball $B_r(u)$ in the random graph for a large $r= O(1)$;
    \item Reveal a configuration $\tilde \omega(B_r(u))$ that dominates the configuration of the FK-dynamics at time $t$ on $B_r(u)$. This configuration will come from simulating FK-dynamics that ignores all updates outside of $B_r(u)$ (effectively inducing the wired boundary condition on $B_r(u)$) and thus can be obtained independently of the dynamics on the rest of the graph;
    \item Add to the set of exposed vertices all vertices of $\partial B_r(u)$ that get connected to $u$ in $\tilde \omega(B_r(u))$. 
\end{enumerate}
The key point of the argument is then to stochastically dominate the exposed vertices by a branching process, which can be shown to be sub-critical (see Lemma~\ref{lemma:exp:decay:wired:treelike}). 
In our setting, the heterogeneity of the degrees causes substantial complications to the argument from~\cite{BlGh21}, because in balls where the branching rate is locally larger than $\gamma$, the overlayed FK-dynamics configuration will actually be highly connected. The presence of high degrees also destroys the $O(1)$ bounds on the maximum number of new vertices that could possibly get exposed in step (3) above; this complicates relevant concentration arguments, as our branching process martingale will no longer have bounded increments.

\subsection{Organization of the remainder of the paper}
In Section~\ref{sec:uniqueness}, we prove that whenever $p<p_u(q,\gamma)$, the random-cluster model on trees of average branching $\gamma$ is in its uniqueness regime, and deduce Lemma~\ref{lemma:exp:decay:wired:treelike}. In Section~\ref{sec:random-graph-estimates}, we prove key properties of the random-graph model $\Pcm$, including Lemmas~\ref{lem:random-graph-treelike}--\ref{lem:random-graph-volume-growth}. Section~\ref{sec:shattering} contains the proof of shattering of the FK-dynamics, and in particular Theorem~\ref{thm:k-R-sparse-whp}. In Section~\ref{sec:correlation-decay-treelike}, we bound the rate of influence decay (Theorem~\ref{theorem:influence-probability-new}) and mixing time (Lemma~\ref{lemma:local-mixing}) in treelike graphs with sparse boundary conditions. Section~\ref{sec:proof-of-main-theorem} combines these ingredients to conclude the $\Theta(n\log n)$ bound on the FK-dynamics for Theorem~\ref{thm:intro:general}. Finally, Section~\ref{sec:potts-slow-down} proves the exponential (in ${\|\vdn\|}_\infty$) lower bound on the Potts Glauber dynamics of Theorem~\ref{thm:Ising-Potts-lower-bound}.

\section{Uniqueness and exponential decay on general trees}\label{sec:uniqueness}
Our main result in this section is to prove Lemma~\ref{lemma:exp:decay:wired:treelike}. We also use this section to deduce some corollaries about uniqueness of infinite-volume random-cluster and Potts measures on general trees of average branching $\gamma$, and apply these results to super-critical Galton--Watson trees. 

\subsection{Exponential decay of connectivities on general trees}
We begin by considering the probability 
$\varphi(\rho)$ that the root $\rho$ is connected to $\partial\T_h$ in $\omega \sim \pi_{\T_h}^1$, and show that $\varphi(\rho) = \varphi_{p,q,\cT_h}(\rho)$ decays exponentially with $h$ for all trees of \emph{average} branching $\gamma$, whenever $p<p_u(q,\gamma)$. 

\begin{lemma}
	\label{lemma:tree:decay}
	Fix $\gamma \ge 1$ and $q \ge 1$ and let $p < p_u(q,\gamma)$. There exists $\theta = \theta(p,q,\gamma) \in (0,1)$ and $C = C(p,q,\gamma)$ such that if $|\partial\T_h| \le \gamma^{h}$, then 
	$
	\varphi(\rho) \le C  \theta^h.
	$
\end{lemma}

\begin{proof}
    Fix $h$ and fix  $\cT_h$ having $|\partial \cT_h|\le \gamma^h$. 
	Recall that for $v \in V(\T_h)$, $\T_v$ denotes the subtree of $\T_h$ rooted at $v$.
	Let $Z(v) = Z_{\T_v,p,q}$ denote the partition function corresponding to $\pi_{\T_v}^1$ (the random-cluster measure on $\cT_v$ with all its vertices in $\partial \cT_h$ wired together). Let  $Z_1(v)$ be the contribution to $Z(v)$ from the 
	configurations on $\T_v$ that contain an open path between $v$ and $\partial\T_h$. Similarly, let $Z_0(v)$
	denote the contribution from the configurations that do not have such a path. Note that $Z(v) = Z_0(v)+Z_1(v)$ and $\varphi(\rho) = \frac{Z_1(\rho)}{Z_0(\rho)+Z_1(\rho)}$.
	
	For $v \in V(\T_h)$, let $N_v$ denote the set of children of $v$.
	Using tree recurrences, and the definition of~\eqref{eq:rcmeasure}, the following identities can be checked; the  proof is similar to that in~\cite[Lemma 33]{BGGSVY} and is provided later.	
	\begin{fact} \label{fact:recurrence} 
	Let $t = p/q + 1 - p$. For any $v \in V(\T_h)$,
		\begin{align*}
			Z_1(v) &= q \prod_{w \in N_v} \left(\frac{Z_1(w)}{q} + \frac{t Z_0(w)}{q}\right) -  q \prod_{w \in N_v} \left(\frac{(1-p)Z_1(w)}{q} + \frac{t Z_0(w)}{q}\right), \\
			Z_0(v) &= q^2 \prod_{w \in N_v} \left(\frac{(1-p)Z_1(w)}{q} + \frac{t Z_0(w)}{q}\right).			 
		\end{align*}
	\end{fact}
 Now consider the function $f: V(\T_h) \to \R$ defined as $$f(v) := q \frac{Z_1(v)}{Z_0(v)}+1\,.$$ 
Using the identities in Fact~\ref{fact:recurrence}, one easily sees that 
\begin{align*}
    f(v) = \prod_{w\in N_v} g(f(w)) \qquad \mbox{for} \qquad g(x):= \frac{x+(q-1)(1-p)}{(1-p)x+p+(q-1)(1-p)}\,.
\end{align*}
    The following calculus bound, which is proved later, holds for the function $g$.
	\begin{fact}
	\label{fact:g:bound}
	Fix $q,\gamma \ge 1$ and $p < p_u(q,\gamma)$. There exists $\xi\in (0,1/\gamma)$ such that $g(x) \le x^{1/\gamma-\xi}$ for all $x \ge 1$.
    \end{fact}
    Now, let $D_k \subset V(\T_h)$ denote the set of vertices at distance $k$ from the root $\rho$ and let $L_k \subseteq D_k$ be the set of leaves at distance $k$ from $\rho$. Setting $\zeta = 1/\gamma-\xi$, 
    and using the facts that $g(1) = 1$, and that if $w$ is a leaf that does not belong to $\partial\T_h$ then $Z_1(w) = 0$ and $g(f(w))=1$,
    we obtain
	\begin{align*}
	f(\rho) &= \prod\limits_{w \in D_1} g(f(w)) = \prod\limits_{w \in D_1\setminus L_1} g(f(w)) \le \prod\limits_{w \in D_1\setminus L_1} f(w)^\zeta\,.
	\end{align*}
		Iterating, and using the fact that $g(x) \le (1-p)^{-1}$ for all $x \ge 1$, we have
		\begin{align*}
	f(\rho)& \le \prod\limits_{w \in D_{h-1}\setminus L_{h-1}} f(w)^{\zeta^{h-1}} \le 
	(1-p)^{- \zeta^{h-1} |\partial\T_h|}\,.
	\end{align*}
	
	Then, recalling $\varphi(\rho) = Z_1(\rho)/(Z_0(\rho)+ Z_1(\rho))$, we get
	\begin{align*}
	\varphi(\rho) \le \frac{Z_1(\rho)}{Z_0(\rho)} = \frac{f(\rho)-1}{q} \le \frac{1}{q} \left(\frac{1}{1-p}\right)^{\zeta^{h-1} \cdot |\partial\T_h|} -  \frac{1}{q} \le \frac{\zeta^{h} \cdot |\partial\T_h|}{q(1-p)^{1/\zeta}}\,,
	\end{align*}	
	where the last inequality follows from the fact that $a^x \le 1+ ax$ when $a \ge 1$ and $x \in [0,1]$ since $\zeta^{h} \cdot |\partial\T_h| \le 1$ when $|\partial\T_h| \le \gamma^h$.
	The proof is complete by setting $\theta = 1 - \gamma \xi$.
\end{proof}

With Lemma~\ref{lemma:tree:decay} on hand, we can now provide the proof of Lemma~\ref{lemma:exp:decay:wired:treelike}, which
gives a precise bound on the rate of decay under stronger assumptions for the growth of $\T_h$. 
\begin{proof}[\textbf{\emph{Proof of Lemma~\ref{lemma:exp:decay:wired:treelike}}}]
    Let $u$ be a vertex in $\partial\T_h$ and for $v \in V(\T_h)$
	let $\vartheta(v,u)$ be the probability that $v$ is connected to $u$ in $\T_v$ under $\pi_{\T_v}^1$.
	Let $\vartheta^\circlearrowleft(v,u)$ be the probability of the same event under $\pi_{\cT_v}^{(1,\circlearrowleft)}$.
	
	By monotonicity we have $\vartheta(\rho,u) \le \vartheta^{\circlearrowleft}(\rho,u)$ and by a standard comparison between boundary conditions (see e.g., Lemma~\ref{lemma:simple-rc-bound}), we have $\vartheta^\circlearrowleft(\rho,u) \le q \vartheta(\rho,u)$. Hence, it suffices to bound $\vartheta(\rho,u)$. 
	Consider the unique path $P=(\rho=v_0,v_1,\dots ,u = v_h)$ between $\rho$ and $u$.
	Let $N_v$ denote the set of children of $v$.
	For $w \in N_{v_0}$, 
	let $I_w$ be the indicator function of the event that there is a 
	path from $v_0$ to $\partial \T_h$ going through $w$; set $I = \sum_{w \in N_\rho: w \neq v_1} I_w$.
	Then, we can write
	\begin{align*}
	\vartheta(\rho,u) \leq p  \cdot \pi_{\cT_{h}}^1(I \geq 1) \cdot \vartheta^\circlearrowleft(v_1,u)+\ps \vartheta(v_1,u) \le \vartheta(v_1,u) \left[pq^2 \cdot \pi_{\cT_{h}}^1(I \geq1)+ \ps\right]\,.
	\end{align*}
	In the first inequality, we used the fact that in order for the root to be connected to the vertex $u$,
	it is required that the root 
	is connected to $v_1$,
	and that  $v_1$ is connected to $u$ in its sub-tree. 
	The former event occurs with probability $p$ or $\ps$,
	depending on whether or not the root is connected to $\partial \cT_h$ through any child besides $v_1$.

	Let $\varphi(w)$ denote the probability that $w$ is connected to $\partial \cT_w$ under $\pi_{\T_w}^1$. Then, $\pi_{\cT_{h}}^1(I \geq1) \le \varphi(v_0)$ and since
	$|\partial \cT_{v_0}| \le \gamma^{h(v_0)}$ by assumption, Lemma~\ref{lemma:tree:decay} implies that for suitable constants $\theta  = \theta(p,q,\gamma) \in (0,1)$ and $C = C(p,q,\gamma) > 0$, we have 
	$\pi_{\cT_{h}}^1(I \geq1) \le C \theta^{h(v_0)}$
	Thus, setting $a = \frac{C pq^2}{\ps}$, and continuing the recursion
	we
	obtain 
	\begin{align*}
	\vartheta(\rho,u) \leq \ps \cdot \vartheta(v_1,u) \left[1 + a \cdot \theta^{h(v_0)}  \right] &\le \ps^{(1-\varepsilon)h} \prod_{i=0}^{(1-\varepsilon)h} \left[1 + a \cdot \theta^{h(v_i)}  \right] \\
	&\le \ps^{(1-\varepsilon)h}  \exp\Big[ a  \sum\nolimits_{i=0}^{(1-\varepsilon)h} \theta^{h(v_i)} \Big] \le  A \ps^{(1-\varepsilon)h},
	\end{align*}
	for a suitable constant $A = A(p,q,\gamma) > 0$.
	Hence, $\vartheta^{\circlearrowleft}(\rho,u) \le A q^2 \ps^{(1-\varepsilon)h}$ and the result follows.
\end{proof}

\subsection{Proofs of auxiliary facts}

We now provide the deferred proofs of Facts~\ref{fact:recurrence} and~\ref{fact:g:bound}. 

\begin{proof}[\textbf{\emph{Proof of Fact~\ref{fact:recurrence}}}]
For $v \in V(\T_h)$, let $N_v$ denote the set of children of $v$
and let $\partial \cT_v \subseteq \partial \cT_h$ be the set of vertices of $\T_v\cap \partial \cT_h$.
We compute $Z_1(v)$ and $Z_0(v)$ by partitioning the space of configurations according to which subtrees of $v$ among $\{\cT_u: u\in N_v\}$ are connected to the $\partial \cT_v$. For each configuration $\omega$, the connectivity of the children of $v$ to their respective boundaries is encoded by a vector $a_\omega \in \{0,1\}^{N_v}$,
where for $u \in N_v$ we have $a_\omega(u) = 1$ when $u$ is connected to $\partial \cT_u$ by a path in $\cT_u$.

We start by proving the identity for $Z_1(v)$.
In this case, we only consider configurations such that ${\|a_\omega\|}_1 \ge 1$.
For a fixed vector $a_\omega$ such that ${\|a_\omega\|}_1 = k$, let
$u_1,\dots,u_{k} \in N_v$ be the neighbors of $v$ for which $a_\omega(u_i) = 1$,
and
let $\hat{u}_1,\dots,\hat{u}_{l} \in N_v$ be the neighbors of $v$ for which $a_\omega(\hat{u}_i) = 0$; hence $l = |N_v| - k$.
Any random-cluster configuration $\omega$ of $\T_v$, can be partitioned into the configuration
on $E(\{v\} \cup \bigcup_{u_i} \T_{u_i})$
and the configuration on $E(\{v\} \cup \bigcup_{\hat{u}_i} \T_{\hat{u}_i})$.

Given a vector $a$, let $W_1(v,a,1)$ denote 
the total weight under the wired boundary condition of the random-cluster configurations on $E(\{v\} \cup \bigcup_{u_i} \T_{u_i})$
that contain a $v$ to $\partial \cT_v$ connection
and a $u_i$ to $\partial \cT_{u_i}$ path in $\cT_{u_i}$ for every $i \in \{1,\dots,k\}$. Similarly,
let $W_1(v,a,0)$ denote the total weight of the configurations on $E(v \cup \bigcup_{\hat{u}_i} \T_{\hat{u}_i})$ in which there is no path between $\hat u_i$ and $\partial \cT_{\hat{u}_i}$ in $\cT_{\hat{u}_i}$ for $i \in \{1,\dots,l\}$. 
Since conditioning on a disconnected configuration on $E(\{v\} \cup \bigcup_{\hat{u}_i}\cT_{\hat u_i})$ has no effect on the weight of the configuration on $E(\{v\} \cup \bigcup_{u_i} \T_{u_i})$, we have the identity 
\begin{equation}
\label{eq:z}
Z_1(v) = \frac 1q \sum_{a \in \{0,1\}^{N_v}:{\|a\|}_1 \ge 1} W_1(v,a,1) W_1(v,a,0)\,.
\end{equation}
Here, the $1/q$ factor comes from merging the two wired boundary components when ${\|a\|}_1 < |N_v|$; if  ${\|a\|}_1 = |N_v|$, we set $W_1(v,a,0)=q$.

We compute $W_1(v,a,1)$ first. We use $\Omega_1(\T_x)$ (resp., $\Omega_0(\T_x)$) for the set of all random-cluster configurations on the subtree $\cT_x$ in which there is (resp., there is not) an open path between $x$ and $\partial \cT_h$ in $\cT_x$.
For a configuration $\eta_i \in \Omega_0(\T_{u_i}) \cup \Omega_1(\T_{u_i})$, we use $\w(\eta_i) = p^{|\eta_i|}(1-p)^{|E(\T_{u_i})|-|\eta_i|} q^{c_1(\eta_i)}$
for the weight of the random-cluster configuration on $\T_{u_i}$ under the wired boundary condition; i.e., $c_1(\eta_i)$ corresponds to the number of connected components on $\eta_i$ taking into consideration the wired boundary condition.
Then, accounting also for the configuration in the edges between $v$ and the $u_i$'s, we have
\begin{align}
W_1(v,a,1) &= \sum_{\eta_1 \in \Omega_1(\T_{u_1})} \cdots \sum_{\eta_k \in \Omega_1(\T_{u_k})} \left(\prod_{i=1}^{k} \w(\eta_i) \right) \frac{1}{q^{k-1}}\left(\sum_{i=1}^k \binom{k}{i} p^{i} (1-p)^{k-i}\right) \label{eq:long:1} \\
&= \frac{1-(1-p)^k}{q^{k-1}} \prod_{i=1}^k Z_1(u_i)\,. \label{eq:long:1.1} 
\end{align}
The re-scaling in~\eqref{eq:long:1} by $\frac{1}{q^{k-1}}$ comes from the fact that the $k$ boundary components in each subtree are all merged into a single component.
By similar reasoning, when $\|a\|_1 < |N_v|$
\begin{align}
W_1(v,a,0) &= \sum_{\eta_1 \in \Omega_0(\T_{\hat{u}_1})} \cdots \sum_{\eta_l \in \Omega_0(\T_{\hat{u}_l})} \left(\prod_{i=1}^{l} \w(\eta_i) \right) \frac{1}{q^{l-1}}\left(\sum_{i=0}^l\binom{l}{i} \left(\frac{p}{q}\right)^{i} (1-p)^{l-i}\right) \label{eq:long:2}  \\
&= \frac{(1-p+p/q)^l}{q^{l-1}} \prod_{i=1}^l Z_0(u_i)\,. \label{eq:long:2.1}
\end{align}
Note that in~\eqref{eq:long:2}, in addition to the re-scaling by $\frac{1}{q^{l-1}}$ from merging the boundary components, any edge between $v$ and one of its children decreases the number of components by $1$; hence the $q^{-i}$ in the term $(\frac{p}{q})^i$. 
 
Recall that $t = 1-p+p/q$. Plugging \eqref{eq:long:1.1}  and \eqref{eq:long:2.1} into~\eqref{eq:z} we obtain
\begin{align*}
Z_1(v) 
&= q \sum_{a \in \{0,1\}^{N_v}:{\|a\|}_1 \ge 1}(1-(1-p)^{{\|a\|}_1}) \prod_{w \in N_v:a(w)=1} \frac{Z_1(w)}{q} \prod_{w \in N_v:a(w)=0} \frac{t Z_0(w)}{q}. 
\end{align*}
Observe next that
\begin{align*}
\sum_{a \in \{0,1\}^{N_v}:{\|a\|}_1 \ge 1}\prod_{w \in N_v:a(w)=1} \frac{Z_1(w)}{q} &\prod_{w \in N_v:a(w)=0}  \frac{t \cdot Z_0(w)}{q} \\
&= \prod_{w \in N_v} \left(\frac{Z_1(w)}{q}+\frac{t \cdot Z_0(w)}{q}\right) - \prod_{w \in N_v} \frac{t \cdot Z_0(w)}{q},
\end{align*}
and
\begin{align*}
\sum_{a \in \{0,1\}^{N_v}:{\|a\|}_1 \ge 1}(1-p)^{{\|a\|}_1}  & \prod_{w \in N_v:a(w)=1} \frac{Z_1(w)}{q} \prod_{w \in N_v:a(w)=0}  \frac{t \cdot Z_0(w)}{q} \\
&= \prod_{w \in N_v} \left(\frac{(1-p)Z_1(w)}{q}+\frac{t \cdot Z_0(w)}{q}\right) - \prod_{w \in N_v} \frac{t \cdot Z_0(w)}{q}.
\end{align*}
Hence, 
\begin{align*}
Z_1(v) &= q \prod_{w \in N_v} \left(\frac{Z_1(w)}{q}+\frac{t \cdot Z_0(w)}{q}\right)  - q \prod_{w \in N_v} \left(\frac{(1-p)Z_1(w)}{q}+\frac{t \cdot Z_0(w)}{q}\right),
\end{align*}
as claimed. The expression for $Z_0(v)$ can be derived from an analogous and slightly simpler argument and is thus omitted.
\end{proof}

\begin{proof}[\textbf{\emph{Proof of Fact~\ref{fact:g:bound}}}]
	We first consider the interval $x \in [1,1+\eta]$ for some $\eta > 0$ small. It can be checked that 
	\begin{align*}
	 g'(x) &= \frac{p (p + q - p q)}{(-1 + q + x - p (-2 + q + x))^2}\,, \\  
	g''(x) &= \frac{-2 p(1 - p) (q+(1-p)p)}{((1-p)x+p+(1-p)(q-1))^3}\,.
	\end{align*}
	Hence, $g'(1) = \ps$ and $|g''|$ is decreasing for $x \ge 1$. Then, from the Taylor expansion of $g$ at $1$, we get
	\begin{align}
	\label{eq:g:ineq2} 
	g(x) &\le 1 + \ps (x-1) + c\eta^2 ,
	\end{align}
	where $c = c(p,q) > 0$ is suitable constant. Similarly, using the Taylor expansion of $x^{1/\gamma-\xi}$ at $1$, we obtain
	$$
	x^{1/\gamma-\xi} \ge 1 + ({1}/{\gamma}-\xi) (x-1) - c' \eta^2
	$$
	for a suitable constant $c' = c'(\gamma,\xi) > 0$.
	Since $\ps < 1/\gamma$ when $p < p_u(q,\gamma)$, then for sufficiently small $\xi$ and $\eta$ (depending on $p,q,\gamma$) we have 
 	$
 	g(x)\le x^{1/\gamma -\xi}
 	$
 	as desired. 
 	
 	We next observe that since $g(x) \le \frac{1}{1-p}$,
 	we have $g(x) \le x^{1/\gamma - \xi}$ for all $x\ge K$ for $K$ sufficiently large (depending on $p,q,\gamma$), importantly independent of $\xi$ as long as $\xi<1/(2\gamma)$, say.
 	
 	It remains to consider the case when $x \in (1+\eta,K)$. For this, let us give an auxiliary form of $p_u(q,\gamma)$: 
 	 	\begin{align}\label{eq:p-u-alternate-form}
 	    p_u(q,\gamma)= \sup\Big\{p: \sup_{x>1} \{g_p(x) - x^{1/\gamma}\} \le 0\Big\}\,.
 	\end{align}
 	 (where we have added the $p$ subscript to $g$ to emphasize the $p$ dependence there). 
    Let us first conclude the proof assuming the equality of~\eqref{eq:p-u-alternate-form}. By direct computation, it can be checked that $\frac{\partial g(x)}{\partial p} > 0$
 	whenever $x > 1$. Hence, fixing $p'\in (p, p_u(q,\gamma))$ for every $x > 1$
 	we have $g_p(x) < g_{p'}(x)$, and by continuity $g_p(x) \le g_{p'}(x) -\delta$ for a sufficiently small $\delta > 0$. By continuity, in fact there exists a uniform choice of $\delta>0$ such that 
 	\begin{align*}
 	    g_p(x)<g_{p'}(x)- \delta \qquad \mbox{for all $x\in [1+\eta,K]$}\,.
 	\end{align*}
 	At the same time, for $\xi$ sufficiently small, depending on $\delta,\gamma,K$, we have 
 	$$
 	|x^{1/\gamma}-x^{1/\gamma-\xi}| \le \delta \qquad \mbox{for all $x\in [1+\eta,K]$}\,.
 	$$
    Combining these two, and using~\eqref{eq:p-u-alternate-form}, we see that for all $x\in [1+\eta,K]$, 
 	$$
 	g_p(x) \le g_{p'}(x)-x^{1/\gamma}+x^{1/\gamma-\xi} \le x^{1/\gamma -\xi} \qquad\mbox{for all $x\in [1+\eta,K]$}\,.
 	$$

 	It remains to establish the equality~\eqref{eq:p-u-alternate-form}. We first rewrite the definition of $p_u(q,\gamma)$ from~\eqref{eq:p_u} as 
 	\begin{align*}
 	    p_u(q,\gamma) = \sup \Big\{ p: \sup_{y>1}\{p-1 + \frac{1}{1+h(y)}\} \le 0\Big\}\,.
 	\end{align*}
 	It therefore suffices to establish that 
 	\begin{align*}
 	    \sup_{y>1}\{ p-1 +\frac{1}{1+h(y)} \} \le 0 \iff \sup_{x>1} \{g(x) - x^{1/\gamma}\}\le 0\,.
 	\end{align*}
 	By substituting $y = x^{1/\gamma}$, and calculating, this reduces to showing that for every $y>1$, 
 	\begin{align*}
 	    \frac{ - (1-p)y^{\gamma+1} + y^{\gamma} - [p+(q-1)(1-p)]y + (q-1)(1-p)}{y^\gamma - y + (y-1)(y^{\gamma} + q-1)} \le 0  
 	\end{align*}
 	if and only if 
 	\begin{align*}
 	    \frac{ - (1-p)y^{\gamma+1} + y^{\gamma} - [p+(q-1)(1-p)]y + (q-1)(1-p)}{(1-p)y^\gamma + p +(q-1)(1-p)} \le 0\,.
 	\end{align*}
 	This equivalence follows because the numerators are the same, and the denominators are both positive whenever $\gamma>1$, $q\ge 1$ and $y>1$.  
\end{proof}

\subsubsection{Uniqueness in general trees}

As a consequence of the decay of the root-to-leaf connectivity we have established, it follows that there is a unique infinite wired random-cluster measure
 whenever $p<p_u(q,\gamma)$ on infinite trees with average branching $\gamma$. 
 The random-cluster measure on the infinite wired tree is defined using the Dobrushin-Lanford-Ruelle (DLR) formalism (see, e.g.,~\cite{Haggstrom,Grimmett}); in particular, the wired boundary condition corresponds to counting all infinite connected components as one. 
 
 Let $\T$ be an infinite tree, let $D_h \subset V(\T)$ denote the set of vertices at distance $h$ from the root of $\T$ and define the branching rate $Br(\cT)$ per~\cite{Lyons} as: 
\begin{align*}
    Br(\T) = \inf\big\{\lambda>0: \inf_h {|D_h|}{\lambda^{-h}} = 0\big\}.
\end{align*}
Observe that if $Br(\T) < \gamma$, then $|D_h| < \gamma^h$ for all sufficiently large $h$.
We prove the following.
\begin{corollary}\label{cor:rc:unique}
    Fix $q \ge 1$, $\gamma > 1$ and  $p < p_u(q,\gamma)$.
    Suppose $\T$ is an infinite tree with ${Br}(\T) < \gamma$.
    Then, there is a unique infinite-volume random-cluster measure on $\T$ under the wired boundary condition.
\end{corollary}
\begin{proof}
Let $\T_h$ denote the subtree of $\T$ that includes all vertices at distance at most $h$ from the root $\rho$ of $\T$.
Let $\pi_\T^1 = \lim_{h \rightarrow \infty} \pi_{\T_h}^1$.
It was established in~\cite[Lemma 3.1]{Haggstrom} that the limiting measure $\pi_\T^1$ 
is a random-cluster measure with parameters $p$ and $q$ and, moreover, that 
any other random-cluster measure on $\T$ with the same parameters is stochastically dominated by $\pi_\T^1$. (We note that Lemma 3.1 from~\cite{Haggstrom} is stated for the case when $\T$ is a  homogeneous tree, but the proof there does not use this assumption, and the result clearly extends to general trees.) Now, since ${Br}(\T) < \gamma$ we have that $|D_h| < \gamma^h$ for sufficiently large $h$, and so Lemma~\ref{lemma:tree:decay} implies that $\pi_{\T}^1(\rho \leftrightarrow \infty) =  \lim_{h \rightarrow \infty} \pi_{\T_h}^1(\rho \leftrightarrow \partial \cT_h) = 0$.
This implies that the conditional probability that any edge $e$ is present, given the configuration outside of $e$, is $\ps$ with probability $1$ (see, e.g., the proof of Theorem 1.8 in ~\cite{Haggstrom}). Hence, $\pi_{\T}^1$ corresponds to the i.i.d.\ distribution on $\{0,1\}^{E(\T)}$ with edge probability $\ps$.
By the same argument, the same is true for any other random-cluster $\mu$ since $\mu \preceq \pi_{\T}^1$, and the result follows.
\end{proof}

\begin{corollary}
    Fix $q\ge 2$ integer, $\gamma>1$ and $p<p_u(q,\gamma)$. Suppose $\cT$ is an infinite tree with $Br(\gamma)<\gamma$. Then there is a unique infinite-volume Potts measure on $\cT$. 
\end{corollary}

\subsection{Galton--Watson trees: volume and uniqueness}
As corollaries of our results on general trees, we can obtain exponential decay and uniqueness results for the random-cluster model on a Galton--Watson random tree. Let $\nu$ denote the progeny distribution for a Galton--Watson tree. 
For $\ell \ge 0$ let $Z_\ell$ be the number of vertices in $\ell$-th generation so that $Z_0 = 1$ and $Z_1 \sim \nu$.
Our first result provides a tail bound for $Z_\ell$ (under mild moment assumptions on $\nu$). 
This bound will 
allow us to argue that the Galton--Watson tree satisfies the volume assumptions of Lemma~\ref{lemma:exp:decay:wired:treelike}, and deduce uniqueness of the random-cluster measure on super-critical Galton--Watson trees when $p<p_u(q,\gamma)$.

\begin{lemma}
	\label{lemma:gw:leaves}
	Let $N \sim \nu$, $\kappa \ge 1$ and
	suppose $\E[N] \ge 1$ and $\E[N^\kappa] < \ew_\kappa$ for some constant $\ew_\kappa$. 
	If $\kappa$ is a power of $2$, there exists $C = C(\kappa,\E[N],\ew_\kappa)$ such that for every $\gamma > 0$ and every $1\le \ell \le h$,
	$$
	\Pr(Z_\ell \ge \gamma^h) \le C h^{2\kappa} \left(\frac{\E[N]}{\gamma}\right)^{\kappa h}.
	$$
\end{lemma}	
\begin{proof}
	Let $\E[N] = \ew$ and $W_\ell = Z_\ell/\ew^\ell$.
	From the definition of the Galton-Watson tree we have that $Z_{\ell+1} = \sum_{i=1}^{Z_\ell} N_{\ell,i}$ for all $\ell \ge 1$, where the $N_{\ell,i}$ are independent copies of $N \sim \nu$. Then,
	\begin{equation}
	\label{eq:gw:id}
	W_{\ell+1} - W_\ell = \frac{1}{\ew^\ell} \sum_{i=1}^{Z_\ell} (W_{1,i} - 1)\,,
	\end{equation}
	where the $W_{1,i}$'s are i.i.d.'s instances of $W_1 = N/\ew$. From~\eqref{eq:gw:id} we deduce that  
	$$
	\E[|W_{\ell+1} - W_\ell|^\kappa\mid Z_\ell = L] = \frac{1}{\ew^{\kappa\ell}} \E\left[\left|\sum_{i=1}^{L} (W_{1,i} - 1)\right|^\kappa\right].
	$$		
	Since the $W_{1,i}$'s are i.i.d.'s with $\E[W_{1,i} - 1] = 0$, and  
	$
	\E[|W_{1,i} -1|^\kappa] \le 2^\kappa(\frac{1}{\ew^\kappa} \E[N^\kappa] + 1)
	$
	is finite, 
	it follows from the Marcinkiewicz--Zygmund inequality that
	$$
	\E\left[ {\left| \sum_{i=1}^{L} (W_{1,i} - 1) \right|}^{\kappa} \right] \le B_\kappa \cdot \E\left[ {\left( \sum_{i=1}^{L} (W_{1,i} - 1)^2 \right)}^{\kappa/2}  \right],
	$$
	where $B_\kappa=(2\lceil \kappa/2 \rceil)^{1/2})^\kappa$ (see Section~10.3 in~\cite{chow2003probability}).
	For $\kappa > 2$, Jensen's inequality then provides the bound
	\begin{equation}
	\label{eq:holder:ineq}
	\E\left[ {\left( \sum_{i=1}^{L} (W_{1,i} - 1)^2 \right)}^{\kappa/2} \right]  \le  L^{\kappa/2-1} \sum_{i=1}^{L} \E[|W_{1,i} - 1|^\kappa]\,.
	\end{equation}
	Combining these inequalities and taking expectations we obtain
	$$
	\E[|W_{\ell+1} - W_\ell|^\kappa]  \le \frac{B_\kappa }{\ew^{\kappa\ell}} \E[Z_\ell^{\kappa/2}] \E[|W_{1} - 1|^\kappa] = \frac{B_\kappa }{\ew^{\kappa\ell/2}} \E[W_\ell^{\kappa/2}] \E[|W_{1} - 1|^\kappa]\,,
	$$
	and since $\E[|W_{1} - 1|^\kappa] \le 2^\kappa(\frac{1}{\ew^\kappa} \E[N^\kappa] + 1)$, for a suitable constant $C = C(\ew,\kappa,\ew_\kappa)$ we have
	\begin{equation}
	\label{eq:gw:mb}
	\E[|W_{\ell+1} - W_\ell|^\kappa]  \le \frac{C }{\ew^{\kappa\ell/2}} \E[W_\ell^{\kappa/2}]\,. 
	\end{equation}
	
	Now, let $Y_1 = W_1$ and for $i \ge 2$ let $Y_i = W_i - W_{i-1}$, so that $W_\ell = \sum_{i=1}^\ell Y_i$. Using the triangle and Jensen's inequalities (as in~\eqref{eq:holder:ineq}) we deduce that 
	$$
	\E[|W_\ell|^\kappa] \le \E\left[\left(\sum_{i=1}^\ell |Y_i| \right)^\kappa\right] \le \ell^{\kappa-1} \sum_{i=1}^\ell \E[|Y_i|^\kappa]\,,
	$$
	and from the bound in~\eqref{eq:gw:mb} we get
	$$
	\E[W_\ell^\kappa] \le C \ell^{\kappa-1} \sum_{i=1}^\ell \frac{\E[W_i^{\kappa/2}]}{\ew^{\kappa i/2}}.
	$$
	From this, letting $\rho(h,\kappa) = \max_{i \le h} \E[W_i^\kappa]$, we obtain the recurrence
	$$
	\rho(h,\kappa) \le C h^{\kappa-1} \rho(h,\kappa/2) \sum_{i=1}^h \frac{1}{\ew^{\kappa i/2}} \le C \cdot h^{\kappa} \cdot \rho(h,\kappa/2),
	$$
	since $\ew \ge 1$ by assumption.
	Since $\kappa$ is a power of two,
	iteratively, for a suitable constant $C_1= C_1(\kappa,\ew,\ew_\kappa)$,
	\begin{align}\label{eq:population-moment-bound}
	\rho(h,\kappa) \le  C_1 \cdot \rho(h,1) \cdot \prod_{i=0}^{\log_2 \kappa} h^{\kappa/2^i} \le C_1 \cdot h^{2\kappa}. 
    \end{align}
	Finally, we note that by Markov's inequality, for any $\ell \le h$
	$$
	\Pr(Z_\ell \ge \gamma^h) = \Pr\Big(W_\ell^\kappa \ge \Big(\frac{\gamma}{\ew}\Big)^{h\kappa}\Big)\le  C_1 h^{2\kappa}\Big(\frac{\ew}{\gamma}\Big)^{\kappa h},
	$$
	as claimed.
\end{proof}

We show next that Galton-Watson trees satisfy (with high probability) a certain growth condition that would allow us to apply the sharp decay of connectivities in random-cluster configurations from Lemma~\ref{lemma:exp:decay:wired:treelike}. We define the following volume growth condition for the random tree, which is stronger than the assumption of Lemma~\ref{lemma:exp:decay:wired:treelike}, and will also reappear later in the paper. 

\begin{definition}\label{def:gamma-eps-tree-growth}
	We say a tree $\cT_h = (V(\cT_h),E(\cT_h))$ satisfies the $(\gamma,\varepsilon)$-tree-growth condition if for every for every $v \in V(\T_h)$ with $h(v) > \varepsilon h$, and 
	every $k$ such that $\varepsilon h < k \le h(v)$, we have 
	$|V(\T_v(k))| \le \gamma^{k}$,
	where $\T_v(k)$ denotes the subtree of $\cT_h$ of height $k$ rooted at~$v$.
\end{definition}

\begin{corollary}\label{cor:random-tree-has-tree-growth}
	Let $N \sim \nu$ and $\kappa \ge 1$.
	Suppose $1 \le \E[N] < \gamma$ and that there exists a constant $\ew_\kappa$ such that $\E[N^\kappa] < \ew_\kappa$. 
	Then, if $\kappa$ is a  sufficiently large power of $2$, there exists a constant $\theta = \theta(\gamma,\kappa,\E[N],\ew_\kappa) \in (0,1)$ such that the Galton-Watson tree truncated at height $h$ with progeny distribution $\nu$ has $(\gamma,\epsilon)$-tree-growth with probability at least 
	$1 -  \theta^{\varepsilon\kappa h}$ for $h$ sufficiently large.
\end{corollary}

\begin{proof}
	Let $\{X_k^{(j)}\}_{j \ge 1}$ be i.i.d. random variables corresponding to the total number of vertices in a Galton--Watson tree of height $k$. By a union bound,
	the probability that the Galton--Watson tree does not satisfy the $(\gamma,\varepsilon)$-tree-growth condition is at most:
	\begin{align}
	&\sum_{l=0}^{(1-\varepsilon)h} \Pr\left( \bigcup_{j=1}^{Z_l} \bigcup_{\varepsilon h 
		\le k \le h - l} \{X_k^{(j)} \ge \gamma^k \} \right) \notag \\
	& \qquad \qquad\le 	(2\gamma)^h   \sum_{l=0}^{(1-\varepsilon)h} \sum_{\varepsilon h 
		\le k \le h - l}  \Pr\left(X_k^{(1)} \ge \gamma^k  \right) + 	\sum_{l=0}^{(1-\varepsilon)h} \Pr\left( Z_l \ge (2\gamma)^h\right)	\label{eq:vol:0}.
	\end{align}
	
	From Lemma~\ref{lemma:gw:leaves}, we know that there exists a constant $C = C(\kappa,\E[N],\ew_\kappa)$ such that 
	\begin{equation}
	\label{eq:vol:1}
	\sum_{l=0}^{(1-\varepsilon)h} \Pr\left( Z_l \ge (2\gamma)^h\right) \le \frac{C h^{2\kappa + 1}}{2^{\kappa h}}.
	\end{equation}
	Now, observe that $X_k^{(1)}$ has the same distribution as $\sum_{j=0}^k {Z}_j$. Hence, Lemma~\ref{lemma:gw:leaves} and a union bound imply that
	there exists $\hat{\gamma} \in (\E[N],\gamma)$ such that
	$$
	\Pr\left(X_k^{(1)} \ge \gamma^k  \right) = \Pr\left(\sum_{j=0}^k {Z}_j \ge \gamma^k\right) \le \sum_{j=0}^k \Pr(Z_j \ge \hat{\gamma}^k) \le C_1 k^{2r+1}\Big(\frac{\E[N]}{\hat\gamma}\Big)^{\kappa k}
	$$
	for a suitable constant $C_1 = C_1(\kappa,\E[N],\ew_\kappa) > 0$ and $k$ large enough. Then,
	\begin{align*}
	\sum_{l=0}^{(1-\varepsilon)h} \sum_{\varepsilon h 
		\le k \le h - l}  \Pr\left(X_k^{(1)} \ge \gamma^k  \right)  
	&\le C_1 \sum_{l=0}^{(1-\varepsilon)h} \sum_{\varepsilon h 
		\le k \le h - l}  	k^{2r+1}\Big(\frac{\E[N]}{\hat\gamma}\Big)^{rk} \\ 
	& \le C_2 h^{2\kappa+2} \Big(\frac{\E[N]}{\hat\gamma}\Big)^{\varepsilon \kappa h},
	\end{align*}	
	for a suitable constant $C_2 > 0$. Plugging this bound and~\eqref{eq:vol:1} into~\eqref{eq:vol:0}, we obtain
	that the probability that the Galton--Watson tree does not satisfy the $(\gamma,\varepsilon)$-tree-growth condition is at most
	$1-\theta^{\varepsilon \kappa h}$ for a suitable $\theta = \theta(\gamma,\kappa,\E[N],\ew_\kappa) \in (0,1)$
	as claimed.
\end{proof}

\subsubsection{Uniqueness in Galton-Watson trees}

Let $\T$ be a Galton-Watson tree with progeny distribution $\nu$ and
let $N \sim \nu$.
By Lemma~\ref{lemma:gw:leaves} and the Borel--Cantelli lemma, with probability one over $\cT$, for any $\gamma>\E[N]$, we have $Z_h\le \gamma^h$ for all sufficiently large $h$. In particular, with probability one, $Br(\cT)<\gamma$
for any $\gamma > \E[N]$. As such, Corollary~\ref{cor:rc:unique} implies that there is a unique random-cluster measure on $\T$ under the wired boundary condition when $p < p_u(q,\gamma)$. 

\begin{corollary}
    Fix $q \ge 1$, $\gamma > 1$ and $p < p_u(q,\gamma)$
    	Let $N \sim \nu$, $\kappa \ge 1$ and
	suppose $1 \le \E[N] < \gamma$ and that there exist a constant $\ew_\kappa$ such that $\E[N^\kappa] < \ew_\kappa$. With probability one over $\cT$, there is a unique random-cluster distribution on $\T$
	under the wired boundary condition. Similarly, at integer $q$, with probability one over $\cT$, there is a unique Potts distribution on $\cT$. 
\end{corollary}

\section{Random-graph estimates}\label{sec:random-graph-estimates}
In this section, we describe the standard revealing scheme for the configuration model with degree sequence $\vdn$. We also formalize the mechanism to translate probability $1-o(1)$ events for $\Pcm$ to $1-o(1)$ events for $\Prg$ and for the Erd\H{o}s--R\'enyi random graph model; we use this to provide a proof of Theorem~\ref{thm:intro:fk-er} given Theorem~\ref{thm:intro:general}. We then use the revealing scheme for the configuration model to prove the random graph estimates of Lemmas~\ref{lem:random-graph-treelike} and~\ref{lem:random-graph-volume-growth}.

\subsection{Configuration model with general degree sequence}
We begin by describing a revealing procedure for the configuration model with degree sequence $\vdn$. To do so, we begin with an important definition providing the \emph{state space} for our revealing procedures of the configuration model. Recall that a \emph{matching} on a graph is an edge-subset such that no vertex belongs to more than one edge. A \emph{perfect matching} is an edge-subset in which every vertex belongs to exactly one edge.

\begin{definition}\label{def:matchings-perfect-matchings}
Given a degree sequence $\vdn = (d_v)_{1\le v \le n}$, to each vertex $v\in \{1,...,n\}$, assign $d_v$ \emph{half-edges}. Consider an auxiliary complete graph $K_{{\|\vdn\|}_1}$ whose ${\|\vdn\|}_1$ vertices are identified with these half-edges. Let $\fM_\vdn$ be the set of all matchings (not necessarily perfect) on $K_{{\|\vdn\|}_1}$, and let $\fPM_n(\vdn)$ be the set of all perfect matchings on $K_{{\|\vdn\|}_1}$. 
\end{definition}

We are now in position to formally define the configuration model of random graphs. 

\begin{definition}\label{def:configuration-model}
Given a degree sequence $\vdn$, the configuration model $\Pcm$ is the uniform distribution over $\fPM_n(\vdn)$, i.e., it is a uniform perfect matching of the ${\|\vdn\|}_1$ half-edges assigned to the vertices $\{1,...,n\}$. This is naturally identified with a multigraph on $\{1,...,n\}$ by identifying all half-edges with the vertex they come from, so that the edges in the matching become edges of the graph between the corresponding vertices. In this manner, with a slight abuse of notation, elements $E\in \fPM_n(\vdn)$ are simply the edge-sets of the multigraph $\cG = (V,E)$. 
\end{definition}

 \begin{remark}\label{def:extend-random-cluster}
  The definitions of the random-cluster model~\eqref{eq:rcmeasure}, and the FK-dynamics extend naturally to multigraphs, where $G = (V,E)$ is such that $V$ is identified with $\{1,...,n\}$ and $E \in \fPM_n$ is a multiset. The random-cluster model and FK-dynamics then live over subsets of $E$, identified with~$\omega: E\to \{0,1\}$, and connected components of a configuration $\omega$ are understood naturally.
  \end{remark}

\subsection{Revealing procedure for the configuration model}
We now describe a simple revealing procedure for generating a sample from the configuration model distribution given fixed degree sequence $\vdn$.

\begin{process}\label{proc:configuration-model-revealing-graph}
Fix a degree sequence $\vdn$ with $\sum_{v} d_v$ even. Suppose $f$ is a (possibly random) function from matchings $A\in \fM_n$, to a half-edge not matched in $A$. 
\begin{enumerate}
\item Initialize the set $A_0= \emptyset$
\item For every $t\ge 0$, if $A_t \notin \fPM_n$ (i.e., there exist un-matched half-edges), construct $A_{t+1}$ as follows: 
\begin{enumerate}
    \item Let $\hat e_{t+1}$ be the half-edge $f(A_t)$
    \item Pick another un-matched half-edge in $A_t$ uniformly at random, and match it with $\hat e_{t+1}$ in $A_{t+1}$. 
\end{enumerate}
\end{enumerate}
\end{process}

For natural choices of the function $f$, we can reveal, for example, a ball in the random graph without revealing any information about the remainder of the random graph. The next definitions give an example of such an $f$ that we will use repeatedly.

\begin{definition}\label{def:BFS-revealing-exposed}
    Given a matching $A \in \fM_n$, the set of \emph{exposed} half-edges of $A$ is the set of un-matched half-edges that belong to the same vertex (among $V = \{1,...,n\}$) as some half-edge that is matched in $A$. Denote this set by $\widehat E(A)$. 
    \end{definition}
    
    \begin{process}\label{proc:BFS-revealing-process}
   
    The breadth-first exploration of a ball $B_r(v) \subset E(\cG)$ is constructed using Process~\ref{proc:configuration-model-revealing-graph} with the following choice of $f$.
    For each $A$, $f(A)$ is an arbitrarily chosen exposed half-edge among $\widehat E(A)$ whose distance in $(V,A)$ to $v$ is at most $r$.
    \end{process}

\subsection{Contiguity with simple random graphs}
The configuration model described above gives a uniform at random \emph{multigraph} with prescribed degree sequence $\vdn$. In the sparse regime of bounded average degree, this happens to be a very useful model for studying random \emph{simple} graphs (i.e., has no self-loops or multi-edges), most notably $\Delta$-regular random graphs, but also a uniformly chosen random simple graph with degree sequence $\vdn$ (as long as the sequence is graphical).

\subsubsection{General degree sequence}
It is well established that in the sparse regime of bounded average degree, the configuration model will have probability uniformly bounded away from zero of being simple, and on that event it is exactly a uniform simple graph with degree sequence $\vdn$. This contiguity can be summarized as follows (see e.g.,~\cite{janson_2009}). 
\begin{lemma}\label{lem:cm-rg-contiguity}
Fix any $\gamma$ and $\kappa$. Suppose $(\vdn)_n\in \cD_{\gamma,\kappa}$ and ${\|\vdn\|}_1 = \Omega(n)$. Then for any sequence of sets $A_n$ of simple graphs on $n$ vertices, we have 
\begin{align*}
    \Pcm(\cG \in A) = o(1) \qquad \mbox{if and only if}\qquad \Prg(\cG \in A) = o(1)\,.
\end{align*}
\end{lemma}

\subsubsection{Erd\H{o}s--R\'enyi random graph}
In the case of the Erd\H{o}s--R\'enyi random graph $G(n,d/n)$, the degree of a vertex $v$ is not fixed, but rather is distributed as $\mbox{Bin}(n-1, d/n)$. Nonetheless, there is a way to first randomly sample $\vdn$ then draw a configuration model on $\vdn$, such that the resulting random graph is contiguous to the Erd\H{o}s--R\'enyi distribution. Let $\mathbb P_{Poi(d)}$ be the distribution over $\vdn = (d_1,...,d_n)$ where $d_i$ are i.i.d.\ $\mbox{Poisson}(d)$ random variables. The following was established in~\cite{KimPoissonCloning}. 

\begin{lemma}\label{lem:poisson-cloning-contiguous}
For any $d = \Theta(1)$, for every sequence of sets $A_n$ of simple graphs on $n$ vertices, we have 
\begin{align*}
    \mathbb E_{Poi(d)} [\Pcm(\cG \in A)] = o(1) \qquad \mbox{if and only if}\qquad \mathbb{P}_{G(n,d/n)}(\cG\in A) = o(1)\,.
\end{align*}
\end{lemma}
In the above lemma, on the event that $\vdn$ does not have ${\|\vdn\|}_1$ even, as a matter of convention, we take the probability in the expectation to be zero. Overloading notation slightly, let $\mathbb P_{Poi(d)}$ be the product distribution over $\vdn\sim \mathbb P_{Poi(d)}$ for each $n$. 

\begin{lemma}\label{lem:Poisson-D-gamma-kappa}
For every $0<d<\gamma$ and every $\kappa\ge 1$,
\begin{align*}
    \mathbb P_{Poi(d)} \big((\vdn)_n\in \cD_{\gamma,\kappa}\big) =1\,.
\end{align*}
\end{lemma}

\begin{proof}
Recall by definition of $\Pvd,\Evd$, that 
\begin{align*}
    \Evd[D] = \frac{1}{{\|\vdn\|}_1} \sum_{v}d_v(d_v -1) =  \frac{\sum_v d_v^2}{{\|\vdn\|}_1} - 1\,.
\end{align*}
Let $0<\epsilon<\gamma - d$. Then for every $n$ large, we have 
\begin{align}\label{eq:Poisson-expectation}
    \mathbb P_{Poi(d)}(\Evd[D]<\gamma -\epsilon) & \le \mathbb P_{Poi(d)} \Big(\frac{1}{{\|\vdn\|}_1}\sum_v d_v^2 \ge \gamma - \epsilon + 1 \Big) \\
    & \le \mathbb P\big(\frac{1}{n}\sum_{v} d_v^2 > d(d+1) +  n^{- \frac 12 + \delta}\big) + \mathbb P\big(\frac{1}{n}\sum_v d_v < d -  n^{- \frac 12 +\delta}\big)\,. \nonumber
\end{align}
To bound either of these terms, notice by Markov's inequality, that \begin{align*}
    \mathbb P\Big(|\sum_{v} d_v^k - \mathbb E[\sum d_v^k]| > \lambda) \le \frac{\mathbb E[|\sum_{v} (d_v^k - \mathbb E[ d_v^k])|^l]}{\lambda^l}\,.
\end{align*}
The numerator on the right-hand side is a sum of i.i.d.\ mean-zero random variables, each of which have all finite moments. As such, for any fixed $l$, the right-hand side above is at most 
\begin{align*}
    Cn^{l/2}\lambda^{-l} \le C n^{-l\delta}\,.
\end{align*}
Taking $l>5\delta^{-1}$, the right-hand side above is $O(n^{-5})$. Therefore, the sum over $n$ of the probabilities of the left-hand side of~\eqref{eq:Poisson-expectation} is finite, and by Borel--Cantelli, with probability one, eventually almost surely, $\Evd[D]<\gamma - \epsilon$, so that $\limsup \Evd[D]<\gamma$. A similar argument yields the uniform boundedness of the $\kappa$'th moments $\Evd[D^\kappa]$ for any $\kappa$, yielding the desired and concluding the proof.
\end{proof}

Given Lemmas~\ref{lem:poisson-cloning-contiguous}--\ref{lem:Poisson-D-gamma-kappa}, our Theorem~\ref{thm:intro:fk-er} becomes a corollary of Theorem~\ref{thm:intro:general}.

\begin{proof}[\textbf{\emph{Proof of Theorem~\ref{thm:intro:fk-er} given Theorem~\ref{thm:intro:general}}}]
Fix $q\ge 1$, $\gamma >0$ and $p<p_u(q,\gamma)$. Suppose $\cG \sim G(n,\gamma/n)$. Fix a large constant $K$ and let $A$ be the set of simple graphs $\cG$ such that the mixing time of FK-dynamics on $\cG$ at parameters $p,q$ satisfies $K^{-1} \log n \le \tmix \le K \log n$. By Lemma~\ref{lem:poisson-cloning-contiguous}, it suffices to show that  
\begin{align*}
    \mathbb E_{Poi(\gamma)}[\Pcm(\cG\notin A)]  = o(1)\,.
\end{align*}
Considering this quantity, for any $\gamma'$,
\begin{align*}
    \limsup_n \mathbb E_{Poi(\gamma)}[\Pcm& (\cG\notin A)]  \\
     & \le \mathbb P_{Poi(\gamma)} \big((\vdn)_n \notin \cD_{\gamma',\kappa}\big) + \sup_{(\vdn)\in \cD_{\gamma',\kappa}} \limsup_n \Pcm(\cG\notin A)\,.
\end{align*}
The first term on the right-hand side is zero for all $\gamma'>\gamma$ and all $\kappa$ by Lemma~\ref{lem:Poisson-D-gamma-kappa}. By Theorem~\ref{thm:intro:general} and Lemma~\ref{lem:cm-rg-contiguity}, the second term is zero if $\gamma'>\gamma$ is such that $p<p_u(q,\gamma')$, and if $\kappa$ and $K$ are sufficiently large (depending on $p,q,\gamma'$). By continuity of $p_u(q,\gamma)$, if $p<p_u(q,\gamma)$, there also exists $\gamma'>\gamma$ such that $p<p_u(q,\gamma')$, concluding the proof. 
\end{proof}

\subsection{Local domination of the configuration model by random trees}
We now dominate balls of volume $o(n^{1/2})$ of the random graph $\cG \sim \Pcm$ by branching processes whose progeny are approximately given by $\Pvd$. To be more precise, we define the following. 

\begin{definition}\label{def:truncated-empirical-distribution}
     Define the \emph{truncated empirical distribution} by letting $\underline\vdn = \vdn \setminus A_{\vdn}$, where $A_\vdn$ are the smallest $2\sqrt n$ elements of $\vdn$, and the set subtraction is done in the multi-set sense. Then let $\mathbb P_{\underline \vdn}$ be the corresponding effective offspring distribution of $\underline \vdn$, i.e., for $x\in \{d_v -1: d_v \in \underline \vdn\}$, 
     \begin{align*}
         \mathbb P_{\underline \vdn}(x) = \frac{\sum_{v:  d_v \in \underline \vdn}(x+1) \mathbf{1}_{\{d_v= x+1 \}}}{{\|\underline \vdn\|}_1}\,.
     \end{align*}
     Let $\underline D\sim \mathbb P_{\underline\vdn}$, and let $\mathbb E_{\underline \vdn}$ be the corresponding expectation.
\end{definition}

\begin{lemma}\label{lem:moments-sequence-to-truncated-sequence}
    If $(\vdn)_n \in \cD_{\gamma,\kappa}$, then $(\underline \vdn)_n\in \cD_{\gamma,\kappa}$.
\end{lemma}

\begin{proof}
Let $A_\vdn$ be the set of $2n^{1/2}$ smallest degrees of $\vdn$. We first of all claim that ${\|\vdn\|}_1 \le (1+o(1)){\|\underline \vdn\|}_1$. 
Indeed this follows from the calculation 
\begin{align*}
    \frac{{\|\vdn\|}_1 - {\|\underline \vdn\|}_1}{{\|\underline \vdn\|}_1} \le \frac{\sum_{v: d_v\in A_\vdn}d_v}{\sum_{v: d_v \notin A_{\vdn}}d_v}\le \frac{2\sqrt{n} \max \{d_v: d_v\in A_{\vdn}\}}{(n-2\sqrt{n})\min\{d_v: d_v\notin A_{\vdn}\}}\le O(n^{-1/2})\,. 
\end{align*}
We then can observe that 
\begin{align*}
   \mathbb E_{\underline \vdn} [\underline D] =  \frac{{\|\vdn\|}_1}{{\|\underline \vdn\|}_1}\frac 1{{\|\vdn\|}_1}\sum_{v: d_v \notin A_{\vdn}} d_v(d_v -1) \le (1+o(1))\mathbb E_{\vdn}[D]\,.
\end{align*}
We now wish to prove the desired moment conditions. Those follow by analogous reasoning: 
\begin{align*}
    \mathbb E_{\underline \vdn} [D^k] & \le \frac{{\|\vdn\|}_1}{{\|\underline \vdn\|}_1}\frac 1{{\|\vdn\|}_1} \sum_{v: d_v \notin A_{\vdn}} d_v(d_v -1)^k \le (1+o(1)) \Evd[D^k]\,. 
\end{align*}
Altogether, these give the desired implications of the lemma. 
\end{proof}

We now wish to show that the balls of the random graph $\cG \sim \Pcm$ are stochastically dominated by random trees with offspring distribution $\mathbb P_{\underline \vdn}$, even conditionally on an already revealed portion $H\in \mathfrak M_n$ of the random graph. However, this will only hold if $|H|\le n^{1/2}$ and the ball does not intersect $H$. We now formalize this notion. 

\begin{process}\label{proc:B-r-out}
     For a subgraph $H = (V(H),E(H))$, let $\widehat E(H)$ be the set of half-edges incident to $H$ but not matched in $H$. (Notice that this definition aligns with the use of $\widehat E(A)$ for the exposed half-edges of $A\in \mathfrak{M}_n$ when taking $H=(V(A),A)$.) For a half-edge $\hat e$ in $\widehat E(H)$, define $B_r(\hat e; H^c)$ as the ball of radius $r$ ``out of $H$". More formally, $B_r(\hat e; H^c)$ is obtained by 
     \begin{enumerate}
         \item Matching $\hat e$ to a vertex $w$. 
         \item Running the breadth-first revealing of $B_{r-1}(w)$ from Process~\ref{proc:BFS-revealing-process} but where $f(A)$ cannot be in $\widehat E(H)$ (i.e., it will be an arbitrarily chosen half-edge of $\widehat E(A)\setminus \widehat E(H)$ at distance at most $r$ from $w$ in $A$). 
     \end{enumerate}
\end{process}

Due to the extra edge from matching $\hat e$, let us say a \emph{single-source} Galton--Watson tree is a Galton--Watson tree whose first generation deterministically has exactly one child. 

\begin{proposition}\label{prop:ball-domination-by-GW-tree}
Consider any degree sequence $\vdn$. Let $\underline \vdn$ be as per Definition~\ref{def:truncated-empirical-distribution}. Let $\widehat {\cT}_h(\underline \vdn)$ be a single-source Galton--Watson tree of at depth $h$ (meaning it is truncated at depth $h$) with offspring distribution $\mathbb{P}_{\underline \vdn}$. Fix an arbitrary $H = (V(H),E(H))\in \mathfrak{M}_n$, and consider $\hat e\in \widehat E(H)$. Then, conditionally on $\{E(H) \subset E(\cG)\}$, we have the stochastic domination 
\begin{align*}
    |B_r(\hat e; H^c)|\mathbf 1_{\{|E(H) \cup E(B_r(\hat e; H^c))| \le n^{1/2}\}} \preceq |\widehat \cT_r(\underline \vdn)|\,,
\end{align*}
On the event that $B_r(\hat e;H^c)$ is a tree, there is an isometry between the graphs such that $B_r(\hat e; H^c)$ is a subset of $\widehat \cT_{r}(\underline \vdn)$. 
\end{proposition}

\begin{proof}
We appeal to the revealing procedure of Process~\ref{proc:configuration-model-revealing-graph} with the choice of breadth-first revealing described in Processes~\ref{proc:BFS-revealing-process} and~\ref{proc:B-r-out}. Begin the single-source Galton--Watson tree with a root vertex and a single child, corresponding to $\hat e$. Iteratively, when a half-edge $\hat f$, corresponding to a vertex $x$ in the single-source Galton--Watson tree, gets matched in the revealing procedure to a vertex $w$, 
\begin{enumerate}
    \item If $w$ had not been exposed yet, identify the other $d_w-1$ half-edges of $w$ with the children of $x$ in the single-source Galton--Watson tree. 
    \item If $w$ is an exposed vertex, do nothing. 
\end{enumerate}
(We say a vertex is exposed if one of its half-edges has already been matched, whether in $H$ or in the revealing.)
Uniformly over any subset of at most $2n^{1/2}$ matched half-edges (forming $n^{1/2}$ edges), the distribution $d_w-1$  is easily seen to be stochastically below $\mathbb P_{\underline \vdn}$ (in which the \emph{smallest} $2n^{1/2}$ half-edges have been removed). Notice then that on the indicator
\begin{align*}
    \mathbf 1_{\{|E(H) \cup E(B_r(\hat e; H^c))| \le n^{1/2}\}}\,,
\end{align*}
throughout the breadth-first revealing process, the number of matched half-edges will always be at most $2n^{1/2}$. Thus, we see that this process maintains the desired stochastic domination relation as compared to the single-source Galton--Watson tree until the number of matched half-edges exceeds $n^{1/2}$. 

When $B_r(\hat e; H^c)$ is a tree, item (2) above never happens, and the isometry goes by identifying the edge containing $\hat f$ in $E(\cG)$ with the edge connecting the corresponding vertex in $\widehat \cT_r(\underline \vdn)$ to its parent.  
\end{proof}

With Proposition~\ref{prop:ball-domination-by-GW-tree}, we can translate the volume growth bounds of Lemma~\ref{lemma:gw:leaves} into the desired volume growth estimate for the random graph $\cG \sim \Pcm$. In this proof, and other proofs relying on the random graph revealing procedure, it will be useful to have an $\ell^\infty$ bound on the degrees. For this, note that $(\vdn)_n\in \cD_{\gamma,\kappa}$ implicitly places a constraint on ${\|\vdn\|}_\infty$, since $\Pvd$ chooses ${\|\vdn\|}_\infty -1$ with probability $\Omega({\|\vdn\|}_\infty/n)$. More precisely, we have the following.

\begin{fact}\label{fact:L-infty-degree-sequence}
If $(\vdn)_n \in \cD_{\gamma,\kappa}$, then
    ${\|\vdn\|}_\infty \le n^{\epsilon_*}$ for $\epsilon_*(\kappa) = 2/(\kappa+1)$.
\end{fact}

\begin{proof}[\textbf{\emph{Proof of Lemma~\ref{lem:random-graph-volume-growth}}}]
We will take a union bound over the probabilities that for a fixed vertex $v\in \{1,...,n\}$, and a fixed $r \ge \epsilon \log_\gamma n$, the graph $\cG$ has $|B_r(v)|\le C\gamma^r$. Fix any such $v,r$ and take $H = (\{v\},\emptyset)$, so that $\hat E(H)$ are exactly the half-edges of $v$. Evidently, for $\hat e\in \hat E(H)$
\begin{align*}
    \mathbb P(|B_{r}(v)| \ge \gamma^{r}) \le  d_v \mathbb P(|B_{r-1}(\hat e)|\ge d_v^{-1} \gamma^r)\,.
\end{align*}
Consider the probability on the right. For each $\hat e\in \hat E(H)$, by Proposition~\ref{prop:ball-domination-by-GW-tree},
\begin{align*}
    |B_{r-1}(\hat e)|\mathbf 1_{\{|B_{r-1}(\hat e)|\le n^{1/2}\}}\preceq Z_r\,,
\end{align*}
where $Z_{r} \sim |\widehat \cT_{r-1}(\underline \vdn)|$, where we recall this is the single-source Galton--Watson tree of depth $r-1$ whose offspring distribution is $\mathbb P_{\underline \vdn}$. 
Now using a union bound,  
\begin{align*}
    \Pcm\big(\cG \mbox{ does not have $(\gamma,\epsilon)$-volume growth}\big) \le \sum_{v} d_v \sum_{r= \epsilon \log_\gamma n}^{\frac 12 \log_\gamma n}  \mathbb P\big(|Z_r|\ge d_v^{-1} \gamma^r\big)\,.
\end{align*}
Since $(\vdn)_n\in \cD_{\gamma,\kappa}$, there exists $\eta>0$ such that it also is in $\cD_{\gamma - \eta,\kappa}$. Fix such an $\eta$. 

Let $\kappa(\gamma,\eta,\epsilon)$ be large, to be chosen later, and let $\epsilon_*(\kappa)$ be such that ${\|\vdn\|}_\infty \le n^{\epsilon_*}$ per Fact~\ref{fact:L-infty-degree-sequence}. Then the right-hand side is at most 
\begin{align*}
    n^{1+\epsilon_*} \sum_{r = \epsilon \log n}^{\frac 12 \log_\gamma n} \mathbb P\big(|Z_{r}| \ge  n^{-\epsilon_*}\gamma^{r}\big) =  n^{1+\epsilon_*}\sum_{r = \epsilon \log n}^{\frac 12 \log_\gamma n} \mathbb P\big(|Z_{r}| \ge   (\gamma^{1-\epsilon_*\log \gamma/\epsilon})^{r}\big)\,.
\end{align*}
Let $\tilde \gamma = \gamma^{1-\epsilon_* \log \gamma/\epsilon}$ and take $\kappa$ to be sufficiently large (so that $\epsilon_*$ is sufficiently small) that $\tilde \gamma >\gamma - \eta/2$. 
By Lemma~\ref{lemma:gw:leaves}, then, the right-hand side above  is at most 
\begin{align*}
    Cn^{1+\epsilon_*}\sum_{r= \epsilon \log n}^{\frac 12 \log_\gamma n} r^{\kappa}\Big(\frac{\gamma-\eta}{\gamma- \eta/2}\Big)^{r \kappa}\,.
\end{align*}
One then sees that if $\kappa$ is large enough, the right-hand side will be $o(n^{-10})$ as desired. 
\end{proof}

\subsection{Treelike nature of the configuration model}
We can also use the breadth-first revealing procedures together with the volume growth estimates, to establish that the random graph given by the configuration model is typically $(L,R)$-$\treelike$ for $L = O(1)$ and $R  \le \frac{1}{2}\log_\gamma n$. 

\begin{proof}[\textbf{\emph{Proof of Lemma~\ref{lem:random-graph-treelike}}}]
By Lemma~\ref{lem:random-graph-volume-growth}, with probability $1-o(n^{-10})$ the random graph $\cG$ has $(\gamma,\epsilon)$ volume growth, say for $\epsilon = 1/4$, as long as $\kappa$ is sufficiently large. Let us work on that event, so that $|B_{R}(v)|\le n^{1/2 - \delta}$ for all $v\in \{1,...,n\}$.  

Now fix any $v$ and perform the breadth-first revealing of $B_R(v)$ per Process~\ref{proc:BFS-revealing-process}. In order for $B_R(v)$ to not be $L$-$\treelike$, it must be the case that for more than $L$ different steps $m$ in the first $n^{1/2 - \delta}$ steps, the half-edge $f(A_{m-1})$ is being matched to a half-edge of $\widehat E(A_{m-1})$. Call such a step \emph{bad}. (If there were at most $L$ bad steps, then the removal of the at-most $L$ edges formed by those at-most $L$ matchings in the revealing scheme, evidently leaves a tree.)

Uniformly over $A_{m-1}$, the probability of the $m$'th step being bad is at most $$(m{\|\vdn\|}_\infty)/({\|\vdn\|}_1 - m)\,.$$
We thus find that for every $\ell\ge 1$, 
\begin{align}\label{eq:B-r-v-not-treelike}
    \Pcm ( B_R(v) \mbox{ is not $\ell$-$\treelike$}) \le \bbP\Big(\bin\Big(n^{1/2 - \delta}, \frac{n^{1/2- \delta}{\|\vdn\|}_\infty}{{\|\vdn\|}_1 - n^{1/2-\delta}}\Big) >\ell \Big)\,.
\end{align}
Recall that the standard Chernoff bound applied to a binomial distribution with mean $\mu = N p$ says that for every $s \ge \mu$, 
\begin{align}\label{eq:Chernoff-Poisson-binomial}
    \mathbb P\big(\bin(N,p) \ge s\big) \le e^{ s - \mu} \Big(\frac{s}{\mu}\Big)^{-s}\,.
\end{align}
Using the assumption that ${\|\vdn\|}_1= \Omega(n)$ and recalling from Fact~\ref{fact:L-infty-degree-sequence} that ${\|\vdn\|}_\infty\le n^{\epsilon_*(\kappa)}$,~\eqref{eq:Chernoff-Poisson-binomial} implies that the right-hand side of~\eqref{eq:B-r-v-not-treelike} is at most $(C n^{-2\delta + \epsilon_*})^\ell$. As a consequence, taking $\kappa$ large enough that $\epsilon_*<\delta$, and choosing $L > 11\delta^{-1}$, we would find that the probability of $B_R(v)$ not being $L$-$\treelike$ is $o(n^{-11})$ for all $v$, and a union bound over $v\in \{1,...,n\}$ implies the desired result. 
\end{proof}

\section{The FK-dynamics shatters quickly on random graphs}\label{sec:shattering}
Our first goal in this section
is to prove the following theorem establishing the existence of $T = O(1)$ (in continuous-time) such that for $t\ge T$, the FK-dynamics chain on the random graph $\cG$ initialized from the all-wired configuration (i.e., all edges are open), denoted $X_{\cG,t}^1$, is shattered, i.e., all the connected components of the FK-dynamics configuration are small; recall Definition~\ref{def:(k,r)-sparse} for a precise formulation.

\begin{theorem}\label{thm:pi-exponential-decay}
	Fix $q \ge 1$, $\gamma>1$, and $p<p_u(q,\gamma)$.
	For every $\epsilon>0$, there exists $\kappa$ such that if $(\vdn)_n\in \cD_{\gamma,\kappa}$, the following holds. There exists $T = O(1)$ such that for every $t\ge T$ and every $v$, with probability $1-o(n^{-10})$, $\cG\sim \Pcm$ is such that 
	\begin{align*}
	\bbP\big(|\cC_v(X_{\cG,t}^1)|\ge n^{\epsilon} \big) \le o(n^{-10})\,. 
	\end{align*}
\end{theorem}

We will then use this to conclude Theorem~\ref{thm:k-R-sparse-whp}, demonstrating that if $t\ge T$, the boundary condition $X_{\cG,t}^1$ induces on any ball of volume $o(\sqrt n)$ is $O(1)$-sparse.

By monotonicity of the FK-dynamics, for every $\cG$, we have that $X_{\cG,t}^1\succeq \pi_{\cG}$, from which it follows that Theorem~\ref{thm:pi-exponential-decay} holds under $\pi_{\cG}$. 

\begin{corollary}\label{cor:pi-exponential-decay}
Fix $q \ge 1$, $\gamma>1$, and $p<p_u(q,\gamma)$. For every $\epsilon>0$, there exists $\kappa$ such that if $(\vdn)_n\in \cD_{\gamma,\kappa}$, then for every $v$, with probability $1-o(n^{-10})$, $\cG \sim \Pcm$ is such that
	\begin{align*}
	\pi_{\cG}\big(|\cC_v(X_{\cG,t}^1)|\ge n^{\epsilon} \big) \le o(n^{-10})\,. 
	\end{align*}
\end{corollary}

While we do not use this corollary here, it may find applications elsewhere.

\subsection{Couplings and revealing schemes for the FK-dynamics on random graphs}
In this section, we define our central revealing procedure for exposing the random graph together with a family of coupled FK-dynamics on subsets of the random graph $\cG$, which together stochastically dominate $X_{\cG,t}^1$. This revealing procedure is essential to the proof of shattering for $X_{\cG,t}^1$ in the uniqueness region after $O(1)$ continuous-time. 

A similar revealing scheme of random graphs with an FK-dynamics chain on top of it was introduced in~\cite{BlGh21}. The revealing scheme we use here builds on that, but makes some key modifications to deal with the non-uniformity of the degrees and the lack of  deterministic control on the volume of small balls of $\cG$. These changes are explicitly laid out in Remark~\ref{rem:revealing-process-modifications}. 

\subsubsection{Grand coupling of localized FK-dynamics}
In this section, we define a grand coupling of FK-dynamics on all possible edge subsets of the random graph $\cG$ in such a way that 
all monotonicities of the model are maintained. 

Recall from Definition~\ref{def:matchings-perfect-matchings} that we use $\fM_n$ as the set of all (not necessarily perfect)  matchings of the complete graph on the ${\|\vdn\|}_1$ many half-edges. The matching $A$ is naturally identified with a set of edges on the original vertex set $\{1,...,n\}$, each pairing of two half-edges becoming an edge between the vertices they belong to. Abusing notation, we will understand $A$ both as a matching element of $\fM_n$ and as a multiset of edges on $\{1,...,n\}$. 

\begin{definition}\label{def:censored-Glauber-chains}
    For an element $A\in \fM_n$, let $\partial A$ be the set of vertices in $\{1,...,n\}$ having half-edges that are not matched in $A$. Let $\pi_{A}^1$ be the random-cluster measure on the edge set $A$ with wired boundary conditions wiring all vertices of $\partial A$. Let $(X_{A,t}^1)_{t\ge 0}$ be the continuous-time FK-dynamics initialized from the all wired configuration on $A$ (as well as outside $A$), and making updates in $A$ according to $\pi_{A}^1$. 
\end{definition}

We next place all the chains $(\mathcal X_{t}^1)_{t\ge 0}= \big\{(X_{A,t}^1)_{t \ge 0}\big\}_{A\in \fM_n}$ on all possible matchings $A\in \fM_n$, in the same probability space, and construct an explicit coupling of them. 

\begin{definition}\label{def:sources-of-randomness}
The probability space we consider will consist of the following sources of randomness:
\begin{enumerate}
\item Independently assign each possible edge $e$ (i.e., each possible pairing of two half-edges), a sequence of times $\mathfrak T_e = (T_1^{e}, T_2^{e},...)$ given by the rings of a rate-1 Poisson clock; and
\item Independently assign each possible $e$ a sequence of $\mbox{Unif}[0,1]$ random-variables $\mathfrak U_e = (U_{1}^{e},U_2^e,...)$. 
\end{enumerate}
We denote by $\cF_t$ the $\sigma$-algebra generated by the processes $(\mathfrak T_e)_e$ up to time $t$, as well as the corresponding set of random variables in $(\mathfrak U_e)_e$. 
\end{definition}

\begin{definition}\label{def:coupling-Glauber-chains}
From $(\mathfrak T, \mathfrak U)$ construct the processes  $(X_{A,t}^{1})_{t\ge 1}$ for all $A\in \fM_n$ as follows: 
\begin{enumerate}
    \item Let $0<t_1< t_2<...$ be the (almost surely distinct) times in $\bigcup_{i}\bigcup_{e}\{T_i^e\}$ in increasing order. 
    \item Initialize $X_{A,0}^1(e) =1$ for all $e$; i.e., the all wired configuration. 
    \item For each $i\ge 1$, let $$X_{A,t}^{1} = X_{A,t_{i-1}}^1\qquad \mbox{for all}\qquad t\in [t_{i-1},t_i)\,.$$
    Then, let $(e_i,k_i)$ be the unique pair for which $t_i = T_{k_i}^{e_i}$ 
and define $X_{A,t_i}^1$ by setting 
$$X_{A,t_i}^1(e) = X_{A,t_{i-1}}^1(e) \qquad \mbox{for all} \qquad e\in A \setminus \{e_i\}$$
and
\[X_{A,t_i}^{1}(e_{i})= 
\begin{cases}
1 & \mbox{if }e_{i}\in A~\mbox{and } U_{e_i,k_i} \le \varrho; \\
0 & \mbox{if }e_{i}\in A~\mbox{and }U_{e_t,k_i} > \varrho;
\end{cases}
\]
for $$\varrho = \pi_{A}^{1}\big(\omega(e_{i})=1\mid\omega(A\setminus \{e_i\})= X_{A,t_{i-1}}^1(A\setminus \{e_i\})\big)\,;$$ 
i.e., if $e_{i}\in A$, we resample $e_{i}$ given the
remainder of the configuration on $A$, together with the wired
boundary condition on $\partial A$, using the same $U_{e_i,k_i}$ for every $X_{A,t_i}^{1}$ such that $e_i\in A$.
\end{enumerate}
\end{definition}

The following two observations are elementary to observe, but of central importance to our analysis. 

\begin{observation}\label{obs:monotonicity-of-coupling}
The coupling defined in Definition~\ref{def:coupling-Glauber-chains} is a monotone coupling. In particular, we have $X_{A',t}^1\le X_{A,t}^1$ for any two matchings $A,A'\in \fM_n$ with $A\subset A'$. As such, we have for every $\cG$ that $$X_{\cG,t}^1(e) \le \min_{A\in \fM_n: A\subset E(\cG)} X_{A,t}^1(e)\,, \qquad \mbox{for all $e\in E(\mathcal G)$ and all $t\ge 0$}\,.$$
\end{observation}

\begin{observation}\label{obs:measurability-of-coupling}
For every $A$, the configuration $X_{A,t}^{1}$ depends only on $(\mathfrak T_e, \mathfrak U_e)_{e\in A}$, and in fact only on their restriction to $\cF_t$ (the $\sigma$-algebra generated by elements of $\mathfrak T,\mathfrak U$ before time $t$). 
\end{observation}

We now use the coupling of Definition~\ref{def:coupling-Glauber-chains} to design a coupling of FK-dynamics chains on random graphs.  

\begin{definition}\label{def:coupling-Glauber-chains-average-graph}
Let $\mathbb{P}_{t}^{1}$ be the distribution over pairs $(\mathcal{G},\omega_t)$
where $\omega_{t}$ is a random-cluster configuration on $\mathcal{G}$ that results by first
drawing $\mathcal{G}\sim \Pcm$, then drawing $\omega_{t}\sim \mathbb P(X_{\cG,t}^{1}\in\cdot)$.
    Likewise, for every set $A\in \fM_n$, let $\mathbb{P}_{A,t}^{1}$
be the distribution over pairs $(\mathcal{G},\omega_{A\cap E({\mathcal{G}}),t})$
where $\omega_{A,t}\sim \mathbb P(X_{A,t}^{1}\in\cdot)$.
Couple, under the distribution $\bbP$, the family of distributions $(\mathbb{P}_{A,t}^{1})_{A\in \fM_n,t\ge 0}$
by selecting the same random graph $\mathcal{G}\sim \Pcm$
for all of them, then using the coupling of Definition~\ref{def:coupling-Glauber-chains} for the family of FK-dynamics $(X_{A,t}^{1})_{A\subset E(\cG),t\ge 0}$ on $\cG$. 
\end{definition}

In this manner, we have constructed a monotone coupling of 
$(\mathcal{G},(X_{A,t}^{1})_{t\ge1})_{A\subset E(\cG)}$. Note that we use this coupling for $A$ which we know have $A \subset E(\cG)$, so that the randomness of the graph is only over the edges of $E(\cG)\setminus A$, which we note $X_{A,t}^1$ is independent of; thus the role of this coupling is only to put the random graphs with their random-cluster configurations on the same probability space. We emphasize that by Observation~\ref{obs:measurability-of-coupling}, if $A \cap B = \emptyset$, then $X_{A,t}^1$ and $X_{B,t}^1$ are independent.

\subsubsection{The joint revealing procedure} We now construct a revealing procedure for $\cG$ and a configuration $\tilde \omega_t$ on $\cG$ that stochastically dominates $X_{\cG,t}^{1}$. This will be inspired by the simultaneous revealing procedure first introduced in~\cite{BlGh21}, with significant modifications that streamline that argument, and deal with the heterogeneity of degrees and volumes of balls in $\cG \sim \Pcm$. 

\begin{definition}
Given a degree sequence $\vdn$, a vertex set $\cV \subset \{1,\dots,n\}$, and a matching $\cA\in \fM_n$, let $\widehat E(\cV,\cA)$ be the set of half-edges incident to $\cV$, and not matched in $\cA$.   
\end{definition}
We note that $\widehat E(V(\cA),\cA) = \widehat E(\cA)$ from Definition~\ref{def:BFS-revealing-exposed}.

For a matching on half-edges, $\cA_0 \in \fM_n$, so that $\cA_0 \subset E(\cG)$, and a subset of vertices $\cV_0\subset V(\cG)$, we construct a procedure to expose (a set containing) the connected components $\cC_{\cV_0}(X_{\cG,t}^1(E(\cG)\setminus \cA_0))$, i.e., the union of all the connected components of the vertices in $\cV_0$ in the configuration $X_{\cG,t}^{1}(E(\cG)\setminus \cA_0)$. The two examples to have in mind are 
\begin{enumerate}
    \item $\cA_0 = \emptyset$ and $\cV_0 = \{v\}$, used to prove Theorem~\ref{thm:pi-exponential-decay};
    \item $\cA_0= E(B_R(v))$, and $\cV_0 = \partial B_R(v)$, used to prove Theorem~\ref{thm:k-R-sparse-whp}.
\end{enumerate}  

In this revealing procedure, the index $m$ will count the number of ``steps", and $k$ will track the number of ``generations". We will keep track of the following variables through our revealing process:
\begin{itemize}
    \item $\mathcal{A}_{m}$: the element of $\fM_n$
          that has been shown to be a subset of $E(\cG)$ through step~$m$;
     \item $\tilde \omega_{m}$: the random-cluster configuration revealed through step $m$;
    \item $\widehat \cE_k$: the set of half-edges we want to explore out of in the $k$-th generation.
\end{itemize}
For $\cA\in \fM_n$, recall from Process~\ref{proc:B-r-out} that $B_r(\hat e; \cA^c)$ is revealed in a breadth-first manner, with the breadth-first exploration rejecting branches through vertices in~$V(\cA)$. 

The revealing process, with parameters $(p,q,\gamma,r,t)$, and input $(\cV_0,\cA_0)$ is defined as follows: see Figure~\ref{fig:revealing-fig1}--\ref{fig:revealing-fig2} for a depiction to accompany the below.

\begin{center}
	\fbox{
		\parbox{0.98\textwidth}{
			\begin{process}\label{proc:revealing-FK-on-graph} 
            
            \textbf{ }
            
			\textbf{Inputs:} $(p,q,\gamma)$; \quad $t\ge 0$; \quad $r\ge 1$; \quad $\cV_0\subset \{1,...,n\}$; \quad $\cA_0\in \fM_n$; 
			
			\textbf{Initialize:} $k = 0$;
            \quad $m = 1$;
            \quad $\widehat \cE_0 = \widehat E (\cV_0,\cA_0)$; 
            \quad $\tilde \omega_0 = \emptyset$;

            \medskip
            \textbf{for each}~$k \ge 0$~\textbf{while}~$\widehat{\mathcal{E}}_k \neq \emptyset$
            
           \quad \textbf{for each}~$\hat e \in \widehat \cE_k$,
\begin{enumerate}[\quad\quad\quad~1.]\setlength{\itemsep}{3mm}

\item Reveal the ball of radius $r$ out from $\hat e$ in the random graph $\cG$:
\smallskip
\begin{enumerate}
    \item Set $\hat e_m = \hat e$. Conditionally on $\cA_{m-1}$, reveal $B_{r}(\hat e_m; \cA_{m-1}^c)$ per Process~\ref{proc:B-r-out}. 
    \item Set $\cA_m = \cA_{m-1}\cup B_r(\hat e_m ; \cA_{m-1}^c)$. 
    \item Let $A_m := \mathcal A_m \setminus \mathcal A_{m-1}$ be the set of new edges revealed to belong to $E(\cG)$. 
\end{enumerate}

\smallskip
\item Simulate the FK-dynamics up to time $t$ on the newly revealed edge set $A_m$:
\smallskip
\begin{enumerate}
\item Reveal $\mathfrak F_{A_m,t} := \{(\mathfrak T^e)_{e\in A_m},(\mathfrak U^e)_{e\in A_m}\}\cap \cF_t$ (as defined in Definition~\ref{def:sources-of-randomness}).
\item Generate $X_{A_m,t}^{1}$
from $\mathfrak F_{A_m,t}$ per Definition~\ref{def:coupling-Glauber-chains}.
\end{enumerate}

\smallskip
\item Update the configuration $\tilde \omega_m$, the boundary half-edges $\widehat \cE_k$, and the step count~$m$: 
\smallskip
\begin{enumerate}
\item Concatenate $X_{A_m,t}^{1}$ with $\tilde \omega_{m-1}$ to obtain a new configuration $\tilde \omega_{m}$ on~$\cA_m$. 
\item Add to $\widehat \cE_{k+1}$ all un-matched half-edges of vertices in $\partial \cA_m$ that are in the component of $\cV_0$ in $\tilde \omega_m(\cA_m)$ and are not in $\widehat \cE_{j}$ for any $j\le k$. 
\item Increase $m$ by $1$.
\end{enumerate}
\end{enumerate}
\end{process} 
	}}
\end{center}

\begin{remark}\label{rem:revealing-process-modifications}
Before proceeding, let us describe the specific differences between the current revealing scheme and that of~\cite{BlGh21}, as well as why these changes are needed to overcome difficulties arising from heterogeneity of the underlying degree sequence. The main changes are as follows:  
\begin{enumerate}
    \item The revealing process is based on half-edges rather than vertices: this ensures that the revealing of the ball $B_r(\hat e_m; \cA_{m-1}^c)$ does not reveal the degrees of the vertices from which the exploration proceeds (which could potentially have high-degree and introduce correlations between generations). 
    \item The revealing of the ball $B_r(\hat e_m;\cA_{m-1}^c)$ does not continue exploring if it encounters any vertex of $V(\cA_{m-1}^c)$. This is important because if $B_r(\hat e_m)$ intersects a dense region of $\cA_{m-1}^c$ that has already been revealed, then the volume of $B_r(\hat e_m)$ would not be independent of $\cA_{m-1}^c$.  
    \item The FK-dynamics is simulated in continuous time, rather than discrete time. This introduces additional independence so that the number of updates taken by each of the localized FK-dynamics chains $X_{A_m,t}^1$ are truly independent of one another. 
\end{enumerate}
\end{remark}

For ease of notation, let $k_{\emptyset}$
be the first $k$ such that $\widehat {\mathcal{E}}_{k}=\emptyset$, i.e., the total number of generations of the revealing procedure. Let 
\begin{align*}
    \fm_k = \sum_{0\le j\le k}|\widehat \cE_j|\,,
\end{align*}
be the total number of half-edges for which $B_r(\hat e_m; \mathcal {A}_{m-1}^c)$ was revealed in step 1.a) of Process~\ref{proc:revealing-FK-on-graph}, so that $\mathfrak m_{k_{\emptyset}}$ counts the total number of half-edges out of which a ball is ever revealed. Let $$\tilde \omega = \tilde \omega_{\fm_{k_\emptyset}}(\cA_{\mathfrak m_{\kappa_\emptyset}} \setminus\cA_0)$$ be the random-cluster configuration revealed when the process terminates. 
The following key observation is a direct consequence of Observation~\ref{obs:monotonicity-of-coupling} and the construction of Process~\ref{proc:revealing-FK-on-graph}. 

\begin{observation}\label{obs:key-observation}
Under the procedure of Process~\ref{proc:revealing-FK-on-graph}, we have $$\tilde \omega(\cA_{\fm_{k_\emptyset}}\setminus \cA_0) \ge X_{\cG,t}^1(\cA_{\fm_{k_\emptyset}}\setminus \cA_0)\,.$$ In particular, the connected component of each vertex in $\cV_0$ in $X_{\cG,t}^1(E(\cG)\setminus\cA_0)$ is a subset of a connected component of a vertex in $\cV_0$ in $\tilde \omega$. 

Thus, if $N_\omega(A)$ denotes the number of vertices in non-trivial (i.e., non-singleton) components of the boundary condition induced by $\omega(E(\cG)\setminus A)$ on $A$, then 
\begin{align*}
	N_{X_{\cG,t}^1}(\cA_0)\le N_{\tilde \omega}(\cA_0)\,.
\end{align*}
%
\end{observation}

\begin{figure}
    \centering
    \begin{tikzpicture}
    \node at (-3.5,0) {
    \includegraphics[width = .4\textwidth] {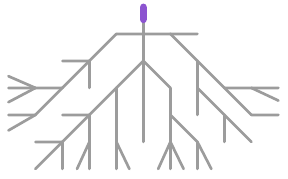}};
    \node at (-3.25,1.6) {$\hat e$};

    \node at (3.5,0) {
    \includegraphics[width = .4\textwidth] {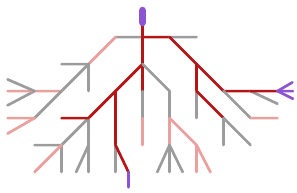}};
        \node at (3.7,1.6) {$\hat e$};
    \end{tikzpicture}
    \caption{Left: We initialize the revealing process with $r=6$ from $\cV_0 = \{v\}$, $\cA_0 = \emptyset$ and the half-edge $\hat e = \hat e_1=  \widehat \cE_0$ (purple). The process begins by revealing $A_1 = B_r(\hat e; \cA_0^c)$, depicted in gray. Right: The process then reveals the configuration $X_{A_1,t}^1$ (open edges  shown in red/pink). Half-edges belonging to vertices in $\partial A_1$ that are in the $X_{A_1,t}^1$-connected component of $v$ (red) are added to form $\widehat \cE_1$ (purple).}
    \label{fig:revealing-fig1}
    \centering
    \begin{tikzpicture}
    \node at (-4,0) {
    \includegraphics[width = .4\textwidth] {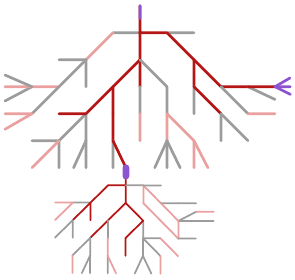}};
        \node[font = \tiny] at (-4.75,-.63) {$\hat e_2$};
    \node at (3.5,0) {
    \includegraphics[width = .48\textwidth] {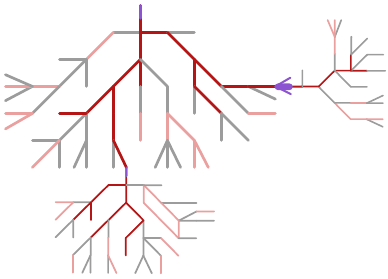}};
            \node[font = \tiny] at (5.3,.75) {$\hat e_3$};
    \end{tikzpicture}
    \vspace{-.25cm}
    \caption{Left: Proceeding from above, in the next generation, starting from  $\hat e_2\in \widehat \cE_1$, the process reveals the edges of $A_1 = B_r(\hat e_2, \cA_1^c)$ in $\cG$; in this case, this is not a tree as it contains a single cycle. The FK-dynamics configuration $X_{A_2,t}^1$ is then generated and concatenated with $\tilde \omega_1$ to form $\tilde \omega_2$. Right: The process continues with $\hat e_3$, revealing $B_r(\hat e_3; \cA_2^c)$ with the FK-dynamics configuration $X_{A_3,t}^1$ on top of it.}
    \label{fig:revealing-fig2}
\end{figure}

With Observation~\ref{obs:key-observation} in hand, we focus on obtaining the stretched exponential tail bound of Theorem~\ref{thm:pi-exponential-decay} for the size of $\cC_v(\tilde \omega)$ (the component of $v$ in $\tilde \omega$) and likewise, the sparsity bound of Theorem~\ref{thm:k-R-sparse-whp}.

\subsubsection{Constructing a dominating branching process}
Towards proving Theorem~\ref{thm:pi-exponential-decay} and~\ref{thm:k-R-sparse-whp}, we construct a branching process (ours will be a size-dependent one but we use the terminology nonetheless) which we will show stochastically dominates the sequence $(\widehat \cE_k)_{k\ge 0}$ of our joint revealing process.
This process $(Z_k)_{k\ge 0}$ will then be shown to be sub-critical, with good tail bounds.

\begin{definition}\label{def:branching-process}
Initialize $Z_{0} =|\widehat \cE_{0}|$. 
    Let $(Z_{k})_{k\ge 0}$ be the branching process, which for each $k\ge 0$, has progeny $(\chi_{i,k})_{i\le Z_k}$, i.e., 
\begin{align*}
    Z_{k+1} = \sum_{i\le Z_k} \chi_{i,k}\,.
\end{align*}
The progeny $\chi_{i,k}$ are distributed as follows. First, let $(\widehat \cT_r^{i,k})_{i,k}$ be i.i.d.\ single-source Galton--Watson trees of depth $r$, with offspring distribution $\mathbb P_{\underline \vdn}$ (from Definition~\ref{def:truncated-empirical-distribution}); recall the single-source here refers to the fact that this is a tree of depth $r$ whose first generation deterministically has one offspring; beyond that first edge, it is simply a Galton--Watson tree of depth $r-1$ with offspring distribution $\mathbb P_{\underline\vdn}$. 
Then the offspring distribution (parametrized by $p,q,\gamma,\epsilon$ and $r,t$), is as follows: 
\begin{enumerate}    
\item With probability $n^{-1/2}$, let $\chi_{i,k} = {\|\vdn\|}_\infty^r \big(\sum_{j<k} Z_j+ \sum_{j<i} \chi_{j,k}\big)$;
\item Otherwise, 
\begin{enumerate}
    \item If $\widehat{\cT}_{r}^{i,k}$ does not satisfy the $(\gamma,\epsilon)$-tree-growth condition  (Definition~\ref{def:gamma-eps-tree-growth}), let $N_t= |E(\widehat{\cT}_{r}^{i,k})|$ and let $\chi_{i,k}$ be a sum of $N_t$ independent random variables drawn from $\mathbb P_{\underline \vdn}$.
    \item If $\widehat \cT_{r}^{i,k}$ does satisfy the $(\gamma,\epsilon)$-tree-growth condition, first generate a configuration on $\widehat \cT_{r}^{i,k}$ by running FK-dynamics with $(1,\circlearrowleft)$ boundary conditions, initialized from $\omega_0 \equiv 1$ for time $t$. Let $N_t$ be the number of vertices of $\partial \widehat \cT_{r}^{i,k}$ that are connected to the root, and let $\chi_{i,k}$ be a sum of $N_t$ independent random variables drawn from $\mathbb P_{\underline \vdn}$.   
\end{enumerate}
\end{enumerate}
\end{definition}

Let us motivate the above construction. Item (1) in the definition of $\chi_{i,k}$ corresponds to cases when either
\begin{itemize}
    \item The ball $B_r(\hat e_m; \cA_{m-1}^c)$ is not a tree, or 
    \item The ball $B_r(\hat e_m; \cA_{m-1}^c)\setminus \{\hat e_m\}$ intersects some already exposed vertex in $\cA_{m-1}$.
\end{itemize}
The $n^{-1/2}$ probability of item (1) is because we will need to take $\cA_0$ possibly as large as $n^{- \frac 12 - o(1)}$. On the latter of these two events, the connected component of $\cV_0$ may, in one step, incorporate many edges of $\cA_{m-1}$, by virtue of an already revealed large connected component of $\cA_{m-1}$. In this case, the best a priori bound we can place on the progeny is the total number of edges revealed up to that point.

In case (2), the newly revealed ball is indeed a tree and does not intersect any already exposed vertex of $\cA_{m-1}$. On the indicator of this event, by Proposition~\ref{prop:ball-domination-by-GW-tree}, the ball is stochastically below $\widehat \cT_{r}^{i,k}$; cases (2a)--(2b) then distinguish whether or not the dominating tree satisfies the $(\gamma,\epsilon)$-tree-growth condition. This is important because if the tree does not satisfy the condition, $p<p_u(q,\gamma)$ will \emph{not} be sub-critical for the $(1,\circlearrowleft)$ random-cluster model on $\widehat \cT_{r}^{i,k}$, and we can only take the full boundary of the tree as our bound on the size of the component of the tree's root.

\subsubsection{Dominating the revealing process by the branching process} 
We are now in position to state the main two lemmas of this section, comparing the revealing procedure to the branching process of Definition~\ref{def:branching-process}, and then establishing its sub-criticality.

Recall that $\mathfrak m_{0} = |\widehat \cE_0|$ and for each $k\ge 1$, $\mathfrak m_{k+1}=\mathfrak m_{k}+|\widehat \cE_{k+1}|$, i.e., in each generation~$k$, $\fm_k$ is the number of half-edges we explore out from. This will be the quantity which we compare to  the population of the  branching process $(Z_k)_k$ of Definition~\ref{def:branching-process}. For notational simplicity, write $\cA_{\infty} = \cA_{\fm_{k_\emptyset}}$.

\begin{lemma}\label{lem:branching-process-domination}
For every $\cA_0, \cV_0$ such that $|\cA_0|,|\cV_0|\le n^{\frac 12 - \delta}$ for $\delta>0$, and every $\ell\ge 1$, 
\begin{align*}
\big(|\widehat \cE_{j}| \mathbf 1_{{\{\mathfrak m_{j} \le n^{1/2 - \delta/ 2} \}}}\big)_{j\le \ell} & \preceq  (Z_{j})_{j\le \ell}\,.
\end{align*}
Furthermore, we have
\begin{align*}
    |\cA_{\infty}\setminus \cA_0| \mathbf 1_{\{\fm_\infty \le n^{1/2 - \delta/2}\}} \preceq \gamma^r \sum_{j=0}^{\infty} Z_j\,.
\end{align*}
\end{lemma}
The proof of Lemma~\ref{lem:branching-process-domination} is briefly deferred to the next subsection; before that proof, we observe that the lemma reduces the analysis of the set of exposed vertices through the revealing process of $(\cG, \tilde \omega)$, and thus, the clusters of $X_{\cG,t}^1$, to the analysis of the process $(Z_k)$, which for most steps is a simple branching process with progeny distribution typically dictated by connectivity probabilities in the wired measure on trees satisfying a $(\gamma,\epsilon)$-tree-growth condition, but occasionally makes large state-dependent jumps. 

Our claim is that if $r$ and $t$ are chosen to be sufficiently large, but $O(1)$, the dominating branching process will be sub-critical. To formalize this claim, let 
\begin{align}\label{eq:bar-tmix}
    \bar \tau_{\textsc{mix}} : = \max_{\mathbb T_r \mbox{ of } (\gamma,\epsilon)-tree-growth} \tmix(\mathbb T_r, (1,\circlearrowleft))\,.
\end{align}
i.e., the maximum over all possible trees of depth $r$ satisfying the $(\gamma,\epsilon)$-tree-growth condition, of the (continuous-time) mixing time with $(1,\circlearrowleft)$ boundary conditions. Now define the \emph{burn-in time}
\begin{align}\label{eq:t-burn}
    T_{\textsc{burn}} = T_{\textsc{burn}}(C_0,r): = C_0 \gamma^r \bar \tau_{\textsc{mix}}\,.
\end{align}

\begin{lemma}\label{lem:branching-process-tail-bounds}
Fix $q\ge 1$, $\gamma>1$ and $p<p_u(q,\gamma)$. For $\epsilon$ sufficiently small and $C_0,r$ and $\kappa$ sufficiently large, if $t\ge T_{\textsc{burn}}(C_0,r)$ and $(\vdn)_n \in \cD_{\gamma,\kappa}$, 
the branching process of Definition~\ref{def:branching-process} satisfies the following tail bound: if $Z_0 \le n^{\frac 12- \delta}$, then for every $M\ge 1$, and every $\lambda: \lambda Z_0 \le n^{\frac 12 - \frac{\delta}2}$,  
\begin{align*}
    \mathbb P\Big(\sum_{j\ge 0} Z_j \ge (1+\lambda) Z_0 \Big) \le CM \exp \Big(\frac{\lambda^{1/M} Z_0 }{C{{\|\vdn\|}}_\infty^{(2+M)r}}\Big) + Cn^{- \delta M/2}\,.
\end{align*}
\end{lemma}

Roughly, the constant $M$ can be thought of as the number of times the ``bad" offspring  distribution of item (1) in Definition~\ref{def:branching-process} is selected, allowing the total population to double, and away from such ``bad" updates, we will show that the branching process indeed satisfies exponential tails.

\subsection{Coupling the revealing process to the branching process}
We next prove the desired stochastic domination relation between the revealing process $(\widehat \cE_k)_k$ and $Z_k$ by constructing a coupling between the two such that the former is below the latter while the total population is at most $n^{1/2 - \delta/2}$. 

\begin{proof}[\textbf{\emph{Proof of Lemma~\ref{lem:branching-process-domination}}}]
We proceed by induction over $\ell\ge 0$. The base case, $Z_0 = |\widehat \cE_0|$, is by construction.  Now fix $\ell \ge 1$ and suppose by way of induction that the following stochastic domination holds:
$$(|\widehat \cE_j|\mathbf 1_{\{\mathfrak m_{j} \le n^{1/2 - \delta/2}\}})_{j\le \ell-1}\preccurlyeq (Z_j)_{j\le \ell-1}\,.$$
Thus, there exists a monotone coupling of the sequence on the left-hand side, such that it is below the sequence $(Z_j)_{j\le \ell-1}$ in the natural element-wise ordering on the sequence.  Working on that coupling, it suffices for us to then show that on the event $\{\mathfrak m_{\ell-1}\le n^{1/2 - \delta/2}\}$, for every  $m\in \{\mathfrak m_{\ell-1}+1,...,\mathfrak m_{\ell}\}$, the distribution of the children of $\widehat e_m$ is stochastically below the progeny distribution of Definition~\ref{def:branching-process}. Here, by children of $\widehat e_m$, we mean the set of half-edges added in step 3.(b) of Process~\ref{proc:revealing-FK-on-graph}. In what follows, denote that set by~$\Xi(\widehat e_m)$.  

Define the event $\Gamma_{\mathsf{good}}$ on the revealed ball $B_r(\widehat e_m; \cA_{m-1}^c)$ as 
the event that
\begin{align*}
        V(B_r(\hat e_m; \cA^c_{m-1})\setminus \{\widehat e_m\}) \cap V(\cA_{m-1}) = \emptyset \qquad \mbox{and}\qquad B_r(\hat e_m; \cA_{m-1}^c) \mbox{ is a tree}\,. 
\end{align*}
On the bad event $\Gamma_{\mathsf{good}}^c$, we take the a priori bound of $\hat E(\cA_{m})$ on the set $\Xi(\widehat e_m)$, namely assuming that in the worst-case all exposed half-edges of $\cA_m$, both those in $B_r(\widehat e_m;\cA_{m-1}^c)$, and those of $\cA_{m-1}$ become connected up to $\cV_0$ in $\tilde \omega_m$. By the inductive hypothesis, the number of such half-edges is at most ${\|\vdn\|}_\infty^r$ times the population of the branching process up to that step, given by $\sum_{j<k}Z_j + \sum_{j<i} \chi_{j,k}$, where the ${\|\vdn\|}_\infty^r$ comes from assuming that in each of these steps the corresponding ball we revealed in the graph has maximal size. We claim that the probability of $\Gamma_{\mathsf{good}}^c$ is at most $n^{-1/2}$. 
To see this, notice that in the breadth-first revealing of $B_r(\hat e_m ; \cA^c_{m-1})$, the probability that the next half-edge that gets matched is matched either with a vertex having an edge in $\cA_{m-1}$, or an already exposed vertex of $B_r(\hat e_m; \cA_{m-1}^c)$ is at most 
$$\frac{|\widehat E(\cA_{m})|}{{\|\vdn\|}_1 - |\widehat E(\cA_{m})|}  \le \frac{\fm_\ell{\|\vdn\|}_\infty }{{\|\vdn\|}_1 - \fm_\ell}\,.$$
Assuming that $(\vdn)_n\in \cD_{\gamma,\kappa}$ (for $\kappa$ to be chosen sufficiently large later), by Fact~\ref{fact:L-infty-degree-sequence}, ${\|\vdn\|}_\infty \le n^{\epsilon_*(\kappa)}$. Using this upper bound, the bound $\fm_{\ell}\le n^{1/2- \delta/2}$ (as otherwise the indicator on the left-hand side of the desired stochastic domination would be zero),  and the lower bound of ${\|\vdn\|}_1 \ge \Omega(n)$, we see that this probability is at most $n^{1/2 + \delta/2 - \epsilon_*r}$. Through the revealing of $B_r(\hat e_m, \cA_{m-1}^c)$ we make at most ${\|\vdn\|}_\infty^r$ attempts at such a bad matching, and thus a union bound implies that 
\begin{align*}
    \mathbb P\big(B_r(\hat e_m; \cA^c_{m-1})\in \Gamma_{\mathsf{good}}^c \mid \mathfrak F_{m-1}, \fm_{\ell} \le n^{1/2 - \delta/2}\big) \le n^{- 1/2 -\delta/2 + 2\epsilon_* r}\,,
\end{align*}
where $\mathfrak F_{m-1}$ is the filtration generated by the randomness of the revealing procedure through the $(m-1)$'th step. 
The right-hand side above is at most $n^{-1/2}$ as long as $\kappa$ is large enough that $\epsilon_*(\kappa) < \frac{1}{2}\delta(2r)^{-1}$. 

Now work on the event that $B_r(\hat e_m; \cA_{m-1}^c)\in \Gamma_{\mathsf{good}}$, and recall from Proposition~\ref{prop:ball-domination-by-GW-tree} that in this case the ball is stochastically dominated by (and in particular there exists a coupling such that it can be be embedded as a subset of) a Galton--Watson tree of depth $r$ with offspring distribution $\mathbb P_{\underline \vdn}$. This is the law of~$\widehat \cT_{r}^{i,\ell-1}$. 

Evidently, on the event that $\widehat \cT_{r}^{i,\ell-1}$ does not satisfy the $(\gamma,\epsilon)$-tree-growth condition, the number of children $|\Xi(\hat e_m)|$ is at most the number of half-edges emanating from $\partial B_{r}(\hat e_m; \cA_{m-1}^c)$, which is at most $|\partial \widehat\cT_{r+1}^{i,\ell-1}|$, or a sum of $|\partial \widehat\cT_r^{i,\ell-1}|\le N_t$ independent draws from $\mathbb P_{\underline \vdn}$.

Finally, suppose we are on the event that $\widehat \cT_{r}^{i,\ell-1}$ does satisfy the $(\gamma,\epsilon)$-tree-growth condition, so that $B_{r}(\hat e_m; \cA_{m-1}^c)$ does as well. In that case, by Proposition~\ref{prop:ball-domination-by-GW-tree}, there is a coupling such that the graph $B_{r}(\hat e_m; \cA_{m-1}^c)$ is a subgraph of $\widehat \cT_r^{i,\ell-1}$. One can then couple the FK-dynamics chain $X_{A_m,t}^{1}$ to $Y_t^1$, the FK-dynamics chain on $\widehat \cT_r^{i,\ell-1}$ with its $(1,\circlearrowleft)$ boundary conditions run for time $t$ initialized from $Y_0 \equiv 1$ such that $X_{A_m,t}^{1}$ is below $Y_t^1$ with probability one. In particular, the vertices of $\partial A_m$ which are in the open cluster of $\hat e_m$ in $X_{A_m,t}^1$, are a subset of the vertices of $\partial \widehat\cT_{r}^{i,\ell-1}$ which are in the open cluster of the root in $Y_t^1$, so that the number of them, call it $N(X_{A_m,t}^1)$ is less than $N_t = N(Y_t)$. Since the law of $X_{A_m,t}^{1}$ is independent of the choice of boundary vertices, and thus degree sequence at $\partial A_m$, the number of half-edges added in step 3.(b) of Process~\ref{proc:revealing-FK-on-graph} is a sum of $N(X_{A_m,t}^1)$ independent draws from the empirical degree distribution at that point, which is stochastically below $\mathbb P_{\underline \vdn}$. Therefore, this establishes the domination on this event of the number of children of $\hat e_m$ by item (2b) of the dominating branching process. 

In order to then deduce the domination of $|\cA_\infty \setminus \cA_0|$ by $\gamma^r$ times the total population of the dominating branching process, we make the following observation. On the event $\Gamma_{\mathsf{good}}^c$, we were already bounding $\Xi(\widehat e_m)$ by $|\cA_m \setminus \cA_0|$, and that in turn by $\chi_{i,k}$, even without the factor of $\gamma^r$.  Similarly in the case of (2a). In the event of (2b), we notice that by the $(\gamma,\epsilon)$-tree-growth condition, the number of edges in $B_r(\hat e_m; \cA_{m-1}^c)$ is at most $\gamma^r$.
\end{proof}

\subsection{Sub-criticality of an auxiliary branching process} 
The branching process of Definition~\ref{def:branching-process} is not a branching process in a traditional sense, as when it follows item (1) in the definition, its offspring count is state-dependent. Such offspring can create large jumps in the total population, and lead to difficulties in the analysis. We analyze the process by means of an auxiliary branching process that captures the behavior of $(Z_j)_j$ in between its rare state-dependent steps.  More formally, we say an offspring of the branching process of Definition~\ref{def:branching-process} is  $\mathsf{bad}$ if item (1) of Definition~\ref{def:branching-process} is taken. 

\begin{definition}\label{def:auxiliary-branching-process}
     Consider the \emph{auxiliary} branching process $(\tilde Z_j)_j$ which is defined exactly as in Definition~\ref{def:branching-process}, except its offspring are conditioned to never be $\mathsf{bad}$. Namely, let $(\tilde \chi_{i,k})_{i,k}$ be a sequence of i.i.d.\ draws from item (2) of Definition~\ref{def:branching-process}, and for a fixed $\tilde Z_0$, construct $(\tilde Z_j)_j$ iteratively by $\tilde Z_j = \sum_{i\le Z_{j-1}} \tilde \chi_{i,j}$. 
\end{definition}

The following lemma establishes sub-criticality and tail bounds for the auxiliary branching process---in other words, the branching process during the epochs between the $\mathsf{bad}$ updates of $(Z_j)_j$. 
\begin{lemma}\label{lem:auxiliary-branching-process}
Fix $q\ge 1$, $\gamma>1$ and $p<p_u(q,\gamma)$. For $\epsilon$ sufficiently small and $C_0,r$ and $\kappa$ sufficiently large, if $t\ge T_{\textsc{burn}}(C_0,r)$ and $(\vdn)_n \in \cD_{\gamma,\kappa}$, 
the auxiliary branching process $(\tilde Z_j)_j$ is uniformly sub-critical, i.e., $\limsup_n \mathbb E[\tilde \chi]<1$. Furthermore, it satisfies the following tail bound: for all $\lambda$ sufficiently large, 
\begin{align*}
    \mathbb P\Big(\sum_{j\ge 0} \tilde Z_j \ge  \lambda \tilde Z_0\Big) \le C\exp\Big( - \frac{\lambda \tilde Z_0}{C{\|\vdn\|}_\infty^{2r}}\Big)\,.
\end{align*}
\end{lemma}

\begin{proof}[\textbf{\emph{Proof of Lemma~\ref{lem:auxiliary-branching-process}: sub-criticality}}]
Let us begin by calculating the mean of the offspring distribution of the auxiliary branching process, which corresponds to the offspring distribution of Definition~\ref{def:branching-process} conditional on being from item (2). By construction, 
\begin{align*}
    \mathbb E[\tilde \chi_{i,k}] &  = \mathbb E[\mathbf 1\{\widehat\cT_{r}^{i,k} \notin (\gamma,\epsilon)\mbox{-tree-growth}\}|\widehat\cT_{r}^{i,k}|] \mathbb E_{\underline \vdn}[D] \\
    & \qquad \qquad + \mathbb E[\mathbf 1\{\widehat\cT_{r}^{i,k}\in (\gamma,\epsilon)\mbox{-tree-growth}\} |\cC_\rho(Y_t^1)\cap \partial \widehat\cT_{r}^{i,k}|]\mathbb E_{\underline \vdn}[D]\,.
\end{align*}
We can bound the first term by Cauchy--Schwarz as 
\begin{align*}
    \mathbb E[\mathbf 1\{\widehat \cT_{r}^{i,k} \notin (\gamma,\epsilon)\mbox{-tree-growth}\} & |\widehat \cT_{r}^{i,k}|]\mathbb E_{\underline \vdn}[D]  \\
    & <  \gamma \cdot \mathbb P(\widehat \cT_{r}^{i,k} \notin (\gamma,\epsilon)\mbox{-tree-growth})^{1/2} \mathbb E[|\widehat \cT_{r}^{i,k}|^2]^{1/2}\,.
\end{align*}
The probability on the right-hand side is at most $\eta^{- \kappa \epsilon r}$ for some $\eta$ small, by Corollary~\ref{cor:random-tree-has-tree-growth} if $(\vdn)\in \cD_{\gamma,\kappa}$.  The expectation above is at most $C \gamma^{r}$ using the moment bound of~\eqref{eq:population-moment-bound}. Thus taking $\kappa$ large depending on $\epsilon$, we see that this product is exponentially small in $r$, and can be taken as close to $0$ as desired by taking $r$ sufficiently large. 

Turning to the second term in the expansion of $\mathbb E[\tilde \chi_{i,k}]$, we can first bound it by 
\begin{align*}
    \mathbb E_{\underline \vdn}[D]\mathbb E[\mathbf 1\{\widehat \cT_{r}^{i,k}\in (\gamma,\epsilon)\mbox{-tree-growth}\}& |\cC_\rho(Y_{t}^1)\cap \partial \widehat \cT_{r}^{i,k}|] \\
    & < \gamma \cdot \sup_{\mathbb T_r \in (\gamma,\epsilon)\mbox{-tree-growth}} \mathbb E [|\cC_\rho(Y_{\mathbb T_r^{(1,\circlearrowleft)},t}^1)\cap \partial \mathbb T_r|]\,.
\end{align*}
where $(Y_{\mathbb T_r^{(1,\circlearrowleft)},t}^1)_{t \ge 0}$ is a continuous-time FK-dynamics on the tree $\mathbb T_r$ with $(1,\circlearrowleft)$ boundary conditions, initialized from the all-wired configuration. 
Now recall that the stationary measure of $Y_{\mathbb T_r^{(1,\circlearrowleft)},t}^{1}$ is $\pi_{\mathbb T_r}^{(1,\circlearrowleft)}$. By the $(\gamma,\epsilon)$-tree-growth condition, $|\mathbb T_r| \le \gamma^r$. As such, there exists some $C_{r,\gamma}>0$ such that the mixing time of $Y_{\mathbb T_r^{(1,\circlearrowleft)},t}^1$, i.e., $\bar \tau_{\textsc{mix}}$ as defined in~\eqref{eq:bar-tmix}, is at most $C_{r,\gamma}$. By sub-multiplicativity of total-variation distance (see e.g.,~\cite{LP}), then, if $t\ge T_\textsc{burn}(C_0,r) = C_0 \gamma^r \bar \tau_{\textsc{mix}}$ as defined in~\eqref{eq:t-burn}, we have 
\begin{align*}
    \sup_{\mathbb T_r \in (\gamma,\epsilon)\mbox{-tree-growth}}\|\mathbb P(Y_{\mathbb T_r^{(1,\circlearrowleft)},t}^1\in \cdot) - \pi_{\mathbb T_r}^{(1,\circlearrowleft)}\|_\tv \le C\exp ( - C_0 \gamma^r/C)\,.
\end{align*}
Using this, for every $\mathbb T_r \in (\gamma,\epsilon)$-tree-growth, we can bound the expectation 
\begin{align*}
    \mathbb E [|\cC_\rho(Y_{\mathbb T_r^{(1,\circlearrowleft)},t}^1)\cap \partial \mathbb T_r|] & \le \mathbb E_{\pi_{\mathbb T_r}^{(1,\circlearrowleft)}} [|\cC_\rho(\omega)\cap \partial \mathbb T_r|] + |\partial \mathbb T_r| \|\mathbb P(Y_{\mathbb T_r^{(1,\circlearrowleft)},t}^1\in \cdot) - \pi_{\mathbb T_r}^{(1,\circlearrowleft)}\|_\tv \\
     & \le |\partial \mathbb T_r| \max_{v\in \partial \mathbb T_r} \pi_{\mathbb T_r}^{(1,\circlearrowleft)} (v\in \cC_\rho(\omega)) + C|\partial \mathbb T_r|e^{ - C_0 \gamma^r/C}\,.
\end{align*}
Using the fact that $\mathbb T_r$ has $(\gamma,\epsilon)$-tree-growth and using the bound of Corollary~\ref{lemma:exp:decay:wired:treelike} to bound the probability of a leaf being in the component of the root, we bound the above by 
\begin{align*}
    C\gamma^{r}\ps^{(1-\epsilon)r}  + C\gamma^r e^{ - C_0 \gamma^r/C}\,.
\end{align*}
Recall that when $p<p_u(q,\gamma)$, we have $\ps<1/\gamma$, from which it follows that for sufficiently small $\delta,\epsilon$ and sufficiently large $C_0$, uniformly over large  $r$ the above quantity is strictly less than $1/\gamma$, so that when multiplied by $\mathbb E_{\underline \vdn} [D]<\gamma$, it is strictly less than $1$. Combining this with the bound on the first term in the expectation, we find that there exist $\epsilon(p,q,\gamma)$ and $C_0(\gamma)$ such that for all $r$ sufficiently large, we have $\limsup_n \mathbb E[\tilde \chi_{i,k}]<1$ as desired. 
\end{proof}

\begin{proof}[\textbf{\emph{Proof of Lemma~\ref{lem:auxiliary-branching-process}: tail bounds}}]
Having established sub-criticality of the dominating branching process, we now wish to boost this to tail bounds on the number of generations, and total population of the branching process. For this, we use the traditional random-walk exploration of a branching process. Namely, the population beyond $\tilde Z_0$ can be expressed as a sum of i.i.d.'s and we can write the active population in the branching process beyond the first generation as the killed random walk 
\begin{align*}
    \tilde Z_0 + \sum_{i \le N_0} (\tilde \chi_i - 1)\,, \qquad \mbox{where}\qquad N_0 = \inf\Big\{j: \tilde Z_0 + \sum_{i\le j} (\tilde \chi_i -1) = 0\Big\}\,,
\end{align*}
where $(\tilde \chi_i)_{i}$ are i.i.d.\ copies from the offspring distribution of Definition~\ref{def:branching-process}. 
Observe that with this representation, the \emph{total} population of the branching process is exactly $N_0$. 
Then, we can express tail bounds for this branching process's total population as 
\begin{align*}
    \mathbb P\Big(\sum_{0\le j<\infty} \tilde Z_j \ge \lambda \tilde Z_0 \Big) \le \mathbb P\big(N_0 >\lambda \tilde Z_0\big) \le \mathbb P\Big(\tilde Z_0 + \sum_{i\le  \lambda \tilde Z_0} (\tilde \chi_i - 1)>0\Big)\,.
\end{align*}
Consider the random variable $\tilde \chi_i -1$; its mean satisfies $\mathbb E[\tilde \chi_1- 1] \le -\eta$ for some $\eta>0$, by the sub-criticality established in the previous proof. 
Thus this is a sum of $\lambda \tilde Z_0$-many i.i.d.\ random variables, the sum has mean smaller than $-\eta \lambda\tilde Z_0$, and the increments are bounded in $\ell_\infty$ by ${\|\vdn\|}_\infty^r$. Thus, 
\begin{align*}
    \mathbb P\Big(\sum_{i\le \lambda \tilde Z_0} (\tilde \chi_i - 1) >- \tilde Z_0\Big) &  \le  \mathbb P\Big(\Big|\sum_{i\le \lambda \tilde Z_0} (\tilde \chi_i - 1) - \mathbb E[\tilde \chi_1 -1]\Big| >\eta (\lambda - \eta^{-1})\tilde Z_0 \Big) 
\end{align*}
As long as $\lambda> 2\eta^{-1}$, by Hoeffding's inequality, this gives 
\begin{align*}
    \mathbb P\Big(\sum_{i\le \lambda \tilde Z_0} (\tilde \chi_i - 1) >- \tilde Z_0\Big) \le C\exp\Big( \frac{\lambda \tilde Z_0}{C {\|\vdn\|}_\infty^{2r}}\Big)
\end{align*}
as desired. 
\end{proof}

\subsection{Controlling the original branching process by a sum of auxiliary branching processes}
Given the sub-criticality and tail bounds for the auxiliary branching process, we can now obtain tail bounds on the original dominating branching process $(Z_j)_j$ as required by Lemma~\ref{lem:branching-process-tail-bounds}.  Let us now construct a process out of i.i.d.\ copies of the auxiliary branching process, that stochastically dominates the original branching process. Let $(\tilde Z_j^{(i)})_j$  be i.i.d. copies of the branching process of Definition~\ref{def:auxiliary-branching-process}, with initializations $\tilde Z_0^{(1)} = Z_0$ and $\tilde Z_0^{(i)} = \|\bd_n\|_\infty^r \sum_{j\ge 0} \tilde Z_{j}^{(i-1)}$ given by the total population of the previous auxiliary branching process. 

In what follows, for fixed $\lambda$, consider the stopping generation 
\begin{align*}
    \kappa_\lambda = \min \Big\{k : \sum_{0\le j \le k}Z_j > \lambda Z_0\Big\}\,. 
\end{align*}
Let $\Gamma_{\mathsf{bad},M}$ be the event that there are at most $M$ many $\mathsf{bad}$ offspring in the first $\kappa_\lambda$ many generations of the branching process $(Z_j)_j$. The following stochastic domination is self-evident by construction. 

\begin{claim}\label{lem:branching-process-auxiliary-comparison}
Given the above construction, $(\sum_{j\ge 0} Z_j)\mathbf 1_{\Gamma_{\mathsf{bad},M}} \preceq \sum_{1 \le i\le M} \sum_{j\ge 0} \tilde Z_j^{(i)}$.
\end{claim}

Given this stochastic domination, we can now establish Lemma~\ref{lem:branching-process-tail-bounds}. 

\begin{proof}[\textbf{\emph{Proof of Lemma~\ref{lem:branching-process-tail-bounds}}}]
By a union bound, we have 
\begin{align*}
    \mathbb P\Big( \sum_{j\ge 0} Z_j \ge \lambda Z_0\Big) \le \mathbb P ( \Gamma_{\mathsf{bad},M}^c)  + \mathbb P\Big(\sum_{1\le i \le M} \sum_{j\ge 0} \tilde Z_j^{(i)}\ge \lambda Z_0\Big) \,.
\end{align*}
The first probability is bounded as follows: for every $\lambda: \lambda Z_0\le n^{\frac 12 - \frac {\delta}{2}}$, we have 
\begin{align*}
    \mathbb P(\Gamma_{\mathsf{bad},M}^c) \le \mathbb P\Big(\mbox{Bin}\big(\lambda Z_0, n^{-1/2}\big) \ge M\Big) \le C n^{-\delta M/2}\,.
\end{align*}
The second probability above can be bounded as 
\begin{align*}
    \mathbb P\Big(\sum_{1\le i \le M} \sum_{j\ge 0} \tilde Z_j^{(i)}\ge \lambda Z_0\Big) \le \sum_{1\le i\le M} \mathbb P\Big( \sum_{j\ge 0} \tilde Z_j^{(i)} \ge \frac{1}{M^{1/M}\|\bd_n\|_\infty^{ri}} \lambda^{1/M} \tilde Z_0^{(i)}  \Big)
\end{align*}
Indeed, if for every $i$, $\sum_{j\ge 0} \tilde Z_j^{(i)}\le M^{-1/M}\lambda^{1/M} \|\bd_n\|_\infty^{-ri} \tilde Z_0^{(i)}$, then $\sum_{1\le i\le M} \sum_{j\ge 0} \tilde Z_j^{(i)} \le \lambda Z_0$. In order to now bound the right-hand side, we use the tail bounds of Lemma~\ref{lem:auxiliary-branching-process} to deduce that 
\begin{align*}
    \mathbb P\Big(\sum_{1\le i \le M} \sum_{j\ge 0} \tilde Z_j^{(i)}\ge \lambda Z_0\Big) \le C M \exp\Big(- \frac{\lambda^{1/M} Z_0}{C M^{1/M} {\|\vdn\|}_\infty^{(2+M)r}}\Big)\,.
\end{align*}
Combined with the bound on $\mathbb P(\Gamma_{\mathsf{bad},M}^c)$, we obtain the desired result. 
\end{proof}

\subsection{Tail bounds on cluster sizes, and shattering of the dynamics}\label{subsec:main-revealing-proof}
We are now in a position to conclude the proof of the tail bounds on clusters of $X_{\cG,t}^1$, and use that to deduce that $X_{\cG,t}^1$ is $(K,R)$-$\sparse$, except with probability $o(n^{-5})$.  We begin by using Lemmas~\ref{lem:branching-process-domination}--\ref{lem:branching-process-tail-bounds} to prove the following tail bound on $|\cA_\infty|$, which we recall counts the number of edges exposed through the revealing process of Process~\ref{proc:revealing-FK-on-graph}. 

\begin{lemma}\label{lem:revealing-procedure-tail-bounds}
Fix $\delta>0$ and consider the revealing procedure for any initial pair $(\cV_0, \cA_0)$ having $|\cA_0|,|\cV_0|$ and $|\widehat \cE_0|$ all at most $n^{\frac 12 - \delta}$. There exist $C_0 (p,q,\gamma), r (p,q,\gamma)$ in the definition of~$\Tburn$ in~\eqref{eq:t-burn} and $\kappa(p,q,\gamma)$ such that for all $t\ge \Tburn$ the following holds. 
For all $(\vdn)_n \in \cD_{\gamma,\kappa}$, $M\ge 1$ and $\lambda: \lambda |\widehat \cE_0| \le n^{\frac{1}{2}- \delta}$,  
\begin{align*}
\mathbb P\Big( |\cA_\infty| \ge |\cA_0| + \gamma^r(\lambda |\widehat \cE_0|) \Big) \le CM \exp \Big(\frac{\lambda^{1/M} Z_0 }{C{\|\vdn\|}_\infty^{(M+2)r}}\Big) + Cn^{- \delta M/2}\,.
\end{align*}
\end{lemma}

\begin{proof}
Define the following stopping generation  
\[
\varsigma = \inf\Big\{\ell: \mathfrak m_{\ell - 1} > \lambda|\widehat \cE_0|  \Big\}\,.
\]
 Similarly define $\varsigma_Z$ as the first $\ell: \sum_{j\le \ell-1} Z_j > \lambda Z_0$. Under the monotone coupling of Lemma~\ref{lem:branching-process-domination}, if $\varsigma_Z = \infty$, then $\varsigma= \infty$, the indicators in the lemma are both $1$, and both 
 $$(|\widehat \cE_j|)_{j} \le (Z_j)_{j}\,, \qquad \mbox{and}\qquad |\cA_\infty \setminus \cA_0| \le \gamma^r \sum_{j=0}^{\infty} Z_j\,,$$
hold. Therefore, 
we obtain 
\begin{align*}
    \mathbb P \Big( |\cA_\infty \setminus \cA_0 | \ge  \gamma^{r} (\lambda |\widehat \cE_0|) \Big) \le \mathbb P \Big( \sum_{k\ge 0} Z_k \ge \lambda Z_0  \Big)\,.
\end{align*}
Lemma~\ref{lem:branching-process-tail-bounds} then implies the desired result.
\end{proof}

We next use Lemma~\ref{lem:revealing-procedure-tail-bounds} and Observation~\ref{obs:key-observation} to deduce tail estimates on the volume and radius of the cluster in $X_{\cG,t}^1$ containing $v$, when $t\ge T_{\textsc{burn}}$. 

\begin{proof}[\textbf{\emph{Proof of Theorem~\ref{thm:pi-exponential-decay}}}]
    Fix some $v\in \{1,...,n\}$, let $\cA_0 = \emptyset$ and let $\cV_0=\{v\}$ in Process~\ref{proc:revealing-FK-on-graph}. In this case $\widehat \cE_0$ is the set of half-edges out from $v$, and thus $|\widehat \cE_0| = d_v$. By Observation~\ref{obs:key-observation}, for each $\cG\sim  \Pcm$, the cluster of $v$ in the configuration $X_{\cG,t}^1$, denoted  $\cC_{v}(X_{\cG,t}^1)$ is a subset of $\cC_v(\tilde \omega)$, which in turn is a subset of $V(\cA_{\mathfrak m_{\emptyset}}\setminus \cA_0)$. Let $C_0,r$ be sufficiently large constants and take $t\ge T=  T_{\textsc{burn}}(C_0,r)$. 
    Then, we have  
    \begin{align*}
    |\cC_{v}(X_{\cG,t}^1)|\le |\cC_v(\tilde \omega)| \le |V(\cA_{\fm_\emptyset}\setminus \cA_{0})|\le 2|\cA_{\fm_\emptyset} \setminus \cA_0|\,.
    \end{align*}
    By Lemma~\ref{lem:revealing-procedure-tail-bounds} and the above, if $(\vdn)_n\in \cD_{\gamma,\kappa}$, there exists $C(p,q,\gamma)$ such that
    \begin{align*}
        \mathbb P \big((\cG,X_{\cG,t}^1): |\cC_v(X_{\cG,t}^1)|\ge \gamma^r \big(\lambda d_v)\big) \le C M \exp \Big(\frac{\lambda^{1/M} d_v }{C{\|\vdn\|}_\infty^{(M+2)r}}\Big) + Cn^{- \delta M/2}\,.
    \end{align*}
    Let $\delta = 1/4$ and let $M = 200$, for instance. For any fixed small $\epsilon>0$, by taking $\kappa$ sufficiently large, by Fact~\ref{fact:L-infty-degree-sequence}, ${\|\vdn\|}_{\infty}^{(M+2)r}< n^{\epsilon/4rM}$; then taking $\lambda = n^{(M-1)\epsilon/4 M}$, 
    we satisfy that $\lambda d_v \le n^{\frac 12 - \delta}$. Then we see that 
    \begin{align*}
        \gamma^r( \lambda d_v)< n^{\epsilon}\,, \qquad \mbox{and}\qquad \frac{\lambda^{1/M} d_v }{C{\|\vdn\|}_\infty^{(M+2)r}} > n^{\epsilon/4M}/C\,.
    \end{align*}
    In turn, the probability above is at most $o(n^{-24})$. 
    Observing that $\mathbb P ((\cG,X_{\cG,t}^{1}): X_{\cG,t}^1\in \cdot ) = \Ecm [\mathbb P(X_{\cG,t}^1\in \cdot)]$, we can use Markov's inequality to write 
    \begin{align*}
        \Pcm\bigg( \cG: \mathbb P\Big(X_{\cG,t}^1: |\cC_v(X_{\cG,t}^1)| \ge n^{\epsilon}\Big)\ge  n^{-12}\bigg)  \le n^{-12}\,,
    \end{align*}
    implying the desired result. 
\end{proof}

We next establish that the $(K,R)$-$\sparse$ property for the random-cluster configuration on $\cG \sim \Pcm$ holds with high probability for all $t\ge T_{\textsc{burn}}$. Towards this, we introduce the following notation. 

\begin{definition}\label{def:mathfrak-V}
    Given a graph $\cG$, a vertex subset $\cV_0$, an edge subset $\cA_0$, and a configuration
$\omega$ on $E(\cG)$, define $\mathfrak{V}_{(\cV_0,\cA_0)}(\omega)$ as the subset of vertices in $\cV_0$ in non-singleton components in the boundary condition induced 
by $\omega(E(\cG)\setminus \cA_0)$.
\end{definition}

\begin{lemma}\label{lem:nontrivial-bdy-component-bound}
Fix $q\ge 1$, $\gamma>1$,  
$p<p_{u}(q,\gamma)$, and $\delta>0$. Let $R \le (\frac{1}{2}-\delta)\log_{\gamma}n$. 
There exist $\kappa$, $K$ as well as $C_0$ and $r$, such that for every $v\in \{1,...,n\}$ for all $t\ge \Tburn(C_0,r)$ and all $(\vdn)_n \in \cD_{\gamma,\kappa}$ the following holds for $\cA_0 = B_R(v)$ and $\cV_0 = \partial B_R(v)$: 
\begin{align*}
    \mathbb P\big((\cG,X_{\cG,t}^{1}): |\mathfrak V_{(\cV_0,\cA_0)}(X_{\cG,t}^1)|>K\big) \le o(n^{-10})\,.
\end{align*}
\end{lemma}

We use Theorem~\ref{thm:pi-exponential-decay} to bound the number of chances the revealing process of Process~\ref{proc:revealing-FK-on-graph} has to reconnect to the vertices of $\cV_0 = \partial B_R(v)$. Intuitively, since the components of $\tilde \omega$ have (stretched) exponential tail bounds, the number of chances at reconnecting is of the same order as $|\cV_0|$; since $R$ is such that $|\cV_0|\le n^{1/2 - \delta}$, the number of such connections (each possibly inducing a non-trivial boundary component) will be dominated by an $\mbox{Bin}(n^{1/2+\delta},n^{-1/2-\delta})$ random variable, yielding the desired tail bound on the probability of this exceeding some large $K$.

\begin{proof}[\textbf{\emph{Proof of Lemma~\ref{lem:nontrivial-bdy-component-bound}}}]
Fix $v\in \{1,...,n\}$ and $\delta>0$, and any $R \le (\frac{1}{2} - \delta)\log_\gamma n$. First of all, we recall from Lemma~\ref{lem:random-graph-volume-growth}, that if we let $\Gamma$ be the event that $\cG$ has the $(\gamma,\epsilon)$-volume-growth property, then  
\begin{align*}
    \Pcm(\Gamma^c) \le o(n^{ - 10})\,.
\end{align*}
We will henceforth work on the event $\Gamma$. Reveal the sub-graph $B_R(v)$ on the event $\Gamma$ (such that its volume is at most $\gamma^R$) and initialize  $\cV_0 = \partial B_R(v)$ and $\cA_0= E(B_R(v))$. We apply the revealing procedure of Process~\ref{proc:revealing-FK-on-graph} with this initialization. Recall from Observation~\ref{obs:key-observation} that the FK-clusters of $\cV_0$ induced by  $\tilde \omega(E(\cG)\setminus \cA_0)$ are a subset of $\cA_{\mathfrak m_{k_\emptyset}}\setminus \cA_0$, and the configuration $\tilde \omega$ satisfies $\tilde \omega(\cA_{\mathfrak m_{k_\emptyset}}\setminus \cA_0) \ge X_{\cG,t}^1(\cA_{\mathfrak m_{k_\emptyset}}\setminus \cA_0)$. Thus, the sets $\mathfrak V_{(\cV_0,\cA_0)}(\tilde \omega)$ and $\mathfrak V_{(\cV_0,\cA_0)}(X_{\cG,t}^1)$, are subsets of $\mathfrak V_{(\cV_0,\cA_0)}(\cA_{\mathfrak m_{k_\emptyset}}\setminus \cA_0)$.

Through the revealing process of Process~\ref{proc:revealing-FK-on-graph}, for each $m$, the edges of $B_r(\hat e_m; \cA_{m-1}^c)$ are revealed one at a time via the breadth-first revealing per Processes~\ref{proc:configuration-model-revealing-graph} and~\ref{proc:B-r-out}. Therefore, $|\mathfrak  V_{(\cV_0,\cA_0)}(\mathcal A_{\mathfrak m_{k_\emptyset}}\setminus \cA_0)|$ is at most the number of times during the revealing of $\mathcal A_{\mathfrak m_{k_\emptyset}}$, that a half-edge is matched up to a half-edge belonging to a vertex that had already been discovered. For a fixed $m$, consider the revealing of $B_r(\hat e_m; \cA_{m-1}^c)$. Conditionally on a discovered edge set $\cA$
the law of the next half-edge to be matched is uniform amongst all un-matched half-edges. Thus, uniformly over the history of the revealing process up to that point, the probability that the next half-edge to be matched is matched up to a vertex of $V(\cA)$ is at most 
\begin{align*}
    \frac{|\cA_{\fm_{k_\emptyset}}|{\|\vdn\|}_\infty}{{\|\vdn\|}_1 - |\cV_{\fm_{k_\emptyset}}| {\|\vdn\|}_\infty}\,.
\end{align*}

We thus obtain for a sufficiently large constant $\Lambda$ (depending on $p,q,\gamma,r$), for all $L\ge 1$, 
\begin{align*}
    \mathbb P \Big((\mathcal G, \tilde \omega): & \cG \in \Gamma, |\mathfrak V_{(\cV_0,\cA_0)}(X_{\cG,t}^1)| > L\Big) \\
    & \le \mathbb P \Big(\Gamma,  |\cA_{\mathfrak m_{k_\emptyset}}|> n^{\frac 12 - \frac{\delta}{2}}\Big)   + \mathbb P \Big(\bin\Big( n^{\frac 12 - \frac{\delta}{2}}{\|\vdn\|}_\infty , 2 n^{-\frac{\delta}{2} -\frac{1}{2}}{\|\vdn\|}_\infty\Big)>L\Big)\,.
\end{align*}
By the bound $|\widehat \cE_0| \le {\|\vdn\|}_\infty \gamma^R \le n^{\frac 12 -\frac{2\delta}{3}}$ as long as ${\|\vdn\|}_\infty \le n^{\epsilon_*}$ for a sufficiently small $\epsilon_*$ (which holds as long as $\kappa$ is sufficiently large in $\delta,M$ by Fact~\ref{fact:L-infty-degree-sequence}), we can apply Lemma~\ref{lem:revealing-procedure-tail-bounds} with a sufficiently large choice of $M$ to deduce that the first term is at most 
\begin{align*}
    \mathbb P \Big(\Gamma, |\cA_{\infty}|> n^{\frac 12 - \frac{\delta}{2}}\Big) & \le \mathbb P\Big(|\cA_{\infty}| \ge |\cA_0| +  \gamma^r {\|\vdn\|}_\infty n^{\frac{\delta}{10}} \Big) \le o(n^{-10})\,.
\end{align*}
For the second term, notice that the mean of the binomial is $2n^{ - \delta} {\|\vdn\|}^2_\infty$. As long as $\kappa$ is sufficiently large so that ${\|\vdn\|}_\infty \le n^{\epsilon_*}$ for sufficiently small $\epsilon_* <\delta/4$, this is $o(n^{-\delta/2})$. Thus, by the Chernoff bound for the binomial~\eqref{eq:Chernoff-Poisson-binomial}, for every fixed $L\ge 1$, 
\begin{align}\label{eq:non-trivial-bc-tail-X-t}
 \mathbb P \Big((\mathcal G, X_{\cG,t}^1): |\mathfrak V_{(\cV_0,\cA_0)}(X_{\cG,t}^1)| > L\Big) \le o(n^{ - \frac{\delta L}{2} \wedge 10})\,.
\end{align}
Choosing $L$ sufficiently large (depending on $\delta$), we can make the right-hand side here $o(n^{-10})$ as well. 
\end{proof}

\begin{proof}[\textbf{\emph{Proof of Theorem~\ref{thm:k-R-sparse-whp}}}]
Given Lemma~\ref{lem:nontrivial-bdy-component-bound}, it is  straightforward to deduce Theorem~\ref{thm:k-R-sparse-whp}. Specifically, take $K$ sufficiently large so that the right-hand side of Lemma~\ref{lem:nontrivial-bdy-component-bound} is $o(n^{-10})$. By a union bound,  
\begin{align}\label{eq:k-r-sparse-union-bound}
    \mathbb P\big((\cG,X_{\cG,t}^{1}): X_{\cG,t}^1 \mbox{ is not }(K,R)\mbox{-}\sparse\big) \le \sum_{v}
    \mathbb P\big((\cG,X_{\cG,t}^{1}): |\mathfrak V_{(\cV_0,\cA_0)}(X_{\cG,t}^1)|>K\big) \le o(n^{ - 9})\,.
\end{align}
By Markov's inequality, 
   \begin{align*}\Pcm \Big(\cG: \mathbb P(X_{\cG,t}^1 & \mbox{ is not }(K,R)\mbox{-}\sparse)> n^{ - 6}\Big) \\
   & \le n^6 \Ecm [\mathbb P(X_{\cG,t}^1\mbox{ is not }(K,R)\mbox{-}\sparse)]\,,
   \end{align*}
and the conclusion follows from the fact that the expectation on the right-hand side is exactly the probability on the left-hand side of~\eqref{eq:k-r-sparse-union-bound}. 
\end{proof}

Let us conclude with a better bound in the special case of $R=0$ from Theorem~\ref{thm:k-R-sparse-whp}; this will be applied to establish our mixing time lower bounds for the Ising/Potts Glauber dynamics.

\begin{lemma}\label{lem:sparsity-for-Potts-lower-bound}
    Fix $q,\gamma$ and suppose $p<p_u(q,\gamma)$. There exists $\kappa$ such that for all $(\vdn)_n\in \cD_{\gamma,\kappa}$,  with probability $1-o(1)$ over $\cG\sim \Pcm$, for every $v\in \{1,...,n\}$ and every $\eta>0$, 
    \begin{align*}
        \pi_{\cG}\big(\omega(E_{v}^c) \mbox{ is not } \eta d_{v}\mbox{-}\sparse\big) \le C\exp( - \eta d_{v}/C)\,.
    \end{align*}
    (Here $\omega(E_{v}^c)$ is viewed as a boundary condition induced by $\omega$ on $E_{v} = \{e: e\ni v\}$.)
\end{lemma}

\begin{proof}
Fix a small $\epsilon>0$ and consider the following modification of the revealing process of Process~\ref{proc:revealing-FK-on-graph}. 
\begin{enumerate}
    \item Label the half-edges of the vertex $v$ $\hat e_v^{(1)},...,\hat e_v^{(d_v)}$
    \item Perform the process of Process~\ref{proc:revealing-FK-on-graph} with $\cV_0 = v$ $\cA_0 = \cA_0^{(1)} := \emptyset$, and $\widehat \cE_0 = \hat e_v^{(1)}$, stopped if either $|\cA_m^{(1)}|\ge n^{\epsilon}$ or in step 1.(a) a $\mathsf{bad}$ step is taken, i.e., some previously exposed vertex gets matched with. 
    \item For $i=1,...,d_v$, if $\hat e_{v}^{(i)}$ is hitherto un-matched, set $\cA_0^{(i)}$ to be the set of all matched edges to that point, and run the process of Process~\ref{proc:revealing-FK-on-graph} with $\cV_0 = v$, $\cA_0 = \cA_0^{(i)}$, and $\widehat \cE_0 = e_v^{(i)}$, stopped if $|\cA_m^{(i)}| \ge n^{\epsilon}$ or a $\mathsf{bad}$ step is taken. 
\end{enumerate}
Observe that in order for $\omega(E_v^c)$ to not be $\eta d_v$-$\sparse$, there must have been more than $\eta d_v/2$ many $i$'s for which the revealing process gets stopped (each such $i$ adds at most two vertices to the set $\mathfrak V_{(v,\emptyset)}(\omega)$ for $\omega \sim \pi_\cG$).  Throughout the entire procedure described above, at most $d_v n^{\epsilon}$ many edges are revealed, which for $\epsilon$ small and $\kappa$ large is at most $n^{1/4}$. By Lemma~\ref{lem:revealing-procedure-tail-bounds} with $M$ taken sufficiently large, for any fixed $i$, the probability of reaching $|\cA_m^{(i)}|\ge n^{\epsilon}$ is at most $o(n^{-10})$
uniformly over the history of the process up to that point. At the same time, for any fixed $i$, the probability of a $\mathsf{bad}$ step being taken for that revealing is at most 
\begin{align*}
 n^{\epsilon} \cdot \frac{n^{1/4}{{\|\vdn\|}}_\infty}{{\|\vdn\|}_1 -n^{1/4}{\|\vdn\|}_\infty} \le o(n^{-1/2})\,.
\end{align*} 
Putting the above together, the probability of more than $\eta d_v/2$ many of the $i$'s being stopped is at most 
\begin{align*}
 	\mathbb P\big(\bin(d_v, n^{-1/2}) \ge \eta d_v/2\big) \le C\exp({- \eta d_v/C})\,.
\end{align*} 
This in turn bounds the probability that $\omega(E_v^c)$ is $\eta d_v$-$\sparse$ as desired. 
\end{proof}

\section{Correlation decay and mixing time on treelike graphs}\label{sec:correlation-decay-treelike} 
Theorem~\ref{thm:k-R-sparse-whp} together with Lemma~\ref{lem:random-graph-treelike} reduce our analysis to treelike balls of radius $(\frac 12 - o(1)) \log_\gamma n$ with $K$-$\sparse$ boundary conditions. 
In this section, we establish sharp bounds on the rate of correlation decay on such treelike graphs (Theorem~\ref{theorem:influence-probability-new}) and bound the mixing time at these local scales (Lemma~\ref{lemma:local-mixing}). 

\subsection{Rate of correlation decay in treelike graphs} 
To prove Theorem~\ref{theorem:influence-probability-new} we will closely follow the approach from~\cite{BlGh21}, where an analogous result was proved for regular graphs (specifically see Proposition 3.3 in~\cite{BlGh21}). The key part of the extension is the use of the $(\gamma,\epsilon)$-volume growth condition to enable the application of Lemma~\ref{lemma:exp:decay:wired:treelike} to all sufficiently large subsets of the graph that are trees.

Let us fix an arbitrary vertex $v \in V$ and for ease of notation set $B := B_R(v)$ and
for each $1\le \ell \le R$, let $Q_\ell = \{u \in B: d(u,v) \ge \ell\}$.
For a boundary condition $\xi$ on $\partial B$, similarly to Definition~\ref{def:mathfrak-V} denote by $\mathfrak V_{B,\xi}$ the set of vertices in non-trivial components of $\xi$ (a component is non-trivial when it has at least two vertices). For any $u \in B$ such that $d(u,v) = \ell$,
let $u \stackrel{Q_\ell}\longleftrightarrow \mathfrak V_{B,\xi}$ denote the event that $u$
is connected to $\mathfrak V_{B,\xi}$ by a path of open edges fully contained in $Q_\ell$.
Define the event 
\[
\Upsilon_{B,\xi} : = \Big\{\omega \in \{0,1\}^{E(B)}: \big|\big\{u \in B:d(u,v)= \ell\,,\, u \stackrel{Q_\ell}\longleftrightarrow \mathfrak V_{B,\xi}\big\}\big| \ge 2 \mbox{ for all $1 \le \ell \le R$}\Big\}\,.
\]
It was proved in~\cite{BlGh21} that on general graphs, the event $\Upsilon_{B,\xi}$ controls the propagation of influence from $\partial B$ to the vertex $v$. 

Recall that $E_v$ denotes the set of edges incident to the vertex $v$. 

\begin{lemma}[Lemma 5.3 in~\cite{BlGh21}]\label{lem:influence-event}
	Fix a graph $G = (V,E)$, a vertex $v\in V$ and consider the ball $B_R(v)$; let $\xi \ge \tau$ denote two boundary conditions on $\partial B_R(v) = \{w\in B_R(v): d(v,w) = R\}$. Then, 
	\begin{align*}
	\|\pi_{B_R(v)}^\xi (\omega(E_v)\in \cdot)- \pi_{B_R(v)}^\tau(\omega(E_v)\in \cdot ) \|_\tv \le \pi_{B_R(v)}^\xi(\Upsilon_{B_R(v),\xi})\,.
	\end{align*}
\end{lemma}

With this lemma in hand, we are able to provide the proof of Theorem~\ref{theorem:influence-probability-new}. 

\begin{proof}[{\textbf{\emph{Proof of Theorem~\ref{theorem:influence-probability-new}}}}]
	By the triangle inequality and Lemma~\ref{lem:influence-event}, we have
	\begin{align*}
	\|\pi_{B_R(v)}^\xi (\omega(E_v)\in \cdot) - \pi_{B_R(v)}^\tau(\omega(E_v)\in \cdot ) \|_\tv & \le \|\pi_{B_R(v)}^\xi (\omega(E_v)\in \cdot)- \pi_{B_R(v)}^0(\omega(E_v)\in \cdot ) \|_\tv \\
	&  \,\,\,+ \|\pi_{B_R(v)}^\tau (\omega(E_v)\in \cdot)- \pi_{B_R(v)}^0(\omega(E_v)\in \cdot ) \|_\tv\\
	&\le \pi_{B_R(v)}^\xi(\Upsilon_{B_R(v),\xi}) + \pi_{B_R(v)}^\tau(\Upsilon_{B_R(v),\tau})\,.
	\end{align*}
	Hence, it suffices to bound
	$
	\pi_{B}^\xi(\Upsilon_{B,\xi})
	$	
	for an arbitrary vertex $v$ of $G$ and any $K$-$\sparse$ boundary condition $\xi$. Fix any such $v$ and let $B = B_R(v)$. 
	Let $H\subset E(B)$ be a set of at most $L$ edges such that the subgraph $(B, E(B)\setminus H)$ is a tree; the existence of such a set is guaranteed by the fact that $B_R(v)$ is $L$-$\treelike$.  
	Let $\mathcal Z= \{d_1,...,d_k\}$ be the subset of distances (from $v$) at which $H$ contains at least one vertex. Observe that each edge of $H$ intersects either one or two consecutive depths (distances from $v$) in $\mathcal Z$ and thus $|\mathcal Z|\le 2L$ since $B$ is $L$-$\treelike$. 
	Letting $d_0 =0$ and $d_{k+1}=R$,
	for $i=0,\dots,k$ we define:
	$$
	\mathcal F_i := \{u \in B:  d_i < d(u,v) < d_{i+1}\}\,.
	$$ 
	For each $0\le i \le k$, the graph 
	$\cF_i = (\mathcal F_i,E(\mathcal F_i))$ is a forest;
	observe that some $\mathcal F_i$'s might be empty.
	For each $i$, let $\mathcal T_{ij} = (\cT_{ij}, E(\cT_{ij}))$ for $j = 0,1,\dots$ denote the distinct connected components (subtrees) of $\mathcal F_i$ so that $\mathcal F_i = \bigcup_{j \ge 0} \mathcal T_{ij}$. 		

	For $\Upsilon_{B,\xi}$ to hold, 
	there must exist two sequences of simple paths $\Gamma = \gamma_{0},\ldots, \gamma_{k}$ and $\Gamma' = \gamma_{0}',\ldots, \gamma_{k}'$
such that $\gamma_i \subset E(\mathcal T_{ij})$ and $\gamma_i' \subset E(\mathcal T_{ij'})$ with $j \neq j'$ such that $\gamma_i$ (resp., $\gamma_i'$) connects the root of $\mathcal T_{ij}$ (resp., $\mathcal T_{ij'}$) 
to one of its leaves.

Observe that any simple path $\mathcal P$ between $v$ and $\mathfrak V_{B,\xi}$ is completely determined by an ordered sequence of vertices from $V(H)$ it uses and its endpoint in $\mathfrak V_{B,\xi}$.
Moreover, it is associated to a unique sequence $\Gamma$, and
each sequence $\Gamma$ can in turn correspond to at most $2^{|V(H)| k} \le 4^{L^2}$ simple paths because there are at most $2L$ vertices in $V(H)$.
Since $\xi$ is $K$-$\sparse$, there are at most $K$ choices for the endpoint of the path between $v$ and $\mathfrak V_{B,\xi}$. In total, we get that there are at most $4^{L^2} K (2L+1)!$ possible simple paths $\Gamma$ (this is a crude upper bound, but it suffices for our purposes). A union bound then implies
	\begin{align}
	\label{eq:event-bound}
	\pi_{B}^\xi (\Upsilon_{B,\xi}) \le  [4^{L^2}K(2L+1)!]^2 \cdot \sup_{\Gamma,\Gamma': V(\Gamma) \cap V(\Gamma') = \emptyset}\, \pi_{B}^\xi(\omega(\Gamma \cup \Gamma')=1)\,. 
	\end{align}
	
	Fix any two such paths $\Gamma, \Gamma'$, and consider the probability that $\omega(\Gamma \cup \Gamma') =1$. 
	The paths $\Gamma$ and $\Gamma'$ are vertex-disjoint by construction, but the events that $\Gamma$ and $\Gamma'$ are open (i.e., that all of their paths are open) in $\omega$ need not be independent. 
	To make them so, we wire all vertices at depths in the set $$\bigcup_{i=0}^{k+1} \{d_i-1,d_i,d_i+1\} \cap [0,R]\,.$$
	Let $\tilde \pi_{B}$ be the resulting random-cluster distribution.
	The monotonicity of the random-cluster measure implies that
	\begin{align}\label{eq:monotonicity-pi-tilde-pi}
	\pi_{B}^\xi(\omega(\Gamma \cup \Gamma')=1) \le \tilde \pi_{B}(\omega(\Gamma \cup \Gamma')=1)\,.
	\end{align}
	The distribution $\tilde \pi_{B}$ is a product measure over the $\mathcal T_{ij}$'s
	with boundary condition $(1,\circlearrowleft)$ in each  $\mathcal T_{ij}$. 
	Hence, since $\Gamma$ and $\Gamma'$ are such that for each $i \ge 0$,
	$\gamma_i$ and $\gamma_i'$ belong to distinct subtrees $\mathcal T_{\gamma_i}$, $\mathcal T_{\gamma_i'}$ of the forest $\mathcal F_i$, and	
	we have
	\begin{align*}
	\tilde \pi_{B}(\omega(\Gamma \cup \Gamma')=1) &= 
	\prod_{i=0}^{k} \pi_{\cT_{\gamma_i}}^{(1,\circlearrowleft)}(\gamma_i) 
	\prod_{i=0}^{k} \pi_{\cT_{\gamma_i'}}^{(1,\circlearrowleft)}(\gamma_i')\,.
	\end{align*}
	Let $h_i= d_{i+1} - d_i$ be the height of the trees in $\mathcal F_i$. Then,
	\begin{align*}
	\tilde \pi_{B}(\omega(\Gamma \cup \Gamma')=1) 
	&\le	\prod_{i: h_i > \sqrt{\varepsilon} R} 
	\pi_{\cT_{\gamma_i}}^{(1,\circlearrowleft)}(\gamma_i) 
	 \pi_{\cT_{\gamma_i'}}^{(1,\circlearrowleft)}(\gamma_i')
	\end{align*}
Since $G$ satisfies the $(\gamma,\varepsilon)$-volume-growth condition of Definition~\ref{def:volume-growth}, for each subtree
	$\cT_{\gamma_i}$ of height at least $\sqrt{\varepsilon} R$, for every vertex of $\cT_{\gamma_i}$ at distance at least ${\varepsilon} R$ from $\partial \cT_{\gamma_i}$, we have $|\partial \cT_{\gamma_i}| \le \gamma^{h_i}$. Hence,
	Lemma~\ref{lemma:exp:decay:wired:treelike} implies that there exists a constant  $A > 0$ such that, uniformly over $\Gamma, \Gamma'$, 
	\begin{align*}
	\tilde \pi_{B}(\omega(\Gamma \cup \Gamma')=1)
		 &\le A^{2L} \prod_{i: h_i > \sqrt{\varepsilon} R} \ps^{2(1-\sqrt{\varepsilon})h_i}  \\
		 &= A^{2L} \ps^{2(1-\sqrt{\varepsilon}) \sum_{i: h_i > \sqrt{\varepsilon} R} h_i} \\
		 &\le A^{2L} \ps^{2(1-\sqrt{\varepsilon}) (R-4L-2L\sqrt{\varepsilon} R)} = A' \ps^{2(1-(2L+1)\sqrt{\varepsilon})R} \,,
	\end{align*}
	for a suitable constant $A' = A'(A,L,K)$.
	Plugging this bound into~\eqref{eq:event-bound}--\eqref{eq:monotonicity-pi-tilde-pi}, we obtain
	$$
	\pi_{B}^\xi (\Upsilon_{B,\xi}) \le A' [K(2L+1)!]^2 \ps^{2(1-(2L+1)\sqrt{\varepsilon})R}\,,
	$$
	and the result follows taking $C = 2A'[4^{L^2}K(2L+1)!]^2$. 
\end{proof}

\subsection{Local mixing of the FK-dynamics}\label{sec:local-mixing-fk-dynamics}
In this section, we prove the mixing time bound of Lemma~\ref{lemma:local-mixing} for treelike graphs with sparse boundary conditions. We start by recalling some standard background concerning mixing times, log-Sobolev inequalities, and the effects of random-cluster boundary conditions on these quantities. 

\bigskip\noindent\textbf{Log-Sobolev inequalities.} \
For a Markov chain on a finite state space $\Omega$ with transition matrix $P$, reversible with respect to a distribution $\mu$, the Dirichlet form is defined for any function $f:\Omega \to \mathbb R$ by
\begin{align}\label{eq:Dirichlet-form}
    \mathcal E(f,f) := \frac{1}{2}\sum_{\omega,\omega'\in \Omega} \mu(\omega) P (\omega, \omega') (f(\omega) - f(\omega'))^2\,,
\end{align}
and its \emph{log-Sobolev constant} is given by 
\begin{align}\label{eq:lsi-constant}
    \alpha(P) := \min_{f: \mbox{Ent}_\mu[f^2]\ne 0} \frac{\mathcal E(f,f)}{\mbox{Ent}_\mu[f^2]}\,, \qquad \mbox{where} \qquad \mbox{Ent}_{\mu} [f^2] = \mathbb E_{\mu}\Big[f^2 \log \frac{f^2}{\mathbb E_{\mu}[f^2]}\Big]\,.
\end{align}
A \emph{log-Sobolev inequality} takes the form $\mathcal E(f,f) \ge \alpha  \mbox{Ent}_\mu[f^2]$ for all functions $f$. It is a standard fact that this inequality implies exponential convergence with rate $\alpha$ in total-variation distance to the stationary distribution (see,~\cite[Eq.~(3.3)]{Diaconis96}). 

\begin{fact}
    \label{fact:log-sobolev-tv-distance}
    Consider an ergodic Markov chain on a finite state space $\Omega$ with transition matrix $P$ reversible with respect to the distribution $\mu$. If the chain has a log-Sobolev constant $\alpha = \alpha(P)$, 
    \begin{align*}
        \max_{x_0\in \Omega} \|\mathbb P(X_t^{x_0}\in \cdot) - \mu\|_\tv\le \frac{1}{\sqrt{2}} e^{-\alpha t} \Big(\log \frac{1}{\min_{x\in \Omega} \mu(x)}\Big)^{1/2}\,,
    \end{align*}
    where $X_t^{x_0}$ is the chain after time $t$, started from initial state $x_0$.
\end{fact}

\bigskip\noindent\textbf{Boundary conditions and the FK-dynamics.} \ 
Two ``similar'' random-cluster boundary conditions (in terms of the wiring they induce)
have similar effects on the underlying random-cluster distribution and on the behavior of the corresponding FK-dynamics.
In turn, the Dirichlet form, and log-Sobolev constants of their corresponding dynamics should be ``close" to one another. 
We compile here a number of definitions and results that formalize this idea.

\begin{definition}[Definition 2.1  from~\cite{BGVfull}]
	 For two boundary conditions (partitions) $\phi \leq \phi'$, define $D(\phi,\phi') := c(\phi) - c(\phi')$ where $c(\phi)$ is the number of components in $\phi$. For two partitions $\phi, \phi'$ that are not comparable, let $\phi''$ be the smallest partition such that $\phi'' \geq \phi$ and $\phi'' \geq \phi'$ and set $D(\phi,\phi') = c(\phi) - c(\phi'')+ c(\phi')- c(\phi'')$. 
\end{definition}

The following lemma is then straightforward from the definition of the random-cluster measure~\eqref{eq:rcmeasure}. 

\begin{lemma}[Lemma 2.2 from~\cite{BGVfull}]
	\label{lemma:simple-rc-bound}
	Let $G=(V,E)$ be an arbitrary graph, $p \in (0,1)$ and $q > 0$. Let $\phi$ and $\phi'$ be any two partitions of $V$, i.e., boundary conditions on $G$. Then, for all random-cluster configurations $\omega\in \{0,1\}^E$, we have
	$$
	q^{-2D(\phi,\phi')} {\pi_{G}^{\phi'}(\omega)} \le \pi_{G}^{\phi}(\omega) \le q ^{2D(\phi,\phi')} \pi_{G}^{\phi'}(\omega)\,.
	$$  
\end{lemma}

The following corollary follows immediately from Lemma~\ref{lemma:simple-rc-bound}, the definition of the transition matrix of the FK-dynamics, and Theorem 4.1.1 in~\cite{SClecture-notes}.

\begin{cor}
    \label{cor:simple-ls-bound}
	Let $G=(V,E)$ be an arbitrary graph, $p \in (0,1)$ and $q > 0$.
	Consider the FK-dynamics on $G$ with boundary conditions $\phi$ and $\phi'$,
	and let $\alpha$, $\alpha'$ denote their log-Sobolev constants, respectively.
	Then, 
	$$
	q^{-5 D(\phi,\phi')} \alpha' \le \alpha \leq   q^{ 5 D(\phi,\phi')} \alpha'\,.
	$$
\end{cor}

We now use the above to bound the rate of convergence to equilibrium on $L$-treelike balls of radius $(\frac 12 - \delta) \log_\gamma n$. 

\begin{lemma}
	\label{lemma:main-tree-mixing}
	Suppose $G=(V,E)$ is $L$-$\treelike$. Let $\xi$ be a $K$-$\sparse$ boundary condition on $G$. For every $p\in (0,1)$ and $q>0$, there exists $\alpha_0(p,q,L,K)>0$ (importantly, independent of $G$) such that the log-Sobolev constant of the FK-dynamics on $G$ with boundary condition $\xi$ is at least $\alpha_0$.  
\end{lemma}

\begin{proof}
Observe first that the FK-dynamics on any tree with free boundary condition has log-Sobolev constant $c_{p,q} = \Omega(1)$. This follows from the observation that the random-cluster model on a tree with free boundary condition is simply the product measure, where every edge is open independently with probability $\ps$,
and the standard fact that the entropy tensorizes over product spaces; see, e.g.,~\cite{EntProd}. 

Now, let $H\subset E$ be a set of at most $L$ edges such that $(V, E\setminus H)$ is a tree. 
Consider the tree $\cT = (V,E \setminus H)$ 
and let $\phi$ be the boundary condition 
that includes all the connections from $\xi$ and adds wirings between 
$w$ and $w'$ for every edge $\{w,w'\} \in H$. 
By Corollary~\ref{cor:simple-ls-bound}, the log-Sobolev constant for the FK-dynamics on $\cT$ with boundary condition $\phi$ is at least ${c_{p,q} \cdot q^{-5(K+L)}}$.

The FK-dynamics on $G$ with boundary condition $\phi$ is a product Markov chain on $\{0,1\}^{E \setminus H} \times \{0,1\}^H$ with stationary distribution $\pi_{\cT}^\phi \otimes \prod_{i=1}^{|H|} \nu_i$, where the $\nu_i$'s are independent $\ber(p)$ distributions. 
Hence, it follows that the log-Sobolev constant of the FK-dynamics on $G$ with boundary condition $\phi$ is at least ${\hat c_{p,q} \cdot q^{-5(K+L)}}$
for a suitable constant $\hat c_{p,q} > 0$.
Finally, we note that by Corollary~\ref{cor:simple-ls-bound}, the log-Sobolev constant on $G$ with boundary conditions $\xi$ (instead of $\phi$) is at least  ${\hat c_{p,q} q^{-5(K+L)-5L}}$.
\end{proof}

Combining the above, we arrive at the following bound on the rate of convergence of the FK-dynamics on treelike graphs with sparse boundary conditions.

\begin{lemma}
    \label{lemma:local-mixing}
    Consider an $L$-$\treelike$ graph $G = (V,E)$ with a $K$-$\sparse$ boundary condition $\xi$. For every $p\in (0,1)$ and $q>0$, there exists $\alpha_0 = \alpha_0(p,q,L,K)>0$ such that
    \begin{align*}
        \max_{x_0\in \Omega} \|\mathbb P(X_t^{x_0}\in \cdot) - \pi_G^\xi\|_\tv\le \frac{1}{\sqrt{2}} e^{-\alpha_0 t} \Big(\log \frac{1}{\min_{x\in \Omega} \pi_G^\xi(x)}\Big)^{1/2}\,.
    \end{align*}
\end{lemma}

\begin{proof}[\textbf{\emph{Proof of Lemma~\ref{lemma:local-mixing}}}]
This follows by combining Lemma~\ref{lemma:main-tree-mixing} and Fact~\ref{fact:log-sobolev-tv-distance}.
\end{proof}

\section{Proof of main theorem}\label{sec:proof-of-main-theorem}
Given the estimates proven in the preceding sections, we can now prove our main result, Theorem~\ref{thm:intro:general}.

\subsection{Proof of main theorem: upper bound}
We begin with the proof of the upper bound.  

\begin{proof}[\textbf{\emph{Proof of Theorem~\ref{thm:intro:general}: upper bound}}]
	\label{subsec:mainthm-proof}
	Fix $q > 1$, $\gamma > 1$ and $p<p_u(q,\gamma)$. (It suffices to consider $\gamma>1$ since $\lim_{\gamma \downarrow 1} p_u(q,\gamma) = 1$, and if $\gamma\ge \gamma'$, then $\cD_{\gamma',\kappa} \subset \cD_{\gamma,\kappa}$.)
	Let $R = (\frac 12 - \delta)\log_\gamma n$,
	where $\delta>0$ is a small constant we choose later.
	For $K$ and $L$ fixed positive constants, $\varepsilon \in (0,1/2)$ and $t \ge 0$, let %
	$\Gamma_t = \Gamma_t(L,K,\delta,\varepsilon,\gamma)$ be the subset of (multi)graphs on $n$ vertices with degree sequence $\vdn$ given by:
	\begin{align*}                        
	\Gamma_{t} = \{\cG: \cG \mbox{ is }(L,R)\mbox{-}\treelike, & \mbox{ has }(\gamma,\varepsilon)\mbox{-volume growth} \\ 		& \mbox{and }\mathbb P(X_{\cG,t}^1 \mbox{ is }(K,R)\mbox{-}\sparse) \ge 1-n^{-5} \}\,.
	\end{align*}		
	By Lemmas~\ref{lem:random-graph-treelike} and~\ref{lem:random-graph-volume-growth}, as well as Theorem~\ref{thm:k-R-sparse-whp}, for every $\delta \in (0,1/2)$ and $\varepsilon \in (0,1/2)$,
	there exist constants $\kappa(p,q,\gamma,\delta)$, $L(\delta)$, $K(p,q,\gamma,\delta)$, and $T(p,q,\gamma)$ such that if $(\vdn)_n \in \cD_{\gamma,\kappa}$ 
	then $\Pcm (\Gamma_T^c ) = o(1)\,.$ 
	Hence, it suffices for us to prove that the mixing time of the FK-dynamics on any $\cG \in \Gamma_T$ is $O(\log n)$.
	
	Fix any $\cG \in \Gamma_T$. Let $((X_{t}^{x_0})_{t\ge 0})_{x_0}$ be the family of FK-dynamics initialized from all possible configurations $x_0$, coupled via the standard grand coupling for the FK-dynamics; i.e., using the same clock rings and the same uniform random variables to make the edge updates while running the chain from different initializations. Recall that this coupling is monotone when $q\ge 1$ so that for every $t\ge 0$, if $X_t^{x_0} \le X_t^{y_0}$, then $X_{t'}^{x_0}\le X_{t'}^{y_0}$ for all $t'\ge t$.  
	Using the standard fact that the coupling time provides a bound on the mixing time (see e.g.,~\cite{LP}),
	by a union bound over the edges, it suffices to show that under this grand coupling,
	\begin{align}\label{eq:wts-main-theorem}
	\mathbb P \big(X_{\hat T}^1 (e) \ne X_{\hat T}^0 (e)\big) \le o(1/|E(\cG)|)\qquad \mbox{for every $e\in E(\cG)$\,.}
	\end{align}

	Now fix any such $e = \{u,v\}$ and for ease of notation, set $B_v = E(B_R(v))$ and $B_v^c = E(\cG) \setminus B_v$.
	Consider two auxiliary copies of the FK-dynamics $Y_t^1$ and $Y_t^0$ that 
	censor (ignore) all updates on edges of $B_v^c$ after time $T$. The censoring inequality from~\cite{PWcensoring} applied to the FK-dynamics~\cite[Theorem 2.5]{GL2} implies that $Y_t^1 \succcurlyeq X_t^1$ and $Y_t^0 \preccurlyeq X_t^0$ for all $t\ge 0$ and thus 
	\[
	\mathbb P \big (X_{t}^1(e)\ne X_{t}^0(e) \big) 
	\le \mathbb P \big(X_{t}^1(e) = 1\big) - \mathbb P \big(X_{t}^0(e) = 1\big)
	\le \mathbb P \big(Y_{t}^1(e) = 1\big) - \mathbb P \big(Y_{t}^0(e) = 1\big)\,.
	\]
	
	Let $\cH_v$ be the set of configurations on $B_v^c$ such that the boundary conditions they induce on $B_v$ are $K$-$\sparse$. (Here and throughout the paper, the boundary condition induced by a configuration $\omega(B^c)$ on a set $B$ wires two vertices $w,w'\in V(B)$ if they are in the same connected component of $\omega(B^c)$.) By definition of $\Gamma_T$ and monotonicity of the FK-dynamics, we have for every $\cG\in \Gamma_T$, 
	\begin{align*}
	    \mathbb P(Y_T^0(B_v^c)\notin \cH_v) \le \mathbb P(Y_T^1(B_v^c)\notin \cH_v) \le n^{-5}\,.
	\end{align*}
	Therefore, $\mathbb P (Y_t^1 (e) = 1)-  \mathbb P ( Y_t^0 (e) = 1)$ is bounded by 
		\begin{align*}
		\max_{\phi^1,\phi^0\in \cH_v} \Big[\mathbb P ( Y_t^1 (e) =1 \mid Y_{T}^1(B_v^c)= \phi^1) - \mathbb P (Y_t^0(e) = 1 \mid Y_T^0 (B_v^c) = \phi^0) \Big] + 2n^{-5} 
		\end{align*}
	Now fix any $\phi^1,\phi^0\in \cH_v$. 
	From the triangle inequality, we have
	\begin{align}
	\mathbb P (Y_{T+s}^1(e) = 1 &  \mid Y_T^1 (B_v^c)  = \phi^1)  - \mathbb P (Y_{T+s}^0(e) =1 \mid Y_T^0(B_v^c) = \phi^0 ) \nonumber \\ 
	& \le \big|\mathbb P (Y_{T+s}^1(e) = 1 \mid Y_T^1 (B_v^c) = \phi^1) - \pi_{\cG}(\omega(e) =1 \mid \omega(B_v^c) =\phi^1)\big| \label{mixing:plus-bound} \\
	& \quad + \big|\pi_{\cG}(\omega(e) =1 \mid \omega(B_v^c) =\phi^1) - \pi_{\cG}(\omega(e) =1 \mid \omega(B_v^c) =\phi^0)\big| \label{mixing:smp-bound} \\ 
	& \quad + \big|\mathbb P (Y_{T+s}^0(e) = 1 \mid Y_T^0 (B_v^c) = \phi^0) - \pi_{\cG}(\omega(e) =1 \mid \omega(B_v^c) =\phi^0)\big|\,. \label{mixing:minus-bound}
	\end{align}
	Observe that the chain $(Y_{T+s}^1)_{s\ge 0}$ 
	may be viewed as an FK-dynamics on $B_v$ with the boundary condition induced by $\phi^1$, initialized from the (random) configuration $Y_{T}^1(B_v)$ 
	and with stationary distribution
	$\pi_{\cG}(\omega(B_v)\in \cdot \mid \omega(B_v^c) = \phi^1) = \pi_{B_v}^{\phi^1}\,;$
	the analogous statement is true for $(Y_{T+s}^0)_{s\ge 0}$ and $\pi_{B_v}^{\phi^0}$. 
	
	Setting $\hat T = T+\hat S_n$ where $\hat S_n =  \hat C \log n$ for a constant $\hat C(p,q,\gamma,L,K)$ sufficiently large, 
	since $B_v$ is $L$-$\treelike$ and $\phi^1$ is $K$-$\sparse$,
	we obtain from Lemma~\ref{lemma:local-mixing}
	that 
	$$
	\big|\mathbb P (Y_{\hat T}^1(e) = 1 \mid Y_T^1 (B_v^c) = \phi^1) - \pi_{\cG}(\omega(e) =1 \mid \omega(B_v^c) =\phi^1)\big| \le n^{-5};
	$$
	the same bound holds for~\eqref{mixing:minus-bound}.
	
	Finally, 
	since both $\phi^1$ and $\phi^0$ induce $K$-$\sparse$ boundary conditions on $B_v$ and $\cG$ is $(L,R)$-$\treelike$ with $(\gamma,\epsilon)$-volume growth,		
	by Theorem~\ref{theorem:influence-probability-new} there exists $C= C(p,q,L,K,\gamma) > 0$ such that~\eqref{mixing:smp-bound} is at most
	$$
	\|\pi_{B_v}^{\phi^1} (\omega(E_v)\in \cdot) - \pi_{B_v}^{\phi^0}(\omega(E_v)\in \cdot )\|_\tv \le C\ps^{2(1-C\sqrt{\epsilon})R}\le C\hat p^{(1-2\delta)(1-C\sqrt{\epsilon})\log_\gamma n}\,,
	$$ 
	where $E_v$ is the set of edges incident to $v$, and we used $R = (\frac{1}{2} - \delta)\log_\gamma n$.
	Setting $\theta = (1-2\delta)(1-C\sqrt{\epsilon})$,
		\begin{equation}\label{eq:tv:bound}
	\|\pi_{B_v}^{\phi^1} (\omega(E_v)\in \cdot) - \pi_{B_v}^{\phi^0}(\omega(E_v)\in \cdot )\|_\tv \le  C \hat p^{\theta \log_\gamma n} 
	= C n^{-\theta(1-\frac{1}{\log_{\hat p\gamma}\gamma})}\,. 
	\end{equation}
    Since $\hat p<1/\gamma$, $\log_{\hat p\gamma}\gamma<0$, there is some $c_{p,\gamma}>0$ such that the right-hand side is $C n^{-\theta(1+c_{p,\gamma})}$. By taking $\epsilon,\delta$ sufficiently small, $\theta$ can be made arbitrarily close to $1$, so that~\eqref{eq:tv:bound} is $o(1/n)$. 

	Now notice that $|E(\cG)| = O(n)$. To see this, observe that
	by Jensen's inequality
	$
	(\frac 1n \sum_v d_v)^2 \le \frac 1n \sum_v d_v^2,
	$
	and since $(\vdn)\in \cD_{\gamma,\kappa}$, we also have $\sum_v d_v^2 \le (1+\gamma) \sum_v d_v$.
	Combining these two inequalities we find that $|E(\cG)| \le \frac{(1+\gamma)n}{2}$. Therefore, each of~\eqref{mixing:plus-bound}--\eqref{mixing:minus-bound} are 
	$o(1/|E(\cG)|)$, implying~\eqref{eq:wts-main-theorem} as desired.  
\end{proof}

\subsection{Lower bound on the mixing time of FK-dynamics}\label{sec:lower-bound-rc}
We now turn to proving the mixing time lower bound of Theorem~\ref{thm:intro:general}. Though the argument is a straightforward adaptation of the proof of the lower bound in~\cite{BlGh21}, given our results on $(\gamma,\epsilon)$-growth of the random graph, and the exponential decay rate on random trees from Lemma~\ref{lemma:exp:decay:wired:treelike}, we include the proof for completeness, demonstrating that our new results give the requisite inputs to adapt the proof of~\cite{BlGh21}. 

\begin{claim}\label{clm:rg-many-trees}
Fix $\epsilon$ small. Suppose $\kappa$ is sufficiently large and $(\vdn)_n \in \cD_{\gamma,\kappa}$. With $\Pcm$-probability $1-o(1)$, $\cG$ satisfies $(\gamma,\epsilon)$-volume growth, and there exist $n^{1/5}$ vertices whose balls of radius $\frac 15 \log_\gamma n$ are disjoint, and are trees.
\end{claim}
\begin{proof}
On the one hand, by Lemma~\ref{lem:random-graph-volume-growth}, with probability $1-o(1)$, $\cG$ satisfies $(\gamma,\epsilon)$-volume growth, as long as $\kappa$ is sufficienlty large (depending on $\epsilon$). We prove the rest of the events have probability $1-o(1)$ by repeated application of the breadth-first revealing of Process~\ref{proc:BFS-revealing-process}. Namely, consider the procedure where we repeatedly take an arbitrary vertex $v$ that has not been discovered yet, and reveal its ball of radius $R = \frac{1}{5} \log_\gamma n$ via Process~\ref{proc:BFS-revealing-process}. Let $v_i$ be the $i$'th vertex to be selected in this procedure, and let $\cA_i$ be $\bigcup_{j\le i} E(B_R(v_j))$. Then, for integer $m \le n$ the probability that one of $(B_R(v_1),...,B_R(v_m))$ is not disjoint trees, is at most
\begin{align*}
   \Pcm \big(\bigcup_{i=1}^{m} \{B_R(v_i)\cap \cA_{i-1} = \emptyset \mbox{ or } B_R(v_i) \mbox{ is not a tree}\}\,,\,  \cG\in (\gamma,\epsilon)\mbox{-volume growth} \mid \cA_{i-1}\big)\,.
\end{align*}
Using the fact that $\cG$ is of $\gamma,\epsilon$-volume growth that we are intersecting with, the event can be rewritten as in its first $\gamma^R$ many matching attempts, none match with anything in $\cA_i$ or any half-edge belonging to a newly discovered half-edge of $B_R(v_i)$. In any one edge matching, uniformly over what has already been revealed, this probability is bounded by 
\begin{align*}
    \frac{{\|\vdn\|}_\infty m n^{1/5}}{{\|\vdn\|}_1 - {\|\vdn\|}_\infty m n^{1/5}}\,,
\end{align*}
which, for $m = n^{1/5}$, is at most $n^{-1/2}$ as long as $\kappa$ is sufficiently large, so that $\epsilon_*(\kappa)<1/10$. As there are at most $n^{2/5}$ edges to match, the probability that no edge gets matched to an already discovered vertex, and thus all the revealed balls form disjoint trees, is at most $\mathbb P(\mbox{Bin}(n^{2/5}, n^{-1/2}) >0)$
which is $o(1)$ simply by a Markov inequality. 
\end{proof}

Fix $\eta \in (0,1/5)$ to be taken sufficiently small later. For every $\cG$ having $n^{1/5}$ many vertices whose balls of radius $\frac 15 \log_\gamma n$ are disjoint trees, choose arbitrarily some $n^\eta$ vertices amongst the $n^{1/5}$ of Claim~\ref{clm:rg-many-trees}, and for each vertex collect a representative edge incident to it to form the set $\cC = \cC_\eta(\cG)$. 
Our proof will rely on a coupling of the restrictions of $X_{t,\cG}$ and $\pi_{\cG}$ to $\cC$ to $\ber (\hat p)$ product chains. For this, let: 
\begin{itemize}
\item $X_t= X_{t,\cG}$ be a realization of the FK-dynamics;
\item $Y_t= Y_{t,\cG}$ be a realization of the FK-dynamics that censors all updates in $E(\cG)\setminus \cC$;
\item $\nu$ as the product measure over $|\cC|$ many $\ber(\hat p)$ random variables. 
\end{itemize}
As before, let $Y_t^0$ be the chain $Y_t$ initialized from the all-$0$ configuration. 
  
\begin{lemma}\label{lem:couplings-to-product-chain}
Let $\cG$ be any graph satisfying $(\gamma,\epsilon)$-volume growth for $\epsilon<1/6$, and having at least $n^{1/5}$ vertices whose balls of radius $\frac 15 \log_\gamma n$ are disjoint trees.
For every $q > 1$, and $p<p_u(q,\gamma)$, there exists $\eta>0$ sufficiently small such that we have the following for $\cC = \cC_{\eta}(\cG)$:
\begin{enumerate}
\item For all $T = O(\log n)$, for all $t\le T$,
\begin{align*}
	\|P(X_{t}^0(\cC) \in \cdot ) - P(Y_t^0(\cC)\in \cdot)\|_\tv \le  o(1)\,.
\end{align*}
\item $\|\pi_{\cG}(\omega(\cC) \in \cdot ) - \nu\|_\tv \le  o(1)\,.$
\end{enumerate}
\end{lemma}
\begin{proof}
We start with part (1).
Our aim is to show that under the grand coupling of $X_t^0$ and $Y_t^0$, for every $t\le T = O(\log n)$, we have $\mathbb P (X_t^0 \ne Y_t^0) \le o(1)$. Under the grand coupling, let $\mathscr T_T = (t_1,t_2,...,t_{s(T)})$ denote the sequence of times on which the updated edge is in $\cC$, so that $s(T)$ counts the number of updates in $\cC$ by time $T$.  We can then bound
\begin{align*}
    \mathbb P (X_t^0 \ne Y_t^0) \le   \mathbb P (s(T) > n^{2\eta}) + \mathbb P (X_t^0 \ne Y_t^0, s(T)\le n^{2\eta})\,.
\end{align*}
The first term on the right-hand side is at most the probability that $\mbox{Pois}(T|\cC|)\ge n^{2\eta}$ which is $o(1)$ by standard tail estimates for Poisson variables. It thus suffices to work on the event $s(T)\le n^{2\eta}$. 

 Let $R : = \frac{1}{6} \log_\gamma n$ and let $Z_t$ be the FK-dynamics chain (coupled to $X_t, Y_t$ through the grand coupling) that freezes the configuration on $\cC  \cup (E(\cG)\setminus \bigcup_{e\in \cC} E(B_R(e)))$ to be all-$1$. Let $Z_t^0$ be the chain $Z_t$ initialized from the configuration that is all-$0$ on  $\bigcup_{e\in \cC} E(B_R(e))\setminus \{e\}$ (but all-$1$ on the frozen edges). 
Observe, trivially, that $X_t^0 \le Z_t^0$ for all $t \ge 0$. Also, observe that the updates of $Z_t^0$ are stochastically dominated by Glauber updates on the union of $2|\cC|$ many $d$-ary trees $(\cT_{e,1},\cT_{e,2})_{e\in \cC}$ of depth $R$,  rooted at the endpoints of the edges of $\cC$, and each having $(1,\circlearrowleft)$ boundary conditions. By monotonicity of the FK-dynamics, for every $t\ge 0$, 
\begin{align}\label{eq:Z-t-stochastic-domination}
    \mathbb P \bigg(Z_t^0 \Big(\bigcup_{e\in \cC} \big\{E(B_R(e))\setminus \{e\}\big\}\Big) \in \cdot\bigg) \preceq \bigotimes_{e\in \cC} \bigotimes_{i\in \{1,2\}} \pi_{\cT_{e,i}}^{(1,\circlearrowleft)}\,.
\end{align}

For each time $t_i \in \sT_T$, when an edge  $e_{t_i}\in \cC$ is updated, $Y_{t_i}^0(e_{t_i})$ is drawn from an independent $\ber(\hat p)$. At the same time, $X_{t_i}^0(e_{t_i})$ is drawn from $\ber(\hat p)$ if the endpoints of $e_{t_i}$ are not connected in $X_{t_i}^0$, which in turn must occur if none of $(\cT_{e,1},\cT_{e,2})_{e\in \cC}$ have an open root-to-leaf path in $Z_t^0$. 
We thus consider the probability of this event. 

Since $\cG$ has $(\gamma,\epsilon)$-volume growth for $\epsilon< 1/6$, every tree among $(\cT_{e,1},\cT_{e,2})_{e\in \cC}$ has at most $\gamma^R$ many leaves. Thus, by the stochastic domination of~\eqref{eq:Z-t-stochastic-domination},  and Lemma~\ref{lemma:exp:decay:wired:treelike}, the probability that the endpoints of $e_{t_i}$ are connected in $Z_{t_i}^0$ is at most $2  C (\hat p \gamma)^{R}$, which for $\eta$ sufficiently small is $O(n^{-3\eta})$. 
 On the event that $\{s(T) \le n^{2\eta}\}$, we can union bound the above probability over the $s(T)$ times in $\sT_T$, to find that $\mathbb P(X_t^0 \ne Y_t^0, s(T)\le n^{2\eta})$ is at most $O(n^{-\eta})= o(1)$ as desired.

For part (2), consider the $2|\cC|$ many $d$-ary trees $(\cT_{e,1}, \cT_{e,2})_{e\in \cC}$ emanating from the endpoints of the edges of $\cC$. 
Notice that if none of $(\cT_{e,1}, \cT_{e,2})_{e\in \cC}$ have an open root-to-leaf path, then the values $\omega(\cC)$ are conditionally distributed as a product of $\ber(\hat p)$ random variables, i.e., $\omega(\cC)$ would conditionally be distributed as $\nu(A)$. 

As such, the total-variation distance $\|\pi_{\cG}(\omega(\cC)\in \cdot) - \nu\|_\tv$ is bounded by the $\pi_{\cG}$-probability that one of $(\cT_{e,1}, \cT_{e,2})_{e\in \cC}$ has an open root-to-leaf path. By the stochastic domination 
\begin{align*}
    \pi_{\cG}\Big(\omega\Big(\bigcup_{e\in \cC} \cT_{e,1}\cup \cT_{e,2}\Big)\in \cdot\Big) \preceq \bigotimes_{e\in \cC} \bigotimes_{i\in \{1,2\}} \pi_{\cT_{e,i}}^{(1,\circlearrowleft)}\,.
\end{align*}
By a union bound, the left-hand side above is then at most 
\begin{align*}
    \sum_{e\in \cC} \sum_{i\in \{1,2\}} \pi_{\cT_{e,i}}^{(1,\circlearrowleft)}(e\leftrightarrow \partial \cT_{e,i})\,,
\end{align*}
which the $(\gamma,\epsilon)$-volume growth condition and Lemma~\ref{lemma:exp:decay:wired:treelike} together show is at most $2n^{\eta} \cdot C(\hat p \gamma)^{R}$. For $\epsilon$ sufficiently small (depending on $p,q,\gamma$) this is $o(1)$. 
\end{proof}

\begin{proof}[\textbf{\emph{Proof of Theorem~\ref{thm:intro:general}: lower bound.}}]
Take any $n$-vertex graph $\cG$ having $(\gamma,\epsilon)$-volume growth for $\epsilon<1/6$ and with $n^{1/5}$ many vertices whose balls of radius $\frac 15 \log_\gamma n$ are disjoint trees.
Note that by Claim~\ref{clm:rg-many-trees}, such graphs have $\Pcm$-probability $1-o(1)$. Take $\eta$ sufficiently small per Lemma~\ref{lem:couplings-to-product-chain}. Consider the event $A^+ \subset \{0,1\}^{\cC}$ that at least $\hat p n^{\eta} - n^{2\eta/3}$ of the edges in $\cC$ are open. 
 Let $(\overline Y_s)$ be the (discrete-time) product Markov chain over $|\cC| = n^{\eta}$ many i.i.d.\ $\ber(\hat p)$ random variables, coupled to $Y_t(\cC)$ via $\overline Y_{s(t)} = Y_t(\cC)$ for all $t$,
 where $s(t)$ counts the number of updates in $\cC$ by time $t$. 
 By item (1) of Lemma~\ref{lem:couplings-to-product-chain}, for every $T = O(\log n)$,
 \begin{align*}
	\mathbb P(X_T^0(\cC) \in A^+) & \le \mathbb P(s(T) > cn^{\eta}\log n) + \mathbb P\big(Y_T^0 \in A^+, s(T)\le cn^{\eta}\log n\big)  + o(1) \\
	& \le \mathbb P (s(T) > c n^{\eta} \log n) + \max_{s \le c n^{\eta} \log n} \mathbb P (\overline Y_s^0 \in A^+) + o(1)\,.
\end{align*}
(In the latter equation, we used the fact that the law of $\overline Y_s^{0}$ only depends on the sequence of times $(t_1,...,t_{s(T)})$ through the number of total updates $s(T)$.) 
Taking $T := c^2 \log n$ for $c>0$ sufficiently small, the probability that $s(T)$ is more than $c n^{\eta} \log n$ is $o(1)$ by tail bounds of a Poisson random variable with rate $T|\cC| = c^2 n^{\eta} \log n$. Turning to the middle term above, by the standard coupon collector bound, for every $c>0$ sufficiently small,  $\sup_{s\le  c  n^{\eta} \log n} \mathbb P(\overline Y_s^0 \in A^+) \le o(1)$. 

Combining the above, we obtain 
\[
\mathbb P(X_T^0(\cC) \in A^+) = o(1)\,.
\]
At the same time, by a Chernoff bound, $\nu(A^+) = 1-o(1)$ and by item (2) of Lemma~\ref{lem:couplings-to-product-chain}, then, $\pi_{\cG}(A^+) = 1-o(1)$. These two together imply that the (continuous-time) mixing time is at least $T = \Omega(\log n)$ as claimed. 
\end{proof}

\section{High-degree vertices slow down mixing for Potts Glauber dynamics}\label{sec:potts-slow-down}

Our lower bound on the mixing time of the Glauber dynamics for the Potts model in a random graph is derived from a bottleneck argument. For the special case of the Erd\H{o}s--R\'enyi random graph, 
the slow down can be attributed to \emph{isolated} stars whose central vertex has degree $\Theta(\frac{\log n}{\log \log n})$. Such a star appears in the random graph with high probability, and since it disconnected from the rest of $\cG$, the mixing time on the star serves as a lower bound for the mixing time on the full graph. This straightforwardly gives a lower bound of $n^{1+\Theta(\frac{1}{\log \log n})}$ on the discrete-time mixing time of the Glauber dynamics; see~\cite[Proposition 1.8]{MS}. 

For more general degree sequences, 
especially when there exist vertices of degree $\omega(\log n)$, the neighborhoods of the high-degree vertices will \emph{not} be isolated from the remainder of the graph, and in fact will correspond to the denser parts of the random graph. We use the exponential decay of random-cluster connectivities when $p<p_u(q,\gamma)$ to still leverage this star structure to give a lower bound on the mixing time of the Potts Glauber dynamics on a random graph that are exponential in its largest degree.  

We will work with the discrete-time Potts Glauber dynamics, which at each step selects a vertex $v\in V$ uniformly at random, and resamples its spin $\sigma_v$ according to the following conditional distribution: 
\begin{align*}
    \mu_{G,\beta,q}(\sigma_v = i\mid \sigma(V\setminus \{v\})) = \frac{e^{\beta \sum_{(v,w)\in E} \mathbf 1\{\sigma_w = i\}}}{\sum_{i=1}^q e^{\beta \sum_{(v,w)\in E}\mathbf 1\{\sigma_w = i\}}}\,,\qquad \mbox{for $i=1,\ldots,q$}\,.
\end{align*}

\begin{proof}[\textbf{\emph{Proof of Theorem~\ref{thm:Ising-Potts-lower-bound}}}]
Let $v_\star$ be a vertex in $\cG$ of maximum degree, and let $m_i(\sigma)$ denote the number of vertices adjacent to $v_\star$ that are assigned spin $i$ in configuration $\sigma$.
Define the following bottleneck set: 
\begin{align*}
    \mathcal A_{\epsilon}: = \Big\{\sigma: \sigma_{v_\star} = 1, m_1(\sigma) - \max_{j \neq 1} \, m_j(\sigma) \ge \lfloor\epsilon d_{v_\star}\rfloor\Big\}\,.
\end{align*}
Our aim is to show that $\mathcal A_{\epsilon}$ is a set of small conductance. Namely, we wish to show that there exists $\epsilon>0$ such that 
\begin{align*}
    \Phi(\cA_\epsilon) = \frac{Q(\cA_\epsilon, \cA_\epsilon^c)}{\mu(\cA_\epsilon)\mu(\cA_\epsilon^c)} \le e^{- \Omega(d_{v_\star})}\,,
\end{align*}
where $Q(\cA_\epsilon,\cA_\epsilon^c) = \sum_{\sigma\in \cA_\epsilon, \sigma'\in \cA_\epsilon^c} \mu(\sigma) P(\sigma,\sigma')$ with $P$ denoting the transition matrix of the discrete-time Glauber dynamics.

For this, notice that we can expand $Q(\cA_\epsilon, \cA_\epsilon^c)$ into its contribution from transitions that exit $\cA_\epsilon$ by flipping the spin of $\sigma_{v_\star}$, and those that exit $\cA_\epsilon$ by flipping the spin of a neighbor of $v_\star$ in the configuration. Hence, let
\begin{align*}
    \widehat\cA_{\epsilon} : = \Big\{\sigma\in \cA_\epsilon: m_1(\sigma) - \max_{j \neq 1} \, m_j(\sigma) = \lfloor\epsilon d_{v_\star}\rfloor\Big\}\,.
\end{align*}
Namely, we can bound 
\begin{align}
    \Phi(\cA_\epsilon) 
    &\le  \frac{\sum_{\sigma\in \cA_\epsilon} \sum_{j =2}^q \mu(\sigma) P(\sigma,\sigma^{v_\star \to j})}{\mu(\cA_\epsilon)\mu(\cA_\epsilon^c)} + \frac{\sum_{\sigma\in \widehat\cA_\epsilon, \sigma'\in \cA_\epsilon^c} \mu(\sigma) P(\sigma,\sigma')}{\mu(\cA_\epsilon)\mu(\cA_\epsilon^c)}\notag\\
    &\le \frac{\max_{\sigma\in \cA_\epsilon, j\ne 1} P(\sigma, \sigma^{v_\star \to j})}{\mu(\cA_\epsilon^c)} + \frac{\mu(\widehat \cA_\epsilon)\max_{\sigma \in \widehat \cA_\epsilon} P(\sigma,\cA_\epsilon^c) }{\mu(\cA_\epsilon)\mu(\cA_\epsilon^c)} \label{eq:conductance}\,,
\end{align}
where $\sigma^{v_\star \to j}$ is the configuration which agrees with $\sigma$ everywhere except on $v_\star$ where it takes spin $j$. Observe first of all, that by the spin symmetry of the model Potts model, $\mu(\cA_\epsilon) \le 1/q$ and thus $\mu(\cA_\epsilon^c) \ge \frac {q-1}{q} \ge \frac 12$. Moreover, by the definition of the Glauber dynamics, the transition matrix $P$ satisfies 
\begin{align*}
    \max_{\sigma \in \cA_\epsilon} P(\sigma, \sigma^{v_\star \to j}) = \frac{1}{n} \cdot \frac{e^{\beta m_j(\sigma)}}{\sum_{j} e^{\beta m_j(\sigma)}} \le  \frac{e^{\beta (m_j(\sigma) - m_1(\sigma))} }{n}\le \frac{e^{ - \beta \epsilon d_{v_\star}}}{n}\,.
\end{align*}
Also, for every $\sigma \in \widehat \cA_\epsilon$, it satisfies 
\begin{align*}
    \max_{\sigma\in \widehat\cA_\epsilon}P(\sigma, \cA_\epsilon^c) \le \frac{d_{v_\star}}{n}\,,
\end{align*}
as one needs to select a neighbor of $v_\star$ to update in order to move from $\sigma\in \widehat \cA_\epsilon$ to $\cA_\epsilon^c$. 
As such, 
\begin{align}\label{eq:conductance-simplified}
    \Phi(\cA_\epsilon) \le \frac{2}{n} e^{- \beta \epsilon d_{v_\star}} + \frac{2d_{v_\star}}{n} \frac{\mu(\widehat \cA_\epsilon)}{\mu(\cA_\epsilon)}\,.
\end{align}
It remains to bound the ratio of the probabilities of the events $\widehat\cA_\epsilon$ to $\cA_\epsilon$. It will be convenient to work with the random-cluster representation of the Potts model. Let 
\begin{align*}
    \cA_{\epsilon}^{\textsc{rc}}: = \big\{\omega \in \{0,1\}^{E(\cG)}: |\{e \in E_{v_\star}: \omega(e) =1\}| \ge \epsilon d_{v_\star}~\textrm{and}~|\mathfrak V_{E_{v_\star}} (\omega)|\le \epsilon d_{v_\star}/2\big\}\,,
\end{align*}
where we recall that 
$E_{v_\star}$ is the set of edges incident to $v_\star$ and
$\mathfrak V_{E_{v_\star}}
(\omega)$ is the set of neighbors of $v_\star$
in non-trivial connected components in the configuration induced by $\omega(E(\cG)\setminus E_{v_\star})$.
In words this is the event that an $\epsilon$ fraction of the edges incident to $v_\star$ are open, and at most $\epsilon d_{v_\star}/2$ of the neighbors of $v_\star$ are connected to one another in the configuration outside the immediate neighborhood of $v_\star$. 

We first note that for some $\epsilon(p,q,\gamma)>0$, with high probability under the random graph, the event $\cA_{\epsilon}^{\textsc{rc}}$ has high probability under the random-cluster measure $\pi$. For this, observe that since $\pi$ stochastically dominates the independent edge percolation measure  with edge probability $\ps$, and by a Chernoff bound, for any $\cG \sim \Pcm$
\begin{align*}
    \pi_{\cG} (|\{e\in E_{v_\star}: \omega(e) = 1\}| <\epsilon d_{v_\star}) \le \mathbb P(\bin(d_{v_\star}, \ps) < \epsilon d_{v_\star}) \le e^{ - \Omega(\hat p d_{v_\star})}\,,
\end{align*}
for $\epsilon$ sufficiently small (say, less than $\ps/2$). By Lemma~\ref{lem:sparsity-for-Potts-lower-bound}, if $\kappa$ is sufficiently large and $(\vdn)_n\in \cD_{\gamma,\kappa}$, for every $\epsilon>0$, we have  with probability $1-o(1)$ over the graph $\cG \sim \Pcm$, 
\begin{align*}
    \pi_{\cG} (|\mathfrak V_{E_{v_\star}}(\omega)| >\epsilon d_{v_\star}/2) \le e^{ - \Omega( \epsilon d_{v_\star})}\,.
\end{align*}
Hence, it follows from a union bound that there exists $\epsilon(p,q,\gamma)$ small, such that with probability $1-o(1)$, $\cG\sim \Pcm$ is such that
\begin{align*}
    \pi_\cG(\cA_{\epsilon}^{\textsc{rc}}) \ge 1-e^{ - \Omega(\epsilon d_{v_\star})}\,.
\end{align*}
As such, as long as $\epsilon>0$ is sufficiently small, we can bound
\begin{align*}
    \frac{\mu(\widehat \cA_\epsilon)}{\mu(\cA_\epsilon)} \le \frac{\mathbb P_{(\mu,\pi)}(\widehat \cA_\epsilon \mid \cA_{4\epsilon}^{\textsc{rc}}) + e^{ - \Omega(\epsilon d_{v_\star})}}{\mathbb P_{(\mu,\pi)}(\cA_\epsilon \mid \cA_{4\epsilon}^{\textsc{rc}}) (1-e^{ - \Omega(\epsilon d_{v_\star})})}\,,
\end{align*}
where $\mathbb P_{(\mu,\pi)}$ denotes the joint Edwards--Sokal distribution over spin-edge configurations; see~\cite{ES,Grimmett}.

Now, consider a random-cluster configuration in $\cA_{4\epsilon}^{\textsc{rc}}$. Fixing a random-cluster configuration $\omega$ in  $\cA_{4\epsilon}^{\textsc{rc}}$, we claim that the probability of $\cA_\epsilon$ given $\omega$ is at least the probability of the following event $\Gamma_\epsilon$, that 
\begin{enumerate}
    \item the component $\cC_{v_\star}(\omega)$ is given state $1$; and 
    \item  amongst the vertices of 
$$\mathfrak V_\star^c := V(E_{v_\star}) \setminus (\cC_{v_\star}(\omega) \cup \mathfrak V_{E_{v_\star}}(\omega))\,,$$
the number of vertices in each state in $[q]$ is within $\epsilon d_{v_\star}/2$ of $|\mathfrak V_\star^c|/q$. 
\end{enumerate}
To see this, note that on $\Gamma_\epsilon$, since $\cC_{v_\star}(\omega)$ has size at least $4\epsilon d_{v_\star}$ and $|\mathfrak V_{E_{v_\star}}(\omega)| \le 2\epsilon d_{v_\star}$, no matter which state the vertices of $\mathfrak V_{E_{v_\star}}(\omega)$ take, $\sigma$ will be such that
$$m_1(\sigma) - \max_{j\ne 1} m_j(\sigma) > (4\epsilon - 2 \epsilon - \epsilon)d_{v_\star} = \epsilon d_{v_\star}\,.$$
(Here, the $4\epsilon$ comes from the sites in $\cC_{v_\star}(\omega)$, the $-2\epsilon$ comes from a worst-possible assignment of states to sites of $\mathfrak V_{E_{v_\star}}(\omega)$, and the $-\epsilon$ comes from the maximal bias on the sites in $\mathfrak V_\star^c$.)

The probability of the event $\Gamma_\epsilon$, when coloring the components of $\omega$ independently, uniformly at random, is at least $1/q$ (for the probability of coloring $\cC_{v_\star}(\omega)$ in state $1$) times 
\begin{align*}
    1 - q\mathbb P\Big(\Big|\bin(|\mathfrak V_\star^c|, 1/q) - |\mathfrak V_\star^c|/q \Big| > \epsilon d_{v_\star}/2\Big) \ge 1- e^{- \Omega(\epsilon d_{v_\star})}\,.
\end{align*}
(Here, we used a union bound over the $q$ different states, and a Chernoff bound.)
In particular, we find that for $\epsilon(p,q,\gamma)>0$ sufficiently small, 
\begin{align*}
    \mathbb P_{(\mu,\pi)}(\cA_\epsilon \mid \cA_{4\epsilon}^{\textsc{rc}}) \ge \min_{\omega \in \cA_{4\epsilon}^{\textsc{rc}}} \mathbb P(\Gamma_\epsilon \mid \omega) \ge  \frac{1}{q}\Big(1- e^{- \Omega(\epsilon d_{v_\star})}\Big)\,.
\end{align*}
On the other hand, the probability of $\widehat \cA_\epsilon$, conditionally on $\cA_{4\epsilon}^{\textsc{rc}}$ is bounded by the probability of the colorings of $\mathfrak V_\star^c$ assigning at least $2\epsilon d_{v_\star} + |\mathfrak V_\star^c|/q$ many of its vertices to some state $j\ne 1$. By a union bound over the $q$ states, and a Chernoff bound, this has probability at most
\begin{align*}
   q\mathbb P\Big(\Big|\bin(|\mathfrak V_\star^c|, 1/q) - |\mathfrak V_\star^c|/q \Big| > 2 \epsilon d_{v_\star}\Big) \le e^{ - \Omega(\epsilon d_{v_\star})}\,.
\end{align*}
At this point, we can plug the above bounds into~\eqref{eq:conductance-simplified} to deduce that for all $\epsilon(p,q,\gamma)>0$ sufficiently small,  
$$\Phi(\cA_\epsilon) \le \frac{1}{n}e^{-\Omega(\beta \epsilon d_{v_\star})}\,.$$
(Notice that $\epsilon$ sufficiently small, needed to scale as $\Theta(1/p)$, so that this is $n^{-1} e^{ - \Omega(\beta^2 d_{v_\star})}$ for small $\beta$.)
Relying on the classical Cheeger bound (see e.g.,~\cite[Theorem 7.4]{LP}), the inverse of $\Phi(\cA_{\epsilon})$ serves as a lower bound on the mixing time of the Glauber dynamics for the Potts model.
\end{proof}

\bibliographystyle{abbrv}
\bibliography{references}

\end{document}